\documentclass[11pt,a4paper, reqno]{amsart}
\usepackage[utf8]{inputenc}
\usepackage{amsfonts}
\usepackage{amssymb}
\usepackage{enumerate}
\usepackage{graphicx}
\usepackage{mathtools}
\usepackage{amsmath, amsthm}
\usepackage{tikz}
\usepackage{color}
\usepackage[T1]{fontenc}
\usepackage[sort,numbers]{natbib}
\usepackage{bbm}
\usepackage[left = 1in, right=1in, top = 1in, bottom=1in]{geometry}
\usepackage{amsmath}
\usepackage{amsfonts}
\usepackage{amssymb}
\usepackage{bbm}
\usepackage{mathrsfs}
\usepackage[utf8]{inputenc}
\usepackage[english]{babel}
\usepackage{hyperref}
\usepackage{relsize}
\usepackage{accents}
\usepackage{appendix}
\usepackage{lipsum}
\usepackage{todonotes}

\newcommand\blfootnote[1]{%
  \begingroup
  \renewcommand\thefootnote{}\footnote{#1}%
  \addtocounter{footnote}{-1}%
  \endgroup
}

\usepackage{environ}
\NewEnviron{eq}{%
\begin{equation}\begin{split}
  \BODY
\end{split}\end{equation}
}

\makeatletter
  \@addtoreset{chapter}{part}
  \@addtoreset{@ppsaveapp}{part}
\makeatother

\long\def\/*#1*/{}

\usepackage{palatino}

\makeatletter
\newsavebox\myboxA
\newsavebox\myboxB
\newlength\mylenA

\newcommand*\xoverline[2][0.75]{%
    \sbox{\myboxA}{$\m@th#2$}%
    \setbox\myboxB\null
    \ht\myboxB=\ht\myboxA%
    \dp\myboxB=\dp\myboxA%
    \wd\myboxB=#1\wd\myboxA
    \sbox\myboxB{$\m@th\overline{\copy\myboxB}$}
    \setlength\mylenA{\the\wd\myboxA}
    \addtolength\mylenA{-\the\wd\myboxB}%
    \ifdim\wd\myboxB<\wd\myboxA%
       \rlap{\hskip 0.5\mylenA\usebox\myboxB}{\usebox\myboxA}%
    \else
        \hskip -0.5\mylenA\rlap{\usebox\myboxA}{\hskip 0.5\mylenA\usebox\myboxB}%
    \fi}
\makeatother

\usepackage{hyperref}
\hypersetup{
  colorlinks,
  linkcolor=blue,
  citecolor=black,
  filecolor=black,
  urlcolor=black,
  pdftitle={},
  pdfauthor={S. Dhara, R. van der Hofstad, J. van Leeuwaarden, S. Sen},
  pdfcreator={S. Dhara},
  pdfsubject={},
  pdfkeywords={}
}
\numberwithin{equation}{section}

\newcommand{\cG}{\mathcal{G}}

\def\sss{\scriptscriptstyle}
\newcommand{\ubar}[1]{\underaccent{\bar}{#1}}

\newcommand*{\Mtilde}[1]{\skew{5}{\tilde}{#1}}
\newcommand{\prob}[1]{\ensuremath{\mathbbm{P}\left(#1\right)}}
\newcommand{\expt}[1]{\ensuremath{\mathbbm{E}\left[#1\right]}}
\newcommand{\var}[1]{\ensuremath{\mathrm{Var}\left(#1\right)}}
\newcommand{\refl}[1]{\ensuremath{\mathrm{refl}(#1)}}

\newcommand{\floor}[1]{\ensuremath{\left\lfloor #1 \right\rfloor}}
\newcommand{\ind}[1]{\ensuremath{\mathbbm{1}_{\left\{#1\right\}}}}
\newcommand{\pto}{\ensuremath{\xrightarrow{\mathbbm{P}}}}
\newcommand{\dto}{\ensuremath{\xrightarrow{d}}}
\newcommand{\asto}{\ensuremath{\xrightarrow{\sss \mathrm{a.s.}}}}
\newcommand{\surp}[1]{\ensuremath{\mathrm{SP}(#1)}}

\newcommand{\ord}{\ensuremath{\mathrm{ord}}}
\newcommand{\PR}{\ensuremath{\mathbbm{P}}}
\newcommand{\E}{\ensuremath{\mathbbm{E}}}
\newcommand{\R}{\ensuremath{\mathbbm{R}}}
\newcommand{\N}{\ensuremath{\mathbbm{N}}}
\newcommand{\e}{\ensuremath{\mathrm{e}}}
\newcommand{\1}{\ensuremath{\mathbbm{1}}}
\newcommand{\OP}{\ensuremath{O_{\sss\PR}}}
\newcommand{\oP}{\ensuremath{o_{\sss\PR}}}
\newcommand{\thetaP}{\ensuremath{\Theta_{\sss\PR}}}
\newcommand{\CM}{\ensuremath{\mathrm{CM}_n(\boldsymbol{d})}}
\newcommand{\shortarrow}{{\sss\downarrow}}

\newcommand{\dif}{\mathrm{d}}
\newcommand{\bld}[1]{\boldsymbol{#1}}

\newcommand{\bZ}{\mathbf{Z}}
\newcommand{\im}{\mathrm{i}}

\newcommand{\biS}{\bar{\mathbf{S}}_{\infty}^\lambda}
\newcommand{\iS}{\bar{S}_{\infty}^\lambda}
\newcommand{\cA}{\mathcal{A}}
\newcommand{\sF}{\mathscr{F}}
\newcommand{\cT}{\mathcal{T}}

\newtheorem{theorem}{Theorem}
\newtheorem{algo}{Algorithm}
\newtheorem{lemma}[theorem]{Lemma}
\newtheorem{proposition}[theorem]{Proposition}
\newtheorem{corollary}[theorem]{Corollary}
\newtheorem{assumption}{Assumption}
\newtheorem{remark}{Remark}
\newtheorem{fact}{Fact}
\newtheorem{defn}{Definition}

\let\plainqed\qedsymbol
\newcommand{\claimqed}{$\lrcorner$}
\newenvironment{claimproof}{\begin{proof}\renewcommand{\qedsymbol}{\claimqed}}{\end{proof}\renewcommand{\qedsymbol}{\plainqed}}

\definecolor{aqua}{rgb}{0.0, 1.0, 1.0}

\begin{document}

\title[Heavy-tailed configuration models at criticality]{Heavy-tailed configuration models at criticality}
\author[Dhara]{Souvik Dhara$^{1,2}$}
\author[van der Hofstad]{Remco van der Hofstad$^3$}
\author[van Leeuwaarden]{Johan S.H. van Leeuwaarden$^{3,4}$}
\author[Sen]{Sanchayan Sen$^5$}
\blfootnote{$^1$Department of Mathematics, Massachusetts Institute of Technology}
\blfootnote{$^2$Microsoft Research -- New England}
\blfootnote{$^3$Department of Mathematics and Computer Science, Eindhoven University of Technology}
\blfootnote{$^4$Department of Econometrics and Operations Research, Tilburg University}
\blfootnote{$^5$Department of Mathematics, Indian Institute of Science}
\date{\today}
\maketitle 
\blfootnote{\emph{Correspondence to:} \href{mailto:dharasouvik1991@gmail.com}{S. Dhara}.}
\blfootnote{\emph{Emails:}  \href{mailto:sdhara@mit.edu}{sdhara@mit.edu}, \href{mailto:r.w.v.d.hofstad@tue.nl}{r.w.v.d.hofstad@tue.nl}, \href{mailto:j.s.h.vanleeuwaarden@uvt.nl}{j.s.h.vanleeuwaarden@uvt.nl}, \href{mailto:sanchayan.sen1@gmail.com}{sanchayan.sen1@gmail.com}}
\blfootnote{\emph{Keywords and phrase}. Critical configuration model, heavy-tailed degree, thinned L\'evy process, augmented multiplicative coalescent, universality, critical percolation}
\begin{abstract} 
 We study the critical behavior of the component sizes for the configuration model when the tail of the degree distribution of a randomly chosen vertex is a regularly-varying function with exponent $\tau-1$, where $\tau\in (3,4)$.
 The component sizes are shown to be of the order $n^{(\tau-2)/(\tau-1)}L(n)^{-1}$ for some slowly-varying function $L(\cdot)$.
 We show that the re-scaled  ordered component sizes converge in distribution to the ordered excursions of a thinned L\'evy process.
 This proves that the scaling limits for the component sizes for these heavy-tailed configuration models are in a different universality class compared to the Erd\H{o}s-R\'enyi random graphs. 
 Also the joint  re-scaled vector of ordered component sizes and their surplus edges is shown to have a distributional limit under a strong topology. 
 Our proof resolves a conjecture by Joseph, \emph{Ann. Appl. Probab.} (2014) about the scaling limits of uniform simple graphs with i.i.d.~degrees in the critical window, and sheds light on the relation between the scaling limits obtained by Joseph and in this paper, which appear to be quite different.
  Further, we use percolation to study the evolution of the component sizes and the  surplus edges within the critical scaling window, which is shown to converge in finite dimension to the augmented multiplicative coalescent process introduced by Bhamidi et. al., \emph{Probab. Theory Related Fields} (2014).
  The main results of this paper  are proved under rather general assumptions on the vertex degrees. 
  We also discuss how these assumptions are satisfied by some of the frameworks that have been studied previously.
\end{abstract}
\section{Introduction}\label{sec:intro}
Most random graph models posses a phase-transition property: there is a model-dependent parameter $\theta$ and a critical value $\theta_c$ such that whenever $\theta >\theta_c$, the largest component of the graph contains a positive proportion of vertices with high probability (w.h.p) and when $\theta\leq\theta_c$, the largest component is of smaller order than the size of the graph w.h.p. 
The random graph is called \emph{critical} when $\theta=\theta_c$. The study of critical random graphs started in the 1990s with the works of \citet{B84}, \citet{L90}, \citet{JKLP93} and \citet{A97} for Erd\H{o}s-R\'enyi random graphs. 
A large body of subsequent work in \cite{NP10b,R12,Jo10,BHL10,BBW12,DHLS15,HJL10} showed that the behavior of a wide array of random graphs at criticality is universal in the sense that  certain graph properties do not depend on the precise description of the model.
One of these universal features is that the scaling limit of the large component sizes, for many graph models, is identical to that of the Erd\H{o}s-R\'enyi random graphs.
All these universality results are obtained under the common assumption that the degree distribution is light-tailed, i.e., the asymptotic degree distribution has sufficiently large moments. 
For critical configuration models, a finite third-moment condition proves to be crucial~\cite{DHLS15}.
However, empirical studies of real-world networks from various fields including physics, biology, computer science \cite{BA99,AB02,KKRRT99,SFFF03,FFF99,SHL16} show that the  degree distribution is often heavy-tailed and of power-law  type.
 A first work towards understanding the critical behavior in the \emph{heavy-tailed} network models is~\cite{BHL12}, which showed that, for rank-one inhomogeneous random graphs, when the weight distribution follows a power-law with exponent $\tau\in (3,4)$, the component sizes and the scaling limits turn out to be quite different from that of the Erd\H{o}s-R\'enyi random graph. This revealed an entirely new universality class for the phase transition of heavy-tailed random graphs.
 In this paper, we study the configuration model with heavy-tailed power-law degrees. The configuration model is the canonical model for generating a random \emph{multi}-graph with a prescribed degree sequence. This model was introduced by \citet{B80} to generate a uniform simple $d$-regular graph on $n$ vertices, when $dn$ is even. The idea was later generalized to general degree sequences by \citet{MR95} and others. 
We will denote the multi-graph generated by the configuration model on the vertex set $[n] = \{1,\dots,n\}$ with the degree sequence $\boldsymbol{d}$  by $\mathrm{CM}_n(\boldsymbol{d})$. The configuration model, conditioned on simplicity, yields a uniform simple graph with the same degree sequence, which explains its popularity. \vspace{.1cm} \\
\noindent {\bf Our main contribution.}
Let $D_n$ be the degree of a uniformly chosen vertex, independently of the random graph $\mathrm{CM}_n(\boldsymbol{d})$. 
The main goal of this paper is to obtain various asymptotic results for the component sizes of  $\mathrm{CM}_n(\boldsymbol{d})$ when $\prob{D_n\geq k}\sim L_0(k)/k^{\tau-1}$ for some $\tau\in (3,4)$ and $L_0(\cdot)$ a slowly-varying function (here $\sim$ denotes an unspecified approximation that will be defined in more detail below). 
In fact, under a general set of assumptions (see Assumptions~\ref{assumption1}~and~\ref{assumption2}), we prove the following:
\begin{enumerate}[(1)]
\item The largest connected components are of the order $n^{(\tau-2)/(\tau-1)}L(n)^{-1}$ and the width of the scaling window is of the order $n^{(\tau-3)/(\tau-1)}L(n)^{-2}$ for some slowly-varying function $L(\cdot)$. 
\item The joint distribution of the re-scaled component sizes and the surplus edges converges in distribution to a suitable limiting random vector in a strong topology.
It turns out that the scaling limits for the re-scaled ordered component sizes can be described in terms of the ordered excursions of a certain thinned L\'evy process that only depends on the asymptotics of the \emph{high}-degree vertices, which is also the case in \cite{BHL12}. 
Further, the scaling limits for the surplus edges can be described by Poisson random variables with the parameters being the areas under the excursions of the thinned L\'evy process. 
\item The results hold conditioned on the graph being simple, thus solving \cite[Conjecture 8.5]{Jo10} in this case.
\item The scaling limits also hold for the graphs obtained by performing critical percolation on a supercritical graph. 
The percolation clusters can be coupled in a natural way using the Harris coupling. This enables us to study the \emph{evolution} of the component sizes and the surplus edges as a dynamic process in the critical window.
The evolution of the component sizes and surplus edges is shown to converge to a version of the \emph{augmented multiplicative coalescent} process that was first introduced in \cite{BBW12}. 
In fact, our results imply that there exists a version of the augmented multiplicative coalescent process whose one-dimensional distribution can be described by the excursions of a thinned L\'evy process and a Poisson process with the intensity being proportional to the thinned L\'evy process, which is also novel. 
\end{enumerate}
Thus, this paper provides a detailed understanding about the critical component sizes and surplus edges for the heavy-tailed graphs in the critical window. 
Before stating our main results precisely, we introduce some notation and concepts.
\subsection{The model}  Consider $n$ vertices labeled by $[n]:=\{1,2,...,n\}$ and a non-increasing sequence of degrees $\boldsymbol{d} = ( d_i )_{i \in [n]}$ such that $\ell_n = \sum_{i \in [n]}d_i$ is even. For notational convenience, we suppress the dependence of the degree sequence on $n$. The configuration model on $n$ vertices having degree sequence $\boldsymbol{d}$ is constructed as follows:
 \begin{itemize}
 \item[] Equip vertex $j$ with $d_{j}$ stubs, or \emph{half-edges}. Two half-edges create an edge once they are paired. Thus, an edge is a pair of half-edges. Initially we have $\ell_n=\sum_{i \in [n]}d_i$ unpaired half-edges. Pick any one half-edge and pair it with a uniformly chosen half-edge from the remaining unpaired half-edges and keep repeating the above procedure until all the unpaired half-edges are exhausted. 
 \end{itemize}
  Note that the graph constructed by the above procedure may contain self-loops or multiple edges. It can be shown that conditionally on $\mathrm{CM}_{n}(\boldsymbol{d})$ being simple, the law of such graphs is uniform over all possible simple graphs with degree sequence $\boldsymbol{d}$ \cite[Proposition 7.13]{RGCN1}. 
 It was further shown in~\cite{J09c} that, under very general assumptions, the asymptotic probability of the graph being simple is positive. 

\subsection{Definition and notation}\label{sec:notation}
We use the standard notation of $\xrightarrow{\sss\PR}$, and $\dto$ to denote convergence in probability and in distribution, respectively.
The topology needed for the  convergence in distribution will always be specified unless it is  clear from the context. 
We often use the Bachmann–Landau notation $O(\cdot)$, $o(\cdot)$ for large $n$ asymptotics of real numbers.
The notation $A_n \sim B_n$ will be used to say that $A_n/B_n\to 1$. We say that a sequence of events $(\mathcal{E}_n)_{n\geq 1}$ occurs with high probability~(w.h.p) with respect to the probability measures $(\mathbbm{P}_n)_{n\geq 1}$  when $\mathbbm{P}_n\big( \mathcal{E}_n \big) \to 1$. Define $f_n = O_{\sss\mathbbm{P}}(g_n)$ when  $ ( |f_n|/|g_n| )_{n \geq 1} $ is tight; $f_n =o_{\sss\mathbbm{P}}(g_n)$ when $f_n/g_n  \xrightarrow{\sss\PR} 0 $  whp; $f_n =\thetaP(g_n)$ if both $f_n=\OP(g_n) $ and $g_n=\OP(f_n)$. For a random variable $X$ and a distribution $F$, we write $X\sim F$ to denote that $X$ has distribution $F$. Denote
\begin{equation}
\ell^p_{\shortarrow}:= \big\{ \mathbf{x}= (x_1, x_2, x_3, ...): x_1 \geq x_2 \geq x_3 \geq ... \text{ and } \sum_{i=1}^{\infty} x_{i}^p < \infty \big\}
\end{equation} with the $p$-norm metric $d(\mathbf{x}, \mathbf{y})= \big( \sum_{i=1}^{\infty} |x_i-y_i|^p \big)^{1/p}$. Let $\ell^2_{\shortarrow} \times \mathbbm{N}^{\infty}$  denote the product topology of $\ell^2_{\shortarrow}$ and $\mathbbm{N}^{\infty}$ with $\mathbbm{N}^{\infty}$ denoting the sequences on $\mathbbm{N}$ endowed with the product topology. Define also
\begin{equation}\mathbb{U}_{\shortarrow}:= \big\{ ((x_i,y_i))_{i=1}^{\infty}\in  \ell^2_{\shortarrow} \times \mathbbm{N}^{\infty}: \sum_{i=1}^{\infty} x_iy_i < \infty \text{ and } y_i=0 \text{ whenever } x_i=0 \; \forall i   \big\}
\end{equation} with the metric \begin{equation} \label{defn_U_metric}\mathrm{d}_{\mathbb{U}}((\mathbf{x}_1, \mathbf{y}_1), (\mathbf{x}_2, \mathbf{y}_2)):= \bigg( \sum_{i=1}^{\infty} (x_{1i}-x_{2i})^2 \bigg)^{1/2}+ \sum_{i=1}^{\infty} \big| x_{1i} y_{1i} - x_{2i}y_{2i}\big|. 
\end{equation} Further, let $\mathbb{U}^0_{\shortarrow} \subset \mathbb{U}_{\shortarrow}$ be given by \begin{equation}\mathbb{U}^0_{\shortarrow}:= \big\{((x_i,y_i))_{i=1}^{\infty}\in\mathbb{U}_{\shortarrow} : \text{ if } x_k = x_m, k \leq m,\text{ then }y_k \geq y_m\big\}.
\end{equation}  Let $(\mathbb{U}^0_{\shortarrow})^k$ denote the $k$-fold product space of $\mathbb{U}^0_{\shortarrow}$. For any $\mathbf{z}\in \mathbb{U}_{\shortarrow}$, $\ord(\mathbf{z})$ will denote the element of $\mathbb{U}^0_{\shortarrow}$ obtained by suitably ordering the coordinates of $\mathbf{z}$. 

 We often use the boldface notation $\mathbf{X}$ for the process $( X(s) )_{s \geq 0}$, unless stated otherwise. $\mathbb{D}[I,E]$ will denote the space of c\`adl\`ag functions from an interval $I$ to the metric space $E=(E,\mathrm{d})$ equipped with the Skorohod $J_1$-topology. 
Consider a decreasing sequence $ \boldsymbol{\theta}=(\theta_1,\theta_2,\dots)\in \ell^3_{\shortarrow}\setminus \ell^2_{\shortarrow}$. Denote   $\mathcal{I}_i(s):=\ind{\xi_i\leq s }$ where $\xi_i\sim \mathrm{Exp}(\theta_i/\mu)$ independently, and $\mathrm{Exp}(r)$ denotes the exponential distribution with rate $r$.  Consider the process \begin{equation}\label{defn::limiting::process}
\bar{S}^\lambda_\infty(t) =  \sum_{i=1}^{\infty} \theta_i\left(\mathcal{I}_i(t)- (\theta_i/\mu)t\right)+\lambda t,
\end{equation}for some $\lambda\in\mathbbm{R}, \mu >0$, and define the reflected version of $\bar{S}_\infty^{\lambda}(t)$ by
\begin{equation} \label{defn::reflected-Levy}
 \refl{ \bar{S}_\infty^{\lambda}(t)}= \bar{S}_\infty^{\lambda}(t) - \min_{0 \leq u \leq t} \bar{S}_\infty^{\lambda}(u).
\end{equation}The process of the form \eqref{defn::limiting::process} was termed \emph{thinned} L\'evy processes in \cite{BHL12} (see also \cite{AHKL16,HKL14}), since the summands are thinned versions of  Poisson processes.
For any function $f\in \mathbb{D}[0,\infty)$,  define $\ubar{f}(x)=\inf_{y\leq x}f(y)$.  $\mathbb{D}_+[0,\infty)$ is the subset of $\mathbb{D}[0,\infty)$ consisting of functions with positive jumps only. Note that $\ubar{f}$ is continuous when $f\in \mathbb{D}_+[0,\infty)$. An \emph{excursion} of a function $f\in \mathbb{D}_+[0,T]$ is an interval $(l,r)$ such that 
\begin{equation}\label{def:excursion}\min\{f(l-),f(l)\}=\ubar{f}(l)=\ubar{f}(r)=\min\{f(r-),f(r)\} \quad \text{and}\quad f(x)>\ubar{f}(r),\ \forall x\in (l,r)\subset [0,T].
\end{equation}Excursions of a function $f\in \mathbb{D}_+[0,\infty)$ are defined similarly.
 We will use $\gamma$ to denote an excursion, as well as the length of the excursion $\gamma$, to simplify notation. 

 Also, define the counting process $\mathbf{N}$ to be the Poisson process that has intensity $\refl{ \bar{S}_\infty^{\lambda}(t)}$ at time $t$ conditionally on $( \bar{S}_\infty^{\lambda}(u) )_{u \leq t}$. Formally, $\mathbf{N}$ is characterized as the counting process for which 
\begin{equation} \label{defn::counting-process}
N(t) - \int\limits_{0}^{t}\refl{ \bar{S}_\infty^{\lambda}(u)}\dif u
\end{equation} is a martingale.  We use  the notation $N(\gamma)$ to denote the number of marks in the interval $\gamma$.

Finally, we define a Markov process $(\mathbf{Z}(s))_{s\in\R}$ on $\mathbbm{D}(\R,\mathbb{U}^0_{\shortarrow})$, called the augmented multiplicative coalescent (AMC) process. Think of  a collection of particles in a system with $\mathbf{X}(s)$ describing their masses and $\mathbf{Y}(s)$ describing an additional attribute at time $s$. Let $K_1,K_2>0$ be constants. The evolution of the system takes place according to the following rule at time $s$:
\begin{itemize}
\item[$\rhd$] For $i\neq j$, at rate $K_1X_i(s)X_j(s)$,  the $i^{th}$ and $j^{th}$ components merge and create a new component of mass $X_i(s)+X_j(s)$ and attribute $Y_i(s)+Y_j(s)$.  
\item[$\rhd$] For any $i\geq 1$, at rate $K_2X_i^2(s)$, $Y_i(s)$ increases to $Y_i(s)+1$. 
\end{itemize}Of course, at each event time, the indices are re-organized to give a proper element of $\mathbb{U}^0_{\shortarrow}$.
This process was first introduced in \cite{BBW12} to study the joint behavior of the component sizes and the surplus edges over the critical window.
 In \cite{BBW12}, the authors extensively study the properties of the  standard version of AMC, i.e., the case  $K_1=1,K_2=1/2$ and showed in \citep[Theorem 3.1]{BBW12} that this is a (nearly) Feller process, a property that will play a crucial rule in the final part of this paper.

\begin{remark}\normalfont Notice that the summation term in \eqref{defn::limiting::process}, after replacing $\theta_i$ by $\mu\theta_i$, is of the form 
\begin{equation}V^{\boldsymbol{\theta}}(s)= \mu^{\alpha}\sum_{i=1}^{\infty} \big(\theta_i\ind{\xi_i\leq s}-\theta_i^2s \big),
\end{equation} where $\xi_i\sim \mathrm{Exp}(\theta_i)$ independently over $i$ and  $\boldsymbol{\theta}\in \ell^3_{\shortarrow}\setminus\ell^2_{\shortarrow}$. Therefore, by \cite[ Lemma 1]{AL98}, the process $\refl{\bar{\mathbf{S}}_\infty^{\lambda}}$ has no infinite excursions almost surely and only finitely many excursions of length at least $ \delta$, for any $\delta >0$.
\end{remark}

\subsection{Main results for critical configuration models}\label{sec:results}
 Throughout this paper we will use the shorthand notation
\begin{subequations}
\begin{equation}\label{eqn:notation-const}
 \alpha= 1/(\tau-1),\qquad \rho=(\tau-2)/(\tau-1),\qquad \eta=(\tau-3)/(\tau-1),
\end{equation}
\begin{equation}\label{eqn:notation-power*}
a_n= n^{\alpha}L(n),\qquad b_n=n^{\rho}(L(n))^{-1},\qquad c_n=n^{\eta} (L(n))^{-2},
\end{equation}
\end{subequations}where $\tau\in (3,4)$ and $L(\cdot)$ is a slowly-varying function.
We state our results under the following assumptions:

\begin{assumption}\label{assumption1}
\normalfont Fix $\tau \in (3,4)$. Let $\boldsymbol{d}=(d_1,\dots,d_n)$ be a degree sequence such that the following conditions hold:
\begin{enumerate}[(i)] 
\item \label{assumption1-1} (\emph{High-degree vertices}) For any fixed $i\geq 1$, 
\begin{equation}\label{defn::degree}
 \frac{d_i}{a_n}\to \theta_i,
\end{equation}where $\boldsymbol{\theta}=(\theta_1,\theta_2,\dots)\in \ell^3_{\shortarrow}\setminus \ell^2_{\shortarrow}$. 
\item \label{assumption1-2} (\emph{Moment assumptions}) Let $D_n$ denote the degree of a vertex chosen uniformly at random from $[n]$, independently of $\mathrm{CM}_n(\boldsymbol{d})$. Then, $D_n\xrightarrow{\sss d} D$, for some integer-valued random variable $D$ and 
\begin{equation}
 \frac{1}{n}\sum_{i\in [n]}d_i\to \mu:=\expt{D}, \qquad \frac{1}{n}\sum_{i\in [n]}d_i^2\to \E[D^2],\qquad  \lim_{K\to\infty}\limsup_{n\to\infty}a_n^{-3} \sum_{i=K+1}^{n} d_i^3=0.
\end{equation}
\item \label{assumption1-3} (\emph{Critical window}) For some $\lambda \in \mathbbm{R}$,
\begin{equation}\label{critical-window}
 \nu_n(\lambda):=\frac{\sum_{i\in [n]}d_i(d_i-1)}{\sum_{i\in [n]}d_i}=1+\lambda c_n^{-1}+o(c_n^{-1}).
\end{equation}
\item \label{assumption1-4} Let $n_1$ be the number of vertices of degree-one. Then $n_1=\Theta(n)$, which is equivalent to assuming that $\prob{D=1}>0$.
\end{enumerate}
\end{assumption} 
Note that Assumption~\ref{assumption1}~\eqref{assumption1-1}-\eqref{assumption1-2} implies $\liminf_{n\to\infty}\E[D_n^3]=\infty$. 
The following three results hold for any $\mathrm{CM}_n(\boldsymbol{d})$ satisfying Assumption~\ref{assumption1}:

\begin{theorem}\label{thm::conv:component:size} Consider $\mathrm{CM}_n(\boldsymbol{d})$ with the degrees satisfying {\rm Assumption~\ref{assumption1}}. Denote the $i^{th}$-largest cluster of $\mathrm{CM}_n(\boldsymbol{d})$ by $\mathscr{C}_{\sss (i)}$. Then, 
 \begin{equation}
  \left( b_n^{-1}|\mathscr{C}_{\sss (i)}|\right)_{i\geq 1}\dto\left(\gamma_i(\lambda)\right)_{i\geq 1},
 \end{equation}
 with respect to the $\ell^2_{\shortarrow}$-topology where $\gamma_i(\lambda)$ is the length of the $i^{th}$ largest excursion of the process $\bar{\mathbf{S}}_\infty^{\lambda} $, while $b_n$ and the constants $\lambda, \mu$ are defined in \eqref{eqn:notation-power*} and {\rm Assumption~\ref{assumption1}}.
\end{theorem}  

\begin{theorem}\label{thm:spls}
 Consider $\mathrm{CM}_n(\boldsymbol{d})$ with the degrees satisfying {\rm Assumption~\ref{assumption1}}. Let $\mathrm{SP}(\mathscr{C}_{\sss (i)})$ denote the number of surplus edges in  $\mathscr{C}_{\sss (i)}$  and $\mathbf{Z}_n:= \ord( b_n^{-1}|\mathscr{C}_{\sss (i)}|,\mathrm{SP}(\mathscr{C}_{\sss (i)}))_{i\geq 1}$, $\mathbf{Z}:=\ord(\gamma_i(\lambda), N(\gamma_i(\lambda)))_{i\geq 1}$. Then, as $n\to\infty$,
 \begin{equation}\label{thm:eqn:spls}
  \mathbf{Z}_n\dto\mathbf{Z}
 \end{equation}with respect to the $\mathbb{U}^0_{\shortarrow}$ topology, where $\mathbf{N}$ is  defined in \eqref{defn::counting-process}.
\end{theorem}
\begin{theorem}\label{thm:simple-graph}
 The results in {\rm Theorem~\ref{thm::conv:component:size}} and {\rm Theorem~\ref{thm:spls}} also hold for $\mathrm{CM}_n(\boldsymbol{d})$ conditioned on simplicity.
\end{theorem}

\begin{remark}\normalfont
The only previous work to understand  the critical behavior of the configuration model with heavy-tailed degrees was by \citet{Jo10} where the degrees were assumed to be an i.i.d.~sample from an exact power-law distribution and the results were obtained for the component sizes of $\mathrm{CM}_n(\boldsymbol{d})$ (as in Theorem~\ref{thm::conv:component:size}). 
We will see that Assumption~\ref{assumption1} is satisfied for i.i.d.~degrees in Section~\ref{sec:iid-degrees}. Thus, a quenched version (conditional on the degrees) of \cite[Theorem~8.3]{Jo10} follows from our results. Further, if the degrees are chosen approximately as the weights chosen in \cite{BHL12}, then our results continue to hold. This sheds light on the relation between the scaling limits in \cite{BHL12}~and~\cite{Jo10} (see Remark~\ref{rem:jos-vs-this}).
Moreover, Theorem~\ref{thm:simple-graph} resolves \cite[Conjecture 8.5]{Jo10}.
\end{remark}

\begin{remark}\label{rem:gen-functional1}\normalfont The conclusions of Theorems~\ref{thm::conv:component:size},~\ref{thm:spls},~and~\ref{thm:simple-graph} hold for more general functionals of the components. Suppose that each vertex $i$ has a weight $w_i$ associated to it and let $\mathscr{W}_i$ denote the total weight of the component $\mathscr{C}_{\sss (i)}$, i.e., $\mathscr{W}_i = \sum_{k\in \mathscr{C}_{\sss (i)}} w_k$. Then, under some regularity conditions on the weight sequence $\bld{w}=(w_i)_{i\in [n]}$, in Section~\ref{sec:comp-functional} we will  show  that the scaling limit for $\mathbf{Z}^w_n:= \ord( b_n^{-1}\mathscr{W}_i,\mathrm{SP}(\mathscr{C}_{\sss (i)}))_{i\geq 1}$ is given by  $\mathbf{Z} = \ord(\kappa\gamma_i(\lambda), N(\gamma_i(\lambda)))_{i\geq 1}$, where the constant $\kappa$ is given by  $\kappa = \lim_{n\to\infty}\sum_{i\in [n]}d_iw_i/\sum_{i\in[n]}d_i$. Observe that, for $w_i= \ind{d_i=k}$, $\mathscr{W}_i$ gives the asymptotic number of vertices of degree $k$ in the $i^{th}$ largest component.
\end{remark}

\begin{remark}\normalfont It might not be immediate why we should work with Assumption~\ref{assumption1}. We will see in Section~\ref{sec:example} that Assumption~\ref{assumption1} is satisfied by the degree sequences in some important and natural cases. The reason to write the assumptions in this form is to make the properties of the degree distribution explicit (e.g. in terms of moment conditions and the asymptotics of the highest degrees) that jointly lead to this universal critical limiting behavior. We explain the significance of Assumption~\ref{assumption1} in more detail in Section~\ref{sec:discussion}.
\end{remark}

\subsection{Percolation on heavy-tailed configuration models} 
Percolation refers to deleting each edge of a graph independently with probability $1-p$.  Consider  percolation on a configuration model $\mathrm{CM}_n(\boldsymbol{d})$ under the following assumptions:
\begin{assumption} \label{assumption2} \normalfont 
 \begin{enumerate}[(i)]
  \item \label{assumption2-1} Assumption~\ref{assumption1}~\eqref{assumption1-1},~and~\eqref{assumption1-2} hold for the degree sequence and $\mathrm{CM}_n(\boldsymbol{d})$ is super-critical, i.e.,
  \begin{equation}
   \nu_n=\frac{\sum_{i\in [n]}d_i(d_i-1)}{\sum_{i\in [n]}d_i}\to \nu =\frac{\expt{D(D-1)}}{\expt{D}}>1.
  \end{equation}
  \item \label{assumption2-2} (\emph{Critical window for percolation}) The percolation parameter $p_n$ satisfies
  \begin{equation}
  p_{n}=p_n(\lambda):=\frac{1}{\nu_n} \big( 1+ \lambda c_n^{-1}+o(c_n^{-1}) \big) \quad \text{for some } \lambda \in \R.
  \end{equation}
  \end{enumerate}
 \end{assumption}
Let $\mathrm{CM}_n(\boldsymbol{d},p_n(\lambda))$ denote the graph obtained through percolation on $\mathrm{CM}_n(\boldsymbol{d})$ with bond retention probability $p_n(\lambda)$. The following result gives the asymptotics for the ordered component sizes and the surplus edges for $\mathrm{CM}_n(\boldsymbol{d},p_n(\lambda))$:
\begin{theorem}\label{thm:percolation}
Consider $\mathrm{CM}_n(\boldsymbol{d},p_n(\lambda))$ satisfying {\rm Assumption~\ref{assumption2}}.
Let $\tilde{\mathbf{S}}_{\infty}^{\lambda}$ denote the process in \eqref{defn::limiting::process} with $\theta_i$ replaced by $\theta_i/\sqrt{\nu}$, let
 $\mathscr{C}_{\sss(i)}^p$ denote the $i^{th}$ largest component of $\mathrm{CM}_n(\boldsymbol{d},p_n)$ and let $\mathbf{Z}_n^p(\lambda):=\ord( b_n^{-1}|\mathscr{C}_{\sss(i)}^p|,\mathrm{SP}(\mathscr{C}_{\sss(i)}^p))_{i\geq 1}$,  $\mathbf{Z}^p(\lambda):=\ord((\nu^{1/2}\tilde{\gamma}_i(\lambda),N(\tilde{\gamma}_i(\lambda)))_{i\geq 1}$, where $\tilde{\gamma}_i(\lambda)$ is the largest excursion of $\tilde{\mathbf{S}}_{\infty}^{\lambda}$. Then, for any $\lambda\in \mathbbm{R}$, as $n\to \infty$,
\begin{equation}\label{eqn:perc:limit}
 \mathbf{Z}_n^p(\lambda)\dto \mathbf{Z}^p(\lambda)
\end{equation}with respect to the $\mathbb{U}^0_{\shortarrow}$ topology.
\end{theorem}
 Now, consider a graph $\mathrm{CM}_n(\boldsymbol{d})$ satisfying Assumption~\ref{assumption2}~\eqref{assumption2-1}. To any edge $(ij)$ between vertices $i$ and $j$ (if any), associate an independent uniform$[0,1]$ random variable $U_{(ij)}$. Note that the graph obtained by keeping only those edges satisfying $U_{(ij)}\leq p_n(\lambda)$ is distributed as $\mathrm{CM}_n(\boldsymbol{d},p_n(\lambda))$. This construction naturally couples the graphs $(\mathrm{CM}_n(\boldsymbol{d},p_n(\lambda)))_{\lambda\in\R}$ using the same set of uniform random variables. This is known as the {\em Harris coupling}. Our next result shows that the evolution of the component sizes and the surplus edges of  $\mathrm{CM}_n(\boldsymbol{d},p_n(\lambda))$, as $\lambda$ varies, can be described by a version of the augmented multiplicative coalescent process described in Section~\ref{sec:notation}:
\begin{theorem}\label{thm:mul:conv} Suppose that {\rm Assumption~\ref{assumption2}} holds, and $\ell_n/n = \mu +o(n^{-\zeta})$ for some $\eta<\zeta<1/2$.
Fix any $k\geq 1$, $-\infty<\lambda_1<\dots<\lambda_k<\infty$. Then, there exists a version $\mathbf{AMC}=(\mathrm{AMC}(\lambda))_{\lambda\in\R}$ of the augmented multiplicative coalescent such that, as $n\to\infty$,
\begin{equation}
 \left(\mathbf{Z}_n^p(\lambda_1), \dots\mathbf{Z}_n^p(\lambda_k)\right) \dto \left(\mathbf{AMC}(\lambda_1),\dots,\mathbf{AMC}(\lambda_k)\right)
\end{equation}with respect to the $(\mathbb{U}^0_{\shortarrow})^k$ topology, where at each $\lambda$, $\mathrm{AMC}(\lambda)$ is distributed as the limiting object in~\eqref{eqn:perc:limit}.
\end{theorem}  
\begin{remark}
\normalfont
 Theorem~\ref{thm:mul:conv} also holds when $\E[D_n^3]\to \E[D^3]<\infty$ with $\alpha=\eta=1/3$, $\rho =2/3$ and $L(n)=1$. This improves \cite[Theorem 3.7]{DHLS15}, which was proved only for the cluster sizes.
\end{remark}
\begin{remark}\normalfont Theorem~\ref{thm:mul:conv}, in fact, shows that there exists a version of the AMC process whose distribution at each fixed $\lambda$ can be described by the excursions of a thinned L\'evy process and an associated Poisson process. 
This did not appear in \cite{BBW12,BM14}, since the scaling limits in those settings were described in terms of the excursions of a Brownian motion with negative parabolic drift.   
\end{remark}
\begin{remark}\label{rem:additional-assumption}\normalfont The additional assumption in Theorem~\ref{thm:mul:conv} about the asymtotics  $\ell_n/n$ is required only in one place for Proposition~\ref{prop:coupling-whp} and the rest of the proof works under Assumption~\ref{assumption2} only. That is why we have separated this assumption from the set of conditions in Assumption~\ref{assumption2}. 
It is worthwhile mentioning that the condition is minor, e.g., we will see that this condition is satisfied under the two widely studied set ups in Section~\ref{sec:example}. 
\end{remark}

\begin{remark}\normalfont 
As we will see in Section~\ref{sec:conv-amc}, the proof of Theorem~\ref{thm:mul:conv} can be extended to more general functionals of the components. For example, the evolution of the number of degree $k$ vertices along with the surplus edges can also be described by an AMC process. The key idea here is that these component functionals become approximately proportional to the component sizes in the critical window  and thus the scaling limit for the component functionals becomes a constant multiple of the scaling limit for the component sizes.
\end{remark}

\section{Important examples}
\subsection{Power-law degrees with small perturbation}\label{sec:example}
As discussed in the introduction, our main goal is to obtain results for the critical configuration model with $\prob{D_n\geq k}\sim L_0(k)k^{-(\tau-1)}$ for some $\tau\in (3,4)$. 
Here, we consider such an example and show that the conditions of Assumption~\ref{assumption1} are satisfied. Thus, the results in Section~\ref{sec:results} hold for \CM\   in the following set-up that is closely related to the model studied in~\cite{BHL12} for rank-1 inhomogeneous random graphs.
\par Fix $\tau\in (3,4)$. Suppose that $F$ is the distribution function of a discrete non-negative random variable $D$ such that
 \begin{equation} \label{defn::F}
 G(x)=1-F(x)= \frac{C_FL_0(x)}{x^{\tau-1}},
 \end{equation}
where $L_0(\cdot)$ is a slowly-varying function so that the tail of the distribution is decaying  like a regularly-varying function. Recall that the inverse of a locally bounded non-increasing function $f:\R\to\R$ is defined as $f^{-1}(x):=\inf\{y:f(y)\leq x\}$. Therefore, using \cite[Theorem 1.5.12]{BGT89}, \begin{equation}\label{eqn:inverse}
 G^{-1}(x)=\frac{C_F^{1/(\tau-1)}L(1/x)}{x^{1/(\tau-1)}}(1+o(1))\quad \text{as } x\to 0,
\end{equation}where $L(\cdot)$ is another slowly-varying function. Note that \cite[Theorem 1.5.12]{BGT89} is stated for positive exponents only. Since our exponent is negative, the asymptotics in \eqref{eqn:inverse} holds for $x\to 0$. 
Suppose that the random variable $D$ is such that
\begin{equation}\label{defn:critical}
 \nu=\frac{\expt{D(D-1)}}{\expt{D}}=1.
\end{equation}Define the degree sequence $\boldsymbol{d}_{\lambda}$ by taking the degree of the $i^{th}$ vertex to be
 \begin{equation}\label{defn::degree:powerlaw}
  d_i=d_i(\lambda):=G^{-1}(i/n)+\delta_{i,n}(\lambda),
  \end{equation}where the $\delta_{i,n}(\lambda)$'s are non-negative integers satisfying the asymptotic equivalence 
  \begin{equation}\label{defn:delta:i} \delta_{i,n}(\lambda) \sim \lambda G^{-1}(i/n) c_n^{-1},\quad \text{as } n\to\infty.
  \end{equation}The $\delta_{i,n}(\lambda)$'s are chosen in such a way that Assumption~\ref{assumption1}~\eqref{assumption1-4} is satisfied.
Fix any $K\geq 1$. 
Notice that  \eqref{eqn:inverse} and \eqref{defn:delta:i} imply that, for all large enough $n$ (independently of $K$), the first $K$ largest degrees $(d_i)_{i\in [K]}$ satisfy 
\begin{equation}\label{defn:high-degree}
 d_i= \left(\frac{n^{\alpha}C_F^{\alpha}L(n/i)}{i^{\alpha}}\right) \left( 1+\lambda c_n^{-1}+o(c_n^{-1})\right).
\end{equation}
Therefore, $\boldsymbol{d}_{\lambda}$ satisfies Assumption~\ref{assumption1}~\eqref{assumption1-1} with $\theta_i=(C_F/i)^{\alpha}$. 
The next two lemmas verify Assumption~\ref{assumption1}~\eqref{assumption1-2},~\eqref{assumption1-3}:

\begin{lemma} \label{lem::order_moments}The degree sequence $\boldsymbol{d}_{\lambda}$ defined in \eqref{defn::degree:powerlaw} satisfies {\rm Assumption~\ref{assumption1}~\eqref{assumption1-2}}.
\end{lemma}
\begin{proof}
 Note that, by \eqref{defn:high-degree}, $d_1^2=o(n)$.
 Also, since $G^{-1}$ is non-increasing 
 \begin{equation}
  \int_{0}^1G^{-1}(x)\dif x-\frac{d_1}{n} \leq \frac{1}{n}\sum_{i\in [n]} G^{-1}(i/n)\leq \int_{0}^1G^{-1}(x)\dif x.
 \end{equation}
  Therefore,
 \begin{equation}\label{tot-deg-error-1}
  \frac{1}{n}\sum_{i\in [n]} d_i=\frac{1}{n}\sum_{i\in [n]}G^{-1}(i/n)(1+O(c_n^{-1}))= \int_{0}^1G^{-1}(x)\dif x +O(d_1/n)+O(c_n^{-1})= \expt{D}+O(b_n^{-1}).
 \end{equation} 
 Similarly, $\sum_{i\in [n]}d_i^2=n\E[D^2]+O(d_1^2)=n\E[D^2]+o(n)$. 
 To prove the condition involving the third-moment, we use Potter's theorem \cite[Theorem 1.5.6]{BGT89}. 
 First note that $3\alpha-1 = (4-\tau)/(\tau-1)>0$ since $\tau\in (3,4)$. 
 Fix $0<\delta<\alpha -1/3$ and $A>1$ and choose $C=C(\delta,A)$ such that for all $i\leq nC^{-1}$, $L(n/i)/L(n)< Ai^{\delta}$. Therefore, \eqref{eqn:inverse} implies 
 \begin{equation}\label{eq:3rd-moment-example}
 a_n^{-3} \sum_{i>K}d_i^3\leq A \sum_{i>K}i^{-3\alpha+3\delta}+ \frac{\sup_{1\leq x\leq C}L(x)^3}{L(n)^3} \sum_{i>nC^{-1}}i^{-3\alpha}.
 \end{equation} 
 From our choice of $\delta$, $-3\alpha+3\delta <-1$ and therefore $\sum_{i\geq 1}i^{-3\alpha+3\delta}<\infty$. By \cite[Lemma 1.3.2]{BGT89},  $\sup_{1\leq x\leq C}L(x)^3<\infty$. Moreover, $\sum_{i>nC^{-1}}i^{-3\alpha}=O(n^{1-3\alpha})$ and $1-3\alpha<0$. Thus,  the proof follows by first taking $n\to \infty$ and then $K\to\infty$. 
\end{proof}
\begin{lemma}\label{lem::nu-n} The degree sequence $\boldsymbol{d}_{\lambda}$ defined in \eqref{defn::degree:powerlaw} satisfies  {\rm Assumption~\ref{assumption1}~\eqref{assumption1-3}}, i.e., there exists $\lambda_0\in \R$ such that
 \begin{equation} \label{eqn::lem::nu_n}
  \nu_n(\lambda)= 1+(\lambda+\lambda_0) c_n^{-1}+o(c_n^{-1}).
 \end{equation}
\end{lemma} 
\begin{proof}
 Firstly, Lemma \ref{lem::order_moments} guarantees  the convergence of the second moment of the degree sequence. However, \eqref{eqn::lem::nu_n} is more about obtaining sharper asymptotics for $\nu_n(\lambda)$. We use similar arguments as in  \cite[Lemma 2.2]{BHL12}. Denote  $\nu_n:=\nu_n(0)$. Note that $\nu_n(\lambda)=\nu_n(1+\lambda c_n^{-1})+o(c_n^{-1})$, so it is enough to verify that \begin{equation}
\nu_n=1+\lambda_0c_n^{-1}+o(c_n^{-1}).
\end{equation}
Consider $d_i(0)$ as given in \eqref{defn::degree:powerlaw} with $\lambda=0$. Lemma~\ref{lem::order_moments} implies
\begin{equation}
 \nu_n = \frac{\sum_{i\in [n]} d_i(0)^2}{n\expt{D}}-1+o(c_n^{-1}).
\end{equation}
 Fix any $K\geq 1$.  We have
 \begin{equation}
  \int_{K/n}^{1} G^{-1}(u)^2 \dif u - \frac{d_K^2}{n}\leq \frac{1}{n}\sum_{i=K+1}^n d_i^2 \leq \int_{K/n}^{1} G^{-1}(u)^2 \dif u.
 \end{equation}Now by \eqref{defn::degree:powerlaw}, $d_K^2/n= \Theta(K^{-2\alpha}L(n/K)^2n^{-\eta})$. Therefore, 
 \begin{equation} \label{expr::nu_difference}
  \nu-\nu_n = \frac{1}{\expt{D}}\left( \sum_{i=1}^{K} \int_{(i-1)/n}^{i/n}G^{-1}(u)^2 \dif u- \frac{1}{n}\sum_{i=1}^K d_i^2 \right)+O(K^{-2\alpha}L(n/K)^2n^{-\eta}).
 \end{equation} Again, using \eqref{defn::degree:powerlaw}, 
 \begin{equation}\label{expr::first_sum}
  \frac{1}{n}\sum_{i=1}^K d_i^2  
  = n^{-\eta} \sum_{i=1}^K \left( \frac{C_F}{i}\right)^{2\alpha}L(n/i)^2+o(c_n^{-1})=c_n^{-1}\sum_{i=1}^K \left( \frac{C_F}{i}\right)^{2\alpha}+\varepsilon(c_n,K),
 \end{equation} where the last equality follows using the fact that $L(\cdot)$ is a slowly-varying function. Note that the error term $\varepsilon(c_n,K)$ in \eqref{expr::first_sum} satisfies $\lim_{n\to\infty} c_n\varepsilon(c_n,K)=0$ for each fixed $K\geq 1$. Again,
 \begin{equation} \label{expr::second_sum}
 \begin{split}
   \sum_{i=1}^K \int_{(i-1)/n}^{i/n} G^{-1}(u)^2 \dif u &= n^{-\eta} \sum_{i=1}^K \int_{i-1}^{i}\left( \frac{C_F}{u}\right)^{2\alpha}L(n/u)^2\dif u+o(c_n^{-1})\\
   &=c_n^{-1}\sum_{i=1}^K \int_{i-1}^{i}\left( \frac{C_F}{u}\right)^{2\alpha}\dif u +\varepsilon'(c_n,K),
   \end{split}
 \end{equation}where $\lim_{n\to\infty} c_n\varepsilon'(c_n,K)=0$ for each fixed $K\geq 1$.
 Thus combining \eqref{expr::nu_difference}, \eqref{expr::first_sum}, and  \eqref{expr::second_sum} and first letting $n\to \infty$ and then $K\to\infty$, we get
 \begin{equation}
  \lim_{n\to\infty} c_n(\nu_n-\nu)=\lambda_0,
 \end{equation}where 
 \begin{equation}\lambda_0 =-\frac{C_F^{2\alpha}}{\expt{D}}\sum_{i=1}^{\infty} \left(  \int_{i-1}^{i}u^{-2\alpha}\dif u-i^{-2\alpha} \right).
 \end{equation}
 Using Euler-Maclaurin summation \cite[Page 333]{H49} it can be seen that  $\lambda_0$ is finite which completes the proof. 
\end{proof}

\begin{remark}
 \normalfont
 Note that if we add approximately $cnc_n^{-1}$ ($c>0$ is a constant) ones in the degree sequence given in \eqref{defn::degree:powerlaw}, then  we end up with another configuration model for which $\lim_{n\to\infty}c_n(\nu_n-\nu)= \zeta'$ with $\zeta> \zeta'$. Similarly, deleting $cnc_n^{-1}$ ones from the degree sequence increases the new $\zeta$ value. This gives an obvious way to perturb the degree sequence in such a way that the configuration model is in different locations within the critical scaling window. In our proofs, we will only use the precise asymptotics of the \emph{high}-degree vertices. Thus, a small (suitable) perturbation in the degrees of the \emph{low}-degree vertices does not change the scaling behavior fundamentally, except for a change in the location inside the scaling window.
\end{remark}  

\begin{remark} \normalfont If $\nu$ in \eqref{defn:critical} is larger than one, then the degree sequence satisfies Assumption~\ref{assumption2}. Therefore, the results for critical percolation also hold in this setting. 
\eqref{tot-deg-error-1} implies that the additional assumption in Theorem~\ref{thm:mul:conv} is also satisfied.

\end{remark}
\subsection{Random degrees sampled from a power-law distribution}\label{sec:iid-degrees}
  We now consider the set-up discussed in \cite{Jo10}. Let $\Delta_1,\dots, \Delta_n$ be i.i.d.~samples from a distribution $F$, where $F$ is defined in \eqref{defn::F}. Therefore, the asymptotic relation in \eqref{eqn:inverse} holds.
   Consider the random degree sequence $\boldsymbol{d}$ where $d_i=\Delta_{\sss(i)}$, $\Delta_{\sss(i)}$ being the $i^{th}$ order statistic of $(\Delta_1,\dots, \Delta_n)$. We show that $\boldsymbol{d}$ satisfies Assumption~\ref{assumption1} almost surely under a suitable coupling. We use a coupling from \cite[Section 13.6]{B68}. 
   Let $(E_1,E_2,\dots )$ be an i.i.d.~sequence of unit rate exponential random variables and let $\Gamma_i:= \sum_{j=1}^iE_j$. Let
  \begin{equation}
   \bar{d}_i=G^{-1}(\Gamma_i/\Gamma_{n+1}).
  \end{equation}   
It can be checked that $(d_1,\dots,d_n)\stackrel{\sss d}{=}(\bar{d}_1,\dots,\bar{d}_n)$ and therefore we will ignore the bar in the subsequent notation. Note that, by the strong law of large numbers, $\Gamma_{n+1}/n \asto 1$. Thus,  for each fixed $i\geq 1$, $\Gamma_{n+1}/(n\Gamma_i)\asto 1/\Gamma_i$. 
Using \eqref{eqn:inverse}, we see that $\boldsymbol{d}$ satisfies Assumption~\ref{assumption1}~\eqref{assumption1-1} almost surely under this coupling with $\theta_i=(C_F/\Gamma_i)^{\alpha}$. 
To see that $\bld{\theta} \in \ell^3_{\shortarrow} \setminus \ell^2_{\shortarrow}$ almost surely, we need to show that 
\begin{gather}\label{eq:gamma-summable}
\PR\bigg(\sum_{i=1}^{\infty}\Gamma_i^{-3\alpha}<\infty\bigg)=1,\quad 
 \PR\bigg(\sum_{i=1}^{\infty}\Gamma_i^{-2\alpha} = \infty\bigg)=1.
\end{gather} 
\eqref{eq:gamma-summable} follows easily from the observations $2\alpha<1<3\alpha$ and $\Gamma_i/i\asto 1$ as $i\to\infty$.

Next, the first two conditions of Assumption~\ref{assumption1}~\eqref{assumption1-2} are trivially satisfied by $\boldsymbol{d}$  almost surely using the strong law of large numbers.
 To see the third condition, using \eqref{eq:gamma-summable}, and $\Gamma_{n+1}/n\asto 1$, we can use arguments identical to \eqref{eq:3rd-moment-example} to show that $\lim_{K\to\infty}\limsup_{n\to\infty} a_n^{-3}\sum_{i>K}d_i^3=0$ on the event $\{\sum_{i=1}^{\infty}\Gamma_i^{-3\alpha}<\infty\}\cap \{\Gamma_{n+1}/n\to 1\}$. Thus, we have shown that the third condition of Assumption~\ref{assumption1}~\eqref{assumption1-2} holds almost surely.
\par To see Assumption~\ref{assumption1}~\eqref{assumption1-3}, an argument similar to Lemma~\ref{lem::nu-n} can be carried out to prove that 
\begin{equation}
 \lim_{n\to\infty}c_n(\nu_n-\nu)\asto \Lambda_0,
\end{equation}where 
\begin{equation}\label{defn:Lambda-0}\Lambda_0:= -\frac{C_F^{2\alpha}}{\expt{D}}\sum_{i=1}^{\infty} \left(  \int_{\Gamma_{i-1}}^{\Gamma_i}u^{-2\alpha}\dif u-\Gamma_i^{-2\alpha} \right).
\end{equation}Therefore, the results in Section~\ref{sec:results} hold conditionally on the degree sequence if we assume the degrees to be an i.i.d.~sample from a distribution of the form \eqref{defn::F}. 
For the percolation results, notice that the additional condition in Theorem~\ref{thm:mul:conv} is a direct consequence of the convergence rates of sums of i.i.d.~sequence of random variables~\cite[Corollary 3.22]{K2006}.

\begin{remark}\label{rem:jos-vs-this}\normalfont Let us recall the limiting object obtained in \cite[Theorem 8.1]{Jo10} and compare this with the limiting object $\bar{\mathbf{S}}_\infty^{\sss \Lambda_0}$, defined in \eqref{defn::limiting::process} with $\theta_i=(C_F/\Gamma_i)^{\alpha}$ and $\Lambda_0$ given by \eqref{defn:Lambda-0}. We will prove an analogue of \cite[Theorem 8.1]{Jo10} in Theorem~\ref{thm::convegence::exploration_process}. Although we use a different exploration process from \cite{Jo10}, the fact that the component sizes are \emph{huge} compared to the number of cycles in a component, one can prove Theorem~\ref{thm::convegence::exploration_process} for the exploration process in \cite{Jo10} also. 
This will indirectly imply that \citeauthor{Jo10}'s limiting object obeys the law of $\bar{\mathbf{S}}_\infty^{\sss \Lambda_0}$, averaged out over the $\Gamma$-values. 
This is counter intuitive, given the vastly different descriptions of the two processes; for example our process does not have independent increments. 
We do not have a direct way to prove the above mentioned claim. 
\end{remark}
\section{Discussion}\label{sec:discussion}
\noindent {\bf Assumptions on the degree distribution.} Let us now briefly explain the significance of Assumption~\ref{assumption1}. Unlike the finite third-moment case \cite{DHLS15}, the high-degree vertices dictate the scaling limit in Theorem~\ref{thm::conv:component:size} and therefore it is essential to fix their asymptotics through Assumption~\ref{assumption1}~\eqref{assumption1-1}. Assumption~\ref{assumption1}~\eqref{assumption1-3} defines the critical window of the phase transition and Assumption~\ref{assumption1}~\eqref{assumption1-4} is reminiscent of the fact that a configuration model with negligibly small amount of degree-one vertices is always supercritical. Assumption~\ref{assumption1}~\eqref{assumption1-2} states the finiteness of the first two moments of the degree distribution and fixes the asymptotic order of the third-moment. The order of the third-moment is crucial in our case. The derivation of the scaling limits for the components sizes is based on the analysis of a walk which encodes the information about the component sizes in terms of the excursions above its past minima \cite{A97,R12,DHLS15,BHL12,BHL10}. Now, the increment distribution turns out to be the size-biased distribution with the sizes being the degrees. Therefore, the third-moment assumption controls the variance of the increment distribution. Another viewpoint is that the components can be locally approximated by a branching process $\mathcal{X}_n$ with the variance of the same order as the third moment of the degree distribution. Thus Assumption~\ref{assumption1}~\eqref{assumption1-2} controls the order of the survival probability of $\mathcal{X}_n$, which is intimately related to the asymptotic size of the largest components. \vspace{.2cm}
\\ 
 {\bf Connecting the barely subcritical and supercritical regimes.} The barely subcritical (supercritical) regime corresponds to the case when $\nu_n(\lambda_n)=1+\lambda_nc_n^{-1}$ for some $\lambda_n\to -\infty$ ($\lambda_n\to \infty$) and $\lambda_n = o(c_n)$. \citet{J08} showed that the size of the $k^{th}$ largest cluster for a subcritical configuration model (i.e., the case $\nu_n\to\nu$ and $\nu<1$) is $d_k/(1-\nu)$ (see \cite[Remark 1.4]{J08}). In \cite{BDHS17}, we show that this is indeed the case for the entire barely subcritical regime, i.e., the size of the $k^{th}$ largest cluster is $d_k/(1-\nu_n(\lambda_n))=\Theta(b_n|\lambda_n|^{-1})$.
 In the barely supercritical case, the giant component can be \emph{locally} approximated  by a  branching process $\mathcal{X}_n$ having variance of the order $a_n^3/n$ and the size of the giant component is of the order $n\rho_n$, where $\rho_n$ is the survival probability of $\mathcal{X}_n$~\cite{HJL16}. 
The asymptotic size of the giant component turns out to be $\Theta(b_n|\lambda_n|)$.
 Therefore, the fact that the sizes of the maximal components in the critical scaling window are $\Theta(b_n)$ for $\lambda_n=\Theta(1)$ proves a continuous phase transition property for the configuration model within the whole critical regime. 
  \vspace{.2cm}\\
{\bf Percolation.} The main reason to study percolation in this paper is to understand the evolution of the component sizes and the surplus edges over the critical window in Theorem~\ref{thm:mul:conv}. 
It turns out that a precise characterization of the evolution of the percolation clusters is necessary to understand the minimal spanning tree of the giant component with i.i.d.~continuous weights on each edge~\cite{ABGM13}.
Also, since the percolated configuration model is again a configuration model \cite{F07,J09}, the natural way to study the evolution of the clusters sizes of configuration models over the critical window is through percolation. \vspace{.2cm} \\
{\bf Universality.} The limiting object in Theorem~\ref{thm::conv:component:size} is identical to that in \cite[Theorem 1.1]{BHL12} for rank-one inhomogeneous random graphs. 
Thus, $\mathrm{CM}_n(\boldsymbol{d})$ with regularly-varying tails lies in the domain of attraction of the new universality class studied in \cite{BHL12}. This is again confirming the predictions made by statistical physicists that the nature of the phase transition does not depend on the precise details of the model. Our scaling limit fits into the general class of limits predicted in \cite{AL98}. 
In the notation of \cite[(6)]{AL98}, the scaling limits $\mathrm{CM}_n(\boldsymbol{d})$, under Assumption~\ref{assumption1}, give rise to the case $\kappa=0$. 
To understand this, let us discuss some existing works. In \cite{A97,AP00,DHLS15,Jo10,BHL10}, the limiting component sizes are  described by the excursions of a Brownian motion with a parabolic drift. 
All these models have a common property: if the component sizes in the barely subcritical regime are viewed as masses then (i) these masses merge as approximate multiplicative coalescents in the critical window, and (ii) each individual mass is negligible/``dust" compared to the sum of squares of the masses in the barely subcritical regime. Indeed, (ii) has been established in \cite[(10)]{A97}, \cite[(4)]{AP00}. 
In the case of \cite{BHL12} and this paper, the barely subcritical component sizes do not become negligible due to the existence of the high-degree vertices (see \cite[Theorem 1.3]{BHL12}). As discussed in \cite[Section 1.4]{AL98}, these \emph{large} barely subcritical clusters can be thought of as nuclei, not interacting with each other and ``sweeping up the smaller clusters in such a way that the relative masses converge''.  It will be fascinating to find a class of random graphs, used to model real-life networks, that has both the nuclei and a good amount of dust in the barely subcritical regime, so that the scaling limits predicted by \cite{AL98} can be observed in complete generality.
 \vspace{.2cm}\\
{\bf Component sizes and the width of the critical window}. We have already discussed how the width of the scaling window and the order of the maximal degrees should lead the asymptotic size of the components to be of the order $b_n$. For the finite third-moment case, the size of the largest component is of the order $n^{2/3}\gg b_n$. We do not have an  intuitive explanation to explain the reduced sizes of the components except for the fact that a similar property is true for the survival probability of a slightly supercritical branching process. The width of the critical window decreases by a factor of $L(n)^{-2}$ as compared to \cite{BHL12} if the size of the high-degree vertices increases by a factor of $L(n)$ (see \eqref{eqn:notation-power*}). Indeed, an increase in the degrees of the high-degree vertices is expected to start the merging of the barely subcritical nuclei earlier, resulting in an increase in the width of the critical window. The fact that the width decreases by a factor of $L(n)^{-2}$ comes out of our calculations. \vspace{.2cm} \\
 {\bf Open problems}. \\ (i) A natural question is to study what the component sizes, viewed as metric spaces, look like.  Recently, \cite{BHS15} studied this problem for rank-1 inhomogeneous random graphs for heavy-tailed weights. In  \cite{BDHS17,BDHS18}, we show that the metric space structure of $\mathrm{CM}_n(\boldsymbol{d})$ is in the  same universality class of the rank-one inhomogeneous model, as shown in \cite{BHS15}. This is the first step in understanding the minimal spanning tree problem (see \cite{ABGM13}) for $\CM$.\\
(ii) As discussed in Section~\ref{sec:iid-degrees} (see Remark~\ref{rem:jos-vs-this}), it will be interesting to get a direct proof of the fact that the limiting object in \cite[Theorem 8.1]{Jo10} is obtained by averaging the distribution of $\mathbf{S}_{\infty}^{\sss \Lambda_0}$ over the collections $(\Gamma_i)_{i\geq 1}$. \\
 (iii) We have only shown the finite-dimensional convergence in Theorem~\ref{thm:mul:conv}. 
 It is an open question to obtain a suitable tightness criterion that would imply the process level convergence of the vector of component sizes and surplus edges over the whole critical window.
\vspace{.2cm}\\
{\bf Overview of the proofs}.  The proofs of Theorems~\ref{thm::conv:component:size} and \ref{thm:spls} consist of two important steps. First, we define an exploration algorithm on the graph that explores one edge of the graph at each step. The algorithm produces a walk, that we call  exploration process, that encodes the information about the number of edges in the explored components in terms of the hitting times to its past minima. In Section~\ref{sec:conv-expl}, the exploration process, suitably rescaled, is shown to converge. The surplus edges in the components are asymptotically negligible  compared to the component sizes; these two facts together give us the finite-dimensional scaling limit of the re-scaled component sizes. The proof of Theorem~\ref{thm::conv:component:size} follows from the asymptotics of the susceptibility function in Section~\ref{sec:conv-comp-size}.
 The joint convergence of the component sizes and surplus edges is proved by verifying a uniform tightness condition on the surplus edges in Section~\ref{sec:surplus}. Then, in Section~\ref{sec:simple-graphs}, we exploit the idea that the large components are explored before any self-loops or multiple edges are created and conclude the proof of Theorem~\ref{thm:simple-graph}. The proof of Theorem~\ref{thm:percolation} is completed by showing that the percolated configuration model is again a configuration model satisfying Assumption~\ref{assumption1}. Section~\ref{sec:conv-amc} is devoted to the proof of Theorem~\ref{thm:mul:conv} which exploits different properties of the augmented multiplicative coalescent process.
\section{Convergence of the exploration process}\label{sec:conv-expl}
  We start by describing how the connected components in the graph can be explored while generating the random graph simultaneously:
\begin{algo}[Exploring the graph]\label{algo-expl}\normalfont  Consider the configuration model $\mathrm{CM}_{n}(\boldsymbol{d})$. The algorithm carries along vertices that can be alive, active, exploring and killed, and half-edges that can be alive, active or killed. 
We sequentially explore the graph as follows:
\begin{itemize}
\item[(S0)] At stage $i=0$, all the vertices and the half-edges are \emph{alive} but none of them are \emph{active}. Also, there are no \emph{exploring} vertices. 
\item[(S1)]  At each stage $i$, if there is no active half-edge at stage $i$, choose a vertex $v$ proportional to its degree among the alive (not yet killed) vertices and declare all its half-edges to be \emph{active} and declare $v$ to be \emph{exploring}. If there is an active vertex but no exploring vertex, then declare the \emph{smallest} vertex to be exploring.
\item[(S2)] At each stage $i$, take an active half-edge $e$ of an exploring vertex $v$ and pair it uniformly to another alive half-edge $f$. Kill $e,f$. If $f$ is incident to a vertex $v'$ that has not been discovered before, then declare all the half-edges incident to $v'$ active (if any), except $f$. 
If $\mathrm{degree}(v')=1$ (i.e. the only half-edge incident to $v'$ is $f$) then kill $v'$. Otherwise, declare $v'$ to be active and larger than all other vertices that are alive. After killing $e$, if $v$ does not have another active half-edge, then kill $v$ also.

\item[(S3)] Repeat from (S1) at stage $i+1$ if not all half-edges are already killed.
\end{itemize}
\end{algo}
Algorithm~\ref{algo-expl} gives a \emph{breadth-first} exploration of the connected components of $\mathrm{CM}_n(\boldsymbol{d})$. Define the exploration process by
   \begin{equation}\label{defn:exploration:process}
    S_n(0)=0,\quad
     S_n(l)=S_n(l-1)+d_{(l)}J_l-2,
    \end{equation} where $J_l$ is the indicator that a new vertex is discovered at time $l$ and $d_{(l)}$ is the degree of the new vertex chosen at time $l$ when $J_l=1$.  Suppose $\mathscr{C}_{k}$ is the $k^{th}$ connected component explored by the above exploration process and define
$\tau_{k}=\inf \big\{ i:S_{n}(i)=-2k \big\}.$
Then  $\mathscr{C}_{k}$ is discovered between the times $\tau_{k-1}+1$ and $\tau_k$, and  $\tau_{k}-\tau_{k-1}-1$ gives the total number of edges in $\mathscr{C}_k$.
 Call a vertex \emph{discovered} if it is either active or killed. Let $\mathscr{V}_l$ denote the set of vertices discovered up to time $l$ and $\mathcal{I}_i^n(l):=\ind{i\in\mathscr{V}_l}$. Note that 
   \begin{equation}
    S_n(l)= \sum_{i\in [n]} d_i \mathcal{I}_i^n(l)-2l=\sum_{i\in [n]} d_i \left( \mathcal{I}_i^n(l)-\frac{d_i}{\ell_n}l\right)+\left( \nu_n(\lambda)-1\right)l.
   \end{equation} 
   Recall the notation in \eqref{eqn:notation-power*}. Define the re-scaled version $\bar{\mathbf{S}}_n$ of $\mathbf{S}_n$ by $\bar{S}_n(t)= a_n^{-1}S_n(\lfloor b_nt \rfloor)$. Then, by Assumption~\ref{assumption1}~\eqref{assumption1-3},
   \begin{equation} \label{eqn::scaled_process}
    \bar{S}_n(t)= a_n^{-1} \sum_{i\in [n]}d_i\left( \mathcal{I}_i^n(tb_n)-\frac{d_i}{\ell_n}tb_n \right)+\lambda t +o(1).
   \end{equation}Note the similarity between the expressions in \eqref{defn::limiting::process} and \eqref{eqn::scaled_process}. We will prove the following:
   \begin{theorem} \label{thm::convegence::exploration_process} Consider the process $\bar{\mathbf{S}}_n:= (\bar{S}_n(t))_{t\geq 0}$ defined in \eqref{eqn::scaled_process} and recall the definition of  $\bar{\mathbf{S}}_\infty^\lambda:=  (\bar{S}_\infty^\lambda(t))_{t\geq 0} $ from \eqref{defn::limiting::process}. Then, 
 \begin{equation}
  \bar{\mathbf{S}}_n \dto \bar{\mathbf{S}}_\infty^\lambda
 \end{equation} with respect to the Skorohod $J_1$-topology.
\end{theorem}
 The proof of Theorem~\ref{thm::convegence::exploration_process} is completed by showing that the summation term in  \eqref{eqn::scaled_process} is predominantly carried by the first \emph{few} terms and the limit of the first few terms gives rise to the limiting process given in \eqref{defn::limiting::process}. 
 Fix $K\geq 1$ to be large. 
 Denote by $\mathscr{F}_l$ the sigma-field containing the information generated up to time $l$ by Algorithm~\ref{algo-expl}. 
 Also, let $\Upsilon_l$ denote the set of time points  up to time $l$ when a component was discovered and $\upsilon_l=|\Upsilon_l|$. 
 Note that we have lost $2(l-\upsilon_l)$ half-edges by time $l$. 
 Thus, on the set $\{\mathcal{I}_i^n(l)=0\}$,
  \begin{equation}\label{eqn:increment:indicator}
  \begin{split}
   \prob{\mathcal{I}_i^n(l+1)=1 \mid \mathscr{F}_l}=
   \begin{cases} \frac{d_i}{\ell_n-2(l-\upsilon_l)-1} & \text{ if } l\notin \Upsilon_l,\\
    \frac{d_i}{\ell_n-2(l-\upsilon_l)} & \text{ otherwise. } 
   \end{cases}
   \end{split}
  \end{equation}
Let $\ell_n(T) = \ell_n -2Tb_n+1$. Then,  uniformly over $l\leq Tb_n$, 
  \begin{equation}\label{eq:prob-ind}
  \prob{\mathcal{I}_i^n(l+1)=1\mid \mathscr{F}_l} \leq \frac{d_i}{\ell_n(T)}\quad\text{ on the set } \{\mathcal{I}_i^n(l)=0\}.
  \end{equation} 
Also, $\mathcal{I}_i^n(l+1) -\mathcal{I}_i^n(l) =0$ on the set  $\{\mathcal{I}_i^n(l)=1\}$.
 Denote $
  M_n^K(l)=a_n^{-1} \sum_{i\in [n]}d_i\big( \mathcal{I}_i^n(l)-\frac{d_i}{\ell_n(T)}l \big).$ Then,
 \begin{equation}
 \begin{split}
  &\E\big[M_n^K(l+1)-M_n^K(l) \mid \mathscr{F}_l\big]=\E\bigg[\sum_{i=K+1}^na_n^{-1}d_i \left(\mathcal{I}^n_i(l+1)-\mathcal{I}_i^n(l)-\frac{d_i}{\ell_n(T)}\right)\mid \mathscr{F}_l\bigg]\\
  &= \sum_{i=K+1}^n a_n^{-1}d_i \left(\E\big[\mathcal{I}^n_i(l+1)\mid \mathscr{F}_l\big]\ind{\mathcal{I}_i^n(l)=0} - \frac{d_i}{\ell_n(T)} \right)\leq 0.
  \end{split}
 \end{equation} Thus $(M_n^K(l))_{l= 1}^{Tb_n}$ is a supermartingale.  Further, \eqref{eqn:increment:indicator} implies that, uniformly for all $l\leq Tb_n$,
  \begin{equation}\label{prob-ind-lb}
   \prob{\mathcal{I}_i^n(l)=0} \leq \left(1-\frac{d_i}{\ell_n} \right)^l.
  \end{equation} 
  Thus, Assumption~\ref{assumption1}~\eqref{assumption1-2} gives
  \begin{equation} \begin{split}
    \big| \E[M_n^K(l)]\big|&= a_n^{-1} \sum_{i=K+1}^n d_i\left( \frac{d_i}{\ell_n(T)}l  -\prob{\mathcal{I}_i^n(l)=1}\right)
    \\& \leq a_n^{-1} \sum_{i=K+1}^n d_i\bigg| 1-\left(1-\frac{d_i}{\ell_n} \right)^l-\frac{d_i}{\ell_n}l \bigg|+a_n^{-1}l\sum_{i\in [n]}d_i^2\left(\frac{1}{\ell_n(T)}-\frac{1}{\ell_n}\right)\\
    &\leq \frac{l^2}{2\ell_n^2 a_n } \sum_{i=K+1}^n d_i^3+o(1) 
    \\&\leq \frac{T^2n^{2\rho}n^{3\alpha}L(n)^3}{2\ell_n^2L(n)^2n^{\alpha}L(n)}\left( a_n^{-3}\sum_{i=K+1}^{n}d_i^3\right)+o(1)=C\bigg(a_n^{-3}\sum_{i=K+1}^{n}d_i^3\bigg)+o(1),
  \end{split}
  \end{equation} for some constant $C>0$, where we have used the fact that 
  \begin{equation}
  a_n^{-1}l\sum_{i\in [n]}d_i^2\Big(\frac{1}{\ell_n(T)}-\frac{1}{\ell_n}\Big)=O(n^{2\rho+1-\alpha-2}/L(n)^3)=O(n^{(\tau-4)/(\tau-1)}/L(n)^3)=o(1),
  \end{equation} uniformly for $l\leq Tb_n$. Therefore, uniformly over $l\leq Tb_n$,
  \begin{equation}\label{expectation::M_n^K}
   \lim_{K\to\infty}\limsup_{n\to\infty}\big| \E[M_n^K(l)]\big|=0.
  \end{equation} Now, note that for any $(x_1,x_2,\dots)$, $0\leq a+b \leq x_i$ and $a,b>0$ one has $\prod_{i=1}^R(1-a/x_i)(1-b/x_i)\geq \prod_{i=1}^R (1-(a+b)/x_i)$. Thus, by \eqref{eqn:increment:indicator}, for all $l\geq 1$ and $i\neq j$, 
  \begin{equation}\label{neg:correlation}
  \prob{\mathcal{I}_i^n(l)=0, \mathcal{I}_j^n(l)=0}\leq \prob{\mathcal{I}_i^n(l)=0}\prob{\mathcal{I}_j^n(l)=0}
  \end{equation} and therefore $\mathcal{I}_i^n(l)$ and $\mathcal{I}^n_j(l)$ are negatively correlated. Observe also that, uniformly over $l\leq Tb_n$, 
  \begin{equation}\label{var-ind-ub}
   \var{\mathcal{I}_i^n(l)}\leq  \prob{\mathcal{I}_i^n(l)=1} \leq \sum_{l_1=1}^l\prob{\text{vertex  }i \text{ is first discovered at stage }l_1 }\leq \frac{ld_i }{\ell_n(T)}.
  \end{equation}  
  Therefore, using the negative correlation in \eqref{neg:correlation}, uniformly over $l\leq Tb_n$, 
  \begin{equation} \label{variance::M_n^k}
   \begin{split}
    \var{M_n^K(l)}&\leq a_n^{-2}\sum_{i=K+1}^n d_i^2 \var{\mathcal{I}_i^n(l)} \leq \frac{l}{\ell_n(T)a_n^2}\sum_{i=K+1}^n d_i^3 \leq Ca_n^{-3}\sum_{i=K+1}^{n}d_i^3,
   \end{split}
  \end{equation}for some constant $C>0$ and by using Assumption~\ref{assumption1}~\eqref{assumption1-2} again,
  \begin{equation}
   \lim_{K\to\infty}\limsup_{n\to\infty} \var{M_n^K(l)}= 0,
  \end{equation}uniformly for $l\leq Tb_n$.
 Now we can use the super-martingale inequality \cite[Lemma 2.54.5]{RW94} stating that for any super-martingale $(M(t))_{t\geq 0}$, with $M(0)=0$, 
 \begin{equation}\label{eqn:supmg:ineq}
  \varepsilon \prob{\sup_{s\leq t}|M(s)|>3\varepsilon}\leq 3\expt{|M(t)|}\leq 3\left(|\expt{M(t)}|+\sqrt{\var{M(t)}}\right).
 \end{equation}
  Using \eqref{expectation::M_n^K}, \eqref{variance::M_n^k}, and \eqref{eqn:supmg:ineq}, together with the fact that $(-M_n^K(l))_{l=1}^{Tb_n}$ is a super-martingale, we get, for any $\varepsilon >0 $,
 \begin{equation} \label{tail::martingale}
  \lim_{K\to\infty}\limsup_{n\to\infty}\PR\bigg(\sup_{l\leq Tb_n}|M_n^K(l)|> \varepsilon \bigg)=0.
 \end{equation}
Define the truncated exploration process
\begin{equation}\label{eqn::truncated_scaled_process}
  \bar{S}_n^K(t)= a_n^{-1} \sum_{i=1}^Kd_i\left( \mathcal{I}_i^n(tb_n)-\frac{d_i}{\ell_n}tb_n \right)+\lambda t.
\end{equation}
 Define $\mathcal{I}_i^n(tb_n)=\mathcal{I}_i^n(\floor{tb_n})$ and recall that $\mathcal{I}_i(s):=\ind{\xi_i\leq s }$ where $\xi_i\sim \mathrm{Exp}(\theta_i/\mu)$. 
 \begin{lemma} \label{lem::convergence_indicators}
   Fix any $K\geq 1$. As $n\to\infty$,
   \begin{equation}
    \left( \mathcal{I}_i^n(tb_n) \right)_{i\in[K],t\geq 0} \dto \left( \mathcal{I}_i(t) \right)_{i\in[K],t\geq 0}.
   \end{equation}
 \end{lemma}
\begin{proof} By noting that $(\mathcal{I}_i^n(tb_n))_{t\geq 0}$ are indicator processes, it is enough to show that 
\begin{equation}
 \prob{\mathcal{I}_i^n(t_ib_n)=0\ \forall i\in [K]} \to \prob{\mathcal{I}_i(t_i)=0\ \forall i\in [K]} = \exp \Big( -\mu^{-1}\sum_{i=1}^{K} \theta_it_i\Big).
\end{equation} for any $t_1,\dots,t_K\in \mathbbm{R}$. Now, 
\begin{equation} \label{lem::eqn::expression1}
 \prob{\mathcal{I}_i^n(m_i)=0,\ \forall i\in [K]}=\prod_{l=1}^{\infty}\Big(1-\sum_{i\leq K:l\leq m_i}\frac{d_i}{\ell_n-\Theta(l)} \Big),
\end{equation}where the $\Theta(l)$ term arises from the expression in \eqref{eqn:increment:indicator} and we note that $\upsilon_l\leq l$. Taking logarithms on both sides of \eqref{lem::eqn::expression1} and using the fact that $l\leq \max m_i=\Theta(b_n)$ we get 
\begin{equation}\label{lem::eqn::ex1}
 \begin{split}
  \prob{\mathcal{I}_i^n(m_i)=0\, \forall i\in [K]}&= \exp\Big( - \sum_{l=1}^{\infty}\sum_{i\leq K:l\leq m_i} \frac{d_i}{\ell_n}+o(1) \Big)= \exp\Big( -\sum_{i\in [K]} \frac{d_im_i}{\ell_n} +o(1) \Big).
 \end{split}
\end{equation} Putting $m_i=t_ib_n$, Assumption~\ref{assumption1}~\eqref{assumption1-1},~\eqref{assumption1-2} gives
\begin{equation} \label{lem::eqn::expression2}
 \frac{m_id_i}{\ell_n}= \frac{\theta_it_i}{\mu} (1+o(1)).
\end{equation}
Hence \eqref{lem::eqn::expression2}, and \eqref{lem::eqn::ex1} complete the proof of Lemma \ref{lem::convergence_indicators}.
\end{proof}

\begin{proof}[Proof of Theorem~\ref{thm::convegence::exploration_process}]
 The proof of Theorem~\ref{thm::convegence::exploration_process} now follows from \eqref{eqn::scaled_process}, \eqref{tail::martingale} and Lemma~\ref{lem::convergence_indicators} by first taking the limit as $n\to \infty$ and then taking the limit as $K\to\infty$.
\end{proof}

For future purposes, we also describe the scaling limit of the reflected process: 
\begin{theorem}\label{thm:conv:refl:process} Recall the definition of $\refl{\bar{\mathbf{S}}_\infty^\lambda}$ from \eqref{defn::reflected-Levy}. As $n\to\infty$,
\begin{equation}
 \refl{\bar{\mathbf{S}}_n} \dto \refl{\bar{\mathbf{S}}_\infty^\lambda}.
\end{equation}
\end{theorem} 
\begin{proof}
 This follows from Theorem~\ref{thm::convegence::exploration_process} and the fact that the reflection is Lipschitz continuous with respect to the Skorohod $J_1$-topology (see \cite[Theorem 13.5.1]{W02}).
\end{proof}

\section{Convergence of component sizes}\label{sec:conv-comp-size} In this section, we complete the proof of Theorem~\ref{thm::conv:component:size}. First, we prove a tail summability condition that ensures that the vector of ordered component sizes is tight in $\ell^2_{\shortarrow}$. 
This also implies that Algorithm~\ref{algo-expl} explores the \emph{large} components before time $Tb_n$ for large $T$. 
Next, we show that the function mapping an element of $\mathbb{D}[0,\infty)$ to its largest excursions is continuous on a special subset $A$ of $\mathbb{D}[0,\infty)$ and the process $\refl{\bar{\mathbf{S}}_{\infty}}$ has sample paths in $A$ almost surely. Therefore, Theorem~\ref{thm::convegence::exploration_process} gives the scaling limit of the number of edges in the components ordered as a non-increasing sequence. 
Finally, we show that the number of surplus edges discovered up to time  $Tb_n$ are negligible and thus the convergence of the component sizes in Theorem~\ref{thm::conv:component:size} follows.

\subsection{Tightness of the component sizes}
The following proposition establishes a uniform tail summability condition that is required for the tightness of the (scaled) ordered vector of component size with respect to the $\ell^2_{\shortarrow}$ topology: 
\begin{proposition}\label{prop-l2-tightness} For any $\varepsilon >0$,
\begin{equation}
 \lim_{K\to\infty}\limsup_{n\to\infty}\PR\bigg(\sum_{i>K}|\mathscr{C}_{\sss(i)}|^2>\varepsilon b_n^2\bigg)=0.
\end{equation}
\end{proposition}
Roughly speaking, the proof is based on the fact that the graph, obtained by removing a large number of high-degree vertices, yields a graph that approaches subcriticality. More precisely,  we prove Lemma~\ref{lem::tail_sum_squares} below to complete the proof of Proposition~\ref{prop-l2-tightness}. This fact is not true for the finite third-moment setting  \cite{DHLS15}. 
However, since the large-degree vertices guide the scaling behavior in the infinite third-moment case, the observation in Lemma~\ref{lem::tail_sum_squares} saves some computational complexity, and gives a different proof of the $\ell^2_{\shortarrow}$ tightness than the aproach with size-biased point processes originally proposed in \cite{A97}. 
\begin{lemma} \label{lem::tail_sum_squares} Consider $\mathrm{CM}_n(\boldsymbol{d})$ satisfying {\rm Assumption~\ref{assumption1}}.
Let $\mathcal{G}^{\sss[K]}$ be the random graph obtained by removing all edges attached to vertices $1,\dots,K$ and let $\boldsymbol{d}'$ be the obtained degree sequence. Suppose $V_n$ is a random vertex of $\mathcal{G}^{\sss[K]}$ chosen independently of the graph and let $\mathscr{C}^{\sss[K]}(V_n)$ be the corresponding component. Let $\{\mathscr{C}_{\sss (i)}^{\sss [K]}:i\geq 1 \}$ be the components of $\mathcal{G}^{\sss [K]}$, ordered according to their sizes. Then,
\begin{equation}\label{expt-cvn-K-removed}
\lim_{K\to\infty} \limsup_{n\to\infty} c_n^{-1}\expt{|\mathscr{C}^{\sss [K]}(V_n)|}=0.
\end{equation}Consequently, for any $\varepsilon > 0$,
\begin{equation}\label{tail-sum-squares-K}
 \lim_{K\to\infty}\limsup_{n\to\infty} \PR\bigg(\sum_{i\geq 1}\big|\mathscr{C}_{\sss (i)}^{\sss [K]}\big|^2> \varepsilon b_n^2\bigg)=0.
\end{equation}
\end{lemma}

\begin{proof}
We make use of a result due to \citet{J09b} regarding bounds on the susceptibility functions for the configuration model. 
In fact, \cite[Lemma 5.2]{J09b} shows that, for any configuration model $\mathrm{CM}_n(\boldsymbol{d})$ with $\nu_n<1$, 
 \begin{equation}\label{bound::expt-cluster-size}
  \expt{|\mathscr{C}(V_n)|} \leq 1+\frac{\expt{D_n}}{1-\nu_n}.
 \end{equation}  Now, conditional on the set of removed half-edges, $\mathcal{G}^{\sss [K]}$ is still a configuration model with some degree sequence $\boldsymbol{d}'$ with $d_i'\leq d_i$ for all $i\in [n]\setminus [K]$ and $d_i'=0$ for $i\in [K]$. Further, the criticality parameter $\nu^{\sss [K]}_n$ of $\mathcal{G}^{\sss [K]}$ satisfies 
 \begin{equation}\label{eqn:nu-K}
  \begin{split}
   \nu^{\sss [K]}_n&= \frac{\sum_{i\in [n]} d_i'(d'_i-1)}{\sum_{i\in [n]} d_i'}\leq \frac{\sum_{i\in [n]}d_i(d_i-1)-\sum_{i=1}^Kd_i(d_i-1)}{\ell_n-2\sum_{i=1}^Kd_i}\\
   &=\nu_n-C_1n^{2\alpha -1}L(n)^2\sum_{i\leq K}\theta_i^2=\nu_n-C_1c_n^{-1}\sum_{i\leq K}\theta_i^2
  \end{split}
 \end{equation}for some constant $C_1>0$.
 Since $\boldsymbol{\theta}\notin \ell^2_{\shortarrow}$, $K$ can be chosen large enough such that $\nu^{\sss[K]}_n < 1$ uniformly for all $n$. Also $\sum_{i\in [n]}d'_i=\ell_n+o(n)$ for each fixed $K$. Let $\mathbbm{E}_K[\cdot]$ denote the conditional expectation, conditioned on the set of removed half-edges. Using \eqref{bound::expt-cluster-size} on $\mathcal{G}^{\sss [K]}$, we get
 \begin{equation}
 \begin{split}
  \mathbbm{E}_K\big[|\mathscr{C}^{\sss[K]}(V_n)|\big]&\leq \frac{C_2}{1-\nu^{\sss[K]}_n}\leq \frac{C_2}{1-\nu_n+C_1c_n^{-1}\sum_{i\leq K}\theta_i^2}\leq  \frac{C_2c_n}{-\lambda+C_1\sum_{i\leq K}\theta_i^2},
 \end{split}
 \end{equation}for some constant $C_2>0$. Using the fact that $\boldsymbol{\theta}\notin \ell^2_{\shortarrow}$, this concludes the proof of \eqref{expt-cvn-K-removed}. The proof of \eqref{tail-sum-squares-K} follows from \eqref{expt-cvn-K-removed} by using the Markov inequality and the observation that
 \begin{equation} \label{comp-vs-rnd-choice}
 \mathbbm{E}\bigg[\sum_{i\geq 1}|\mathscr{C}_{\sss (i)}^{\sss [K]}|^2\bigg]= n\expt{|\mathscr{C}^{\sss [K]}(V_n)|},
 \end{equation}as well as $b_n^2 = nc_n$ by \eqref{eqn:notation-power*}.
\end{proof} 
\begin{proof}[Proof of Proposition~\ref{prop-l2-tightness}] Denote the sum of squares of the component sizes excluding the components containing vertices $1,2,\dots, K$ by $\mathscr{S}_K$.
 Note that \begin{equation}
 \sum_{i>K}|\mathscr{C}_{\sss (i)}|^2\leq \mathscr{S}_K\leq \sum_{i\geq 1} |\mathscr{C}_{\sss(i)}^{\sss [K]}|^2. 
 \end{equation}  Thus, Proposition~\ref{prop-l2-tightness} follows from Lemma~\ref{lem::tail_sum_squares}.
\end{proof}

\subsection{Large components are explored early}
As remarked at the beginning of Section~\ref{sec:conv-comp-size}, an important consequence of Proposition~\ref{prop-l2-tightness} is that after time $\Theta(b_n)$, Algorithm~\ref{algo-expl} does not explore large components. The precise statement needed to complete our proof is given below. This is an essential ingredient to conclude the convergence of the component sizes from the convergence of the exploration process since Theorem~\ref{thm::convegence::exploration_process} only gives information about the components explored on the time scale of the order $b_n$.
\begin{lemma}\label{lem:no-large-comp-later} Let $\mathscr{C}_{\max}^{\sss \geq T}$ be the largest among those components which are started exploring after time $Tb_n$ by {\rm Algorithm~\ref{algo-expl}}. Then, for any $\varepsilon >0$,
\begin{equation}
 \lim_{T\to\infty}\limsup_{n\to\infty}\prob{|\mathscr{C}_{\max}^{\sss \geq T}|>\varepsilon b_n}=0.
\end{equation}
\end{lemma}

\begin{proof}
 Define the event $\mathscr{A}_{\sss K,T}^n:= \{\text{all the vertices of }[K] \text{ are explored before time }Tb_n\}$. Recall the definition of $\mathscr{C}_{\sss (i)}^{\sss [K]}$ from Lemma~\ref{lem::tail_sum_squares}. Firstly, note that 
 \begin{equation}\label{eq:CgeqT1}
  \prob{|\mathscr{C}_{\max}^{\sss \geq T}|>\varepsilon b_n, \ \mathscr{A}_{\sss K,T}^n}\leq \PR\bigg(\sum_{i\geq 1}\big|\mathscr{C}_{\sss (i)}^{\sss [K]}\big|^2> \varepsilon^2 b_n^2\bigg).
 \end{equation}Moreover, using \eqref{eqn:increment:indicator} and the fact that $d_jb_n=\Theta(n)$, we get 
 \begin{equation}\label{eq:CgeqT2}
  \begin{split} \prob{(\mathscr{A}_{\sss K,T}^{n})^c}&=\prob{\exists j\in [K]: j \text{ is not explored before }Tb_n}\\
  &\leq \sum_{j=1}^K \prob{j \text{ is not explored before }Tb_n}\\
  &\leq \sum_{j=1}^K\left(1-\frac{d_j}{\ell_n-\Theta(Tb_n)} \right)^{Tb_n}\leq \sum_{j=1}^K \e^{-CT},
  \end{split}
 \end{equation} where $C>0$ is a constant that may depend on $K$. Now, by \eqref{eq:CgeqT1},
 \begin{equation}
  \prob{|\mathscr{C}_{\max}^{\sss \geq T}|>\varepsilon b_n}\leq \PR\bigg(\sum_{i\geq 1}\big|\mathscr{C}_{\sss (i)}^{\sss [K]}\big|^2> \varepsilon^2 b_n^2\bigg) + \prob{(\mathscr{A}_{\sss K,T}^{n})^c}.
 \end{equation} The proof follows by taking $\limsup_{n\to\infty}$, $\lim_{T\to\infty}$, $\lim_{K\to\infty}$ respectively and using  \eqref{tail-sum-squares-K}, \eqref{eq:CgeqT2}. 
\end{proof}

\subsection{Sample path properties} Recall the definition of an excursion from \eqref{def:excursion}. Define the set of excursions of a function $f$ by
\begin{equation}
\mathcal{E}:= \{(l,r): (l,r) \text{ is an excursion of }f\}.
\end{equation}  
We also denote  the set of excursion end-points by $\mathcal{Y}$, i.e.,
\begin{equation}
\mathcal{Y}:= \{r>0: (l,r)\in \mathcal{E}\}.
\end{equation}
\begin{defn}\label{defn::good_function}\normalfont A  function $f\in \mathbb{D}_+[0,T]$ is said to be \emph{good} if the following holds:
\begin{enumerate}[(a)]
\item  $\mathcal{Y}$ does not have an isolated point and the complement of $\cup_{(l,r)\in \mathcal{E}}(l,r)$ has Lebesgue measure zero;
\item $f$ does not attain a local minimum at any point of $\mathcal{Y}$.
\end{enumerate} 
\end{defn}

\begin{remark}\label{rem:cont-at-r} \normalfont We claim that if a function $f\in \mathbb{D}_+[0,T]$ is good, then $f$ is continuous on $\mathcal{Y}$. To see this, fix  any $\delta>0$ and denote the set of excursions of length at least $\delta$ by $\mathcal{E}_\delta$. Let $r$ be the excursion endpoint of an excursion in $\mathcal{E}_\delta$ and suppose that $f(r)>f(r-)$. Thus, there is no excursion endpoint in $(r-\delta,r)$. Moreover, since $f$ is right-continuous, there exists $\delta '>0 $ such that $f(x)>f(r-)+\varepsilon$ for all $x\in (r,r+\delta')$, where $\varepsilon = (f(r)-f(r-))/2>0$. Thus there is no excursion endpoint on $(r-\delta,r+\delta')$ and thus $r$ is an isolated point contradicting Definition~\ref{defn::good_function}. We conclude that $f$ is continuous at excursion endpoints of the excursions in $\mathcal{E}_{\delta}$, and since $\delta>0$ is arbitrary the claim is established.
\end{remark}
Let $\mathcal{L}_i(f)$ be the length of the $i^{th}$ largest excursion of $f$ and define $\Phi_m:\mathbb{D}_+[0,T]\to \mathbbm{R}^m$ by 
 \begin{equation}
 \Phi_m(f)= (\mathcal{L}_1(f), \mathcal{L}_2(f), \dots, \mathcal{L}_m(f)).
 \end{equation}Note that $\Phi_m(\cdot)$ is well-defined for any good function  defined in  Definition~\ref{defn:good_function-infty}.  
 \begin{lemma}\label{lem::good:function:continuity} Suppose that $f\in \mathbb{D}_+[0,T]$ is good. Then, $\Phi_m$ is continuous at $f$ with respect to the subspace topology on $\mathbb{D}_+[0,T]$ induced by the Skorohod $J_1$-topology.
 \end{lemma}
\begin{proof}
 We extend the arguments of \cite[Proposition 22]{NP10b}. The proof  here is for $m=1$ and  similar arguments hold for $m>1$.  Let $\mathfrak{L}$ denote the set of continuous functions $\Lambda:\mathbbm{R}_+\to\mathbbm{R}_+$  that are strictly increasing and $\Lambda(0)=0, \Lambda(T)=T$.   Suppose $E_1=(l,r)$ is the longest excursion of $f$ on $[0,T]$, thus $\Phi_1(f)=r-l$. For any $\varepsilon > 0$ (small), choose $\delta >0$ such that 
 \begin{equation}\label{f-large-interval}
 f(x)> \min\{f(r-),f(r)\}+\delta\quad \forall x\in (l+\varepsilon,r-\varepsilon).
 \end{equation}
 Let $||\cdot||$ denote the sup-norm on $[0,T]$.  Take any sequence of functions $f_n\in \mathbb{D}_+[0,T]$ such that $f_n\to f$, i.e., there exists $\{\Lambda_n\}_{n\geq 1} \subset \mathfrak{L}$ such that for all large enough $n$, 
 \begin{equation}\label{f_n-f-close}
 ||f_n\circ \Lambda_n-f||< \frac{\delta}{6}\  \text{ and }\  ||\Lambda_n-I||< \varepsilon,
 \end{equation}where $I$ is the identity function. 
 Now, by Remark~\ref{rem:cont-at-r}, $f$ is continuous at $r$. This implies that $f(r-)=f(r)$, and using \eqref{f-large-interval} and \eqref{f_n-f-close}, for all large enough $n$,
 \begin{equation} \label{liminf_f_excursion}
  f_n(y)> f_n\circ \Lambda_n(r)+\frac{2\delta}{3}\quad \forall y \in (l+2\varepsilon, r-2\varepsilon).
 \end{equation}Further, using the continuity of $f$ at $r$, $f_n(r)\to f(r)$ and thus, for all sufficiently large $n$,
 \begin{equation}
  |f_n\circ\Lambda_n(r)-f_n(r)|\leq |f_n\circ\Lambda_n(r)-f(r)|+|f_n(r)-f(r)|< \frac{\delta}{3}.
 \end{equation} Hence, \eqref{liminf_f_excursion} implies that, for all sufficiently large $n$,
 \begin{equation}
  f_n(y)> f_n(r)+\frac{\delta}{3}\quad \forall y\in (l+2\varepsilon, r-2\varepsilon).
 \end{equation} Thus, for any $\varepsilon>0$, we have
\begin{equation}\label{eq:liminf-exc}
\liminf_{n\to\infty}\Phi_1(f_n)\geq r-l-4\varepsilon= \Phi_1(f)-4\varepsilon.
\end{equation}
Now we turn to a suitable upper bound on $\limsup_{n\to\infty}\Phi_1(f_n)$. First, we claim that one can find $r_1,\dots,r_k\in \mathcal{Y}$ such that $r_1\leq \Phi_1(f)+\varepsilon, T-r_k< \Phi_1(f)+\varepsilon, $ and $r_i-r_{i-1}\leq \Phi_1(f)+\varepsilon, \forall i=2, \dots,k$. The claim is a consequence of Definition~\ref{defn::good_function}~(a). 
 Now, Definition \ref{defn::good_function}~(b) implies that for any small $\varepsilon >0$, there exists $\delta >0$ and $x_i\in (r_i,r_i+\varepsilon)$ such that $f(r_i)-f(x_i)> \delta$ $\forall i$. Again, since $r_i$ is a continuity point of $f$, $f_n(r_i)\to f(r_i)$. Thus, using \eqref{f_n-f-close}, for all large enough $n$,
 \begin{equation}
  f_n(r_i)-f_n(\Lambda_n(x_i))> \frac{\delta}{2}.
 \end{equation}
 Now, $\Lambda_n(x_i)\in (r_i,r_i+2\varepsilon)$ for all sufficiently large $n$, since $x_i\in (r_i,r_i+\varepsilon)$. Thus, for all large enough $n$, there exists a point $z_i^n\in (r_i,r_i+2\varepsilon)$ such that
 \begin{equation}
  f_n(r_i)-f_n(z_i^n)> \frac{\delta}{2}.
 \end{equation}Also the function $f_n$ only has positive jumps and $\ubar{f}_n(r_i)\to \ubar{f}(r_i)$, as $\ubar{f}_n$ is continuous, where we recall that $\ubar{f}(x)=\inf_{y\leq x}f(y)$. Therefore, $f_n$ must have an excursion ending point in $(r_i,r_i+2\varepsilon)$ for all large enough $n$. Also, using the fact that the complement of $\cup_{(l,r)\in \mathcal{E}}(l,r)$ has Lebesgue measure zero, $f$ has an excursion endpoint  $r_i^0\in(l_i-\varepsilon,l_i)$. The previous argument shows that $f_n$ has to have an excursion endpoint in $(r_i^0, r_i^0+2\varepsilon)$ and thus in $(l_i-\varepsilon,l_i+2\varepsilon)$, for all large $n$. Therefore, for any $\varepsilon >0$,
 \begin{equation}\label{eq:limsup-exc}
 \limsup_{n\to\infty} \Phi_1(f_n)\leq \Phi_1(f)+3\varepsilon.
\end{equation}  Hence the proof follows from  \eqref{eq:liminf-exc} and \eqref{eq:limsup-exc}.
\end{proof}

\begin{remark}\label{rem:excursion-area-cont} \normalfont For $f\in\mathbb{D}_{+}[0,T]$, let $\mathcal{A}_i(f)$ denote the area under the excursion $\mathcal{L}_i(f)$. Let $(f_n)_{n\geq 1}$ be a sequence of functions on $f\in\mathbb{D}_{+}[0,\infty)$ such that $f_n\to f$, with respect to the Skorohod $J_1$-topology, where $f$ is good. 
Then, \eqref{f_n-f-close}, \eqref{eq:liminf-exc} and \eqref{eq:limsup-exc} also implies that $(\mathcal{A}_1(f_n),\dots,\mathcal{A}_m(f_n))$ converges to $(\mathcal{A}_1(f),\dots,\mathcal{A}_m(f))$, for any $m\geq 1$.
\end{remark}

\begin{defn}\label{defn:good_function-infty}\normalfont A stochastic process $\mathbf{X}\in \mathbb{D}_+[0,\infty)$ is said to be good if 
\begin{enumerate}[(a)]
\item the sample paths are good almost surely when restricted to $[0,T]$, for every fixed $T>0$;
\item $\mathbf{X}$ does not have an infinite excursion almost surely;
\item for any $\varepsilon >0$, $\mathbf{X}$ has only finitely many excursions of length more than $\varepsilon$ almost surely.
\end{enumerate} 
\end{defn}
\begin{lemma}\label{lem:levy-as-good}The thinned L\'evy process $\bar{\mathbf{S}}_\infty^\lambda$ defined in \eqref{defn::limiting::process} is good.
\end{lemma}
\begin{proof}
 Let us make use of the properties of the process $\bar{\mathbf{S}}_\infty^\lambda$ that were established in \citep{AL98}. $\bar{\mathbf{S}}_\infty^\lambda$ satisfies Definition~\ref{defn:good_function-infty}~(b),(c) by \cite[(8)]{AL98}.   The fact that the excursion endpoints of $\bar{\mathbf{S}}_\infty^\lambda$ do not have any isolated points almost surely follows directly from \citep[Proposition 14 (d)]{AL98}. Further,  \citep[Proposition 14 (b)]{AL98} implies that, for any $u>0$, $\prob{\bar{S}_\infty^\lambda(u)=\inf_{u'\leq u}\bar{S}_\infty^\lambda(u')}=0$. Taking the integral with respect to the Lebesgue measure and interchanging the limit by using Fubini's theorem, we conclude that almost surely 
 \begin{equation}
  \int_0^T \ind{\bar{S}_\infty^\lambda(u)=\inf_{u'\leq u}\bar{S}_\infty^\lambda(u')}\mathrm{d}u=0,
 \end{equation}which verifies Definition~\ref{defn::good_function}~(a). 
 Next, in order to verify Definition~\ref{defn::good_function}~(b), let $r\in \mathcal{R}$ be any excursion end-point of $\biS$. 
It is enough to show that 
\begin{eq}\label{eq:non-perfect-levy}
\inf \{t>0: \iS(r+t) - \iS(r)<0\}=0, \quad \text{almost surely.}
\end{eq}
Let $V_r = \{i: \mathcal{I}_i (r) = 1\}$. 
Thus conditional on the sigma-field $\sigma (\bar{S}_\infty^\lambda(s):s\leq r)$, the process $(\iS(r+t) - \iS(r))_{t\geq 0}$ is distributed as $\hat{\mathbf{S}}_{\infty}^\lambda$ given by  
\begin{eq}
\hat{S}_{\infty}^\lambda(t) =  \sum_{i\notin V_r} \theta_i\left(\mathcal{I}_i(t)- (\theta_i/\mu)t\right)+\lambda t.
\end{eq}
 Now, let $\mathbf{L}$ be the L\'evy process defined as
 \begin{equation}\label{defn:perfect-levy}
  L(t)=\sum_{i\notin V_r} \theta_i\left(\mathcal{N}_i(t)- (\theta_i/\mu)t\right)+\lambda t,
 \end{equation}where $(\mathcal{N}_i(t))_{t\geq 0}$ is a Poisson process with rate $\theta_i$ which are independent for different $i$. Via the natural coupling that states $\mathcal{I}_i(t)\leq \mathcal{N}_i(t)$, we can assume that $\hat{S}_\infty^\lambda(t)\leq L(t)$ for all $t>0$. Using \cite[Theorem VII.1]{Ber96},
 \begin{equation}\label{defn:perfect-levy-1}
  \inf \{t>0: L(t)<0\}=0, \quad \text{almost surely,}
 \end{equation}
and thus, \eqref{eq:non-perfect-levy} follows, and the proof is complete. 
\end{proof}

\subsection{Finite-dimensional convergence}
As described in Section~\ref{sec:conv-expl}, the excursion lengths of the exploration process $\bar{\mathbf{S}}_n$ give the total number of edges in the explored components. Lemma~\ref{lem:surp:poisson-conv} below estimates the number of surplus edges in the components explored upto time $\Theta(b_n)$. This enables us to compute the scaling limits for the component sizes using the results from the previous section and complete the proof of Theorem~\ref{thm::conv:component:size}.

\begin{lemma} \label{lem:surp:poisson-conv} Let $N_n^\lambda(k)$ be the number of surplus edges discovered up to time $k$ and $\bar{N}^\lambda_n(u) = N_n^\lambda(\lfloor ub_n \rfloor)$. Then, as $n\to\infty$,
 \begin{equation}\label{eq:limit-joint-comp-sp}
 (\bar{\mathbf{S}}_n,\bar{\mathbf{N}}_n^\lambda)\dto (\bar{\mathbf{S}}_{\infty}^\lambda,\mathbf{N}),
 \end{equation} where $\mathbf{N}$ is the counting process  defined in \eqref{defn::counting-process}.
 \end{lemma}

\begin{proof}
 We write $
N_n^{\lambda}(l)=\sum_{i=2}^l \xi_i$,
where $\xi_i=\ind{\mathscr{V}_i=\mathscr{V}_{i-1}}$. Let $A_i$ denote the number of active half-edges after stage $i$ while implementing Algorithm~\ref{algo-expl}. Note that 
\begin{equation}\label{eq:increment-surplus-prob}
 \prob{\xi_i=1\vert \mathscr{F}_{i-1}}=\frac{A_{i-1}-1}{\ell_n-2i-1}= \frac{A_{i-1}}{\ell_n}(1+O(i/n))+O(n^{-1}),
\end{equation} uniformly for $i\leq Tb_n$ for any $T>0$. 
Therefore, the instantaneous rate of change of the re-scaled process $\bar{\mathbf{N}}_n^{\lambda}$ at time $t$, conditional on the past, is 
\begin{equation}\label{eqn:intensity}
 b_n\frac{A_{\floor{tb_n}}}{n\mu}\left( 1+o(1)\right) +o(1)= \frac{1}{\mu}\refl{\bar{S}_n(t)}\left( 1+o(1)\right) +o(1).
\end{equation} 
We first argue that, for any fixed $u>0$, 
\begin{equation}\label{1d-tight-surplus}
\big(\bar{N}_n^\lambda (u)\big)_{n\geq 1} \text{ is tight in }\R_+.
\end{equation}
Fix $\varepsilon>0$. 
Using Theorem~\ref{thm:conv:refl:process}, and the fact that the supremum of a process is continuous with respect to the Skorohod $J_1$-topology \cite[Theorem 13.4.1]{W02}, we can choose $K\geq 1$ large enough so that 
\begin{eq}\label{eq:tightness-supremum}
\limsup_{n\to\infty} \PR\Big(\sup_{i\leq \floor{ub_n}} A_{i}>K a_n\Big) <\varepsilon.
\end{eq}
Fix times $0<l_1<\dots<l_m\leq \floor{ub_n}$, and let $\cA(l_1,\dots,l_m)$
denote the event that the surplus edges appear at times $l_1,\dots,l_m$ and $A_{l_{j}-1}\leq K a_n$ for all $j \in [m]$. Then,
\begin{eq}\label{eq:ub-surplus-restricted-suprema}
&\PR\bigg(\sum_{i=2}^{\floor{ub_n}}\xi_i \geq m, \text{ and }\sup_{i\leq \floor{ub_n}} A_{i}\leq K a_n\bigg) \leq \sum_{0<l_1<\dots<l_m\leq \floor{ub_n}} \PR(\cA(l_1,\dots,l_m)) \\
&\leq \sum_{0<l_1<\dots<l_m\leq \floor{ub_n}} \E \big[\PR(\text{surplus created at }l_m\vert \sF_{l_m-1})\ind{A_{l_m - 1} \leq K a_n} \mathbbm{1}_{\cA(l_1,\dots,l_{m-1})}\big] 
\\
&\leq \frac{Ka_n}{\ell_n -2\floor{ub_n} +1} \sum_{0<l_1<\dots<l_m\leq \floor{ub_n}} \PR(\cA(l_1,\dots,l_{m-1})).
\end{eq}
Continuing the iteration in the last step, it follows that
\begin{eq}\label{ub-tightness-sup-bounded}
\PR\bigg(\sum_{i=2}^{\floor{ub_n}}\xi_i \geq m, \text{ and }\sup_{i\leq \floor{ub_n}} A_{i}\leq K a_n\bigg)  \leq (1+o(1))\Big(\frac{K a_n}{\ell_n}\Big)^m \frac{\floor{ub_n}\dots (\floor{ub_n} - m+1)}{m!},
\end{eq}
which tends to zero in the iterated limit $\lim_{m\to\infty}\limsup_{n\to\infty}$. 
An application of \eqref{eq:tightness-supremum} now yields~\eqref{1d-tight-surplus}.

Next, let $\mathbf{S}_n'$ be the process obtained by discarding the points where a surplus edge was added. More precisely, if $\zeta_ l = S_n(l) - S_n(l-1)$, then  we can define $S_n'(l) = S_n'(l-1) +\zeta'_{l}$, where 
\begin{equation}\label{eq:defn-zetal}
\zeta'_l = \zeta_{k_l}, \quad \text{with} \quad k_l = \inf\{j>k_{l-1}: \zeta_j \neq -2\}, \ k_0 =0.
\end{equation} 
Let $\bar{S}_n'(t) = a_n^{-1} S_n'(\floor{tb_n})$. 
Also, let $d_{J_1, T}$ denote the metric for the  Skorohod $J_1$-topology on $\mathbb{D}([0,T], \R)$.
We claim that, for any $T>0$ and $\varepsilon>0$,
\begin{equation}\label{eq:Sn-Snprime-close}
\lim_{n\to\infty}\PR\big( d_{J_1,T} (\bar{\mathbf{S}}_n',\bar{\mathbf{S}}_n) > \varepsilon\big) = 0.
\end{equation}
First, let $1\leq l_1<\dots<l_K\leq \floor{Tb_n}$ denote the times where the  surplus edges have occurred. 
Also, let $\cA$ be the good event that $l_{j}+1<l_{j+1}$ for all $j\leq K$, i.e., none of the surplus edges occur in consecutive steps.
Note that
\begin{eq}\label{eq:good-event-surplus}
&\PR\Big(\cA^c\bigcap \Big\{\sup_{i\leq \floor{Tb_n}} A_{i}\leq K a_n\Big\} \Big) \leq Tb_n \Big(\frac{Ka_n}{\ell_n}\Big)^2 = O(b_n^{-1}),
\end{eq}and thus using \eqref{eq:tightness-supremum}, $\PR(\cA^c) \to 0$. 
We now restrict ourselves to $\cA$. 
Putting $l_0 =0$ and $l_{K+1} = \floor{Tb_n}+1$, let 
\begin{eq}\label{eq:time-change-definition}
\Lambda_n(l) = 
\begin{cases}
l+j-1 \quad &\text{for } l_{j-1}<l<l_j\\
l_j+j-1 \quad &\text{for } l = l_j - 0.5\\
l_j+j \quad &\text{for } l = l_j.
\end{cases}
\end{eq}
$\Lambda_n(t)$ is obtained by linearly interpolating between the values specified by \eqref{eq:time-change-definition}. 
Also, note that the definition of $\Lambda_n$ works well on $\cA$, and on $\cA^c$, we define $\Lambda_n(t) =t$.
Using \eqref{1d-tight-surplus} and  \eqref{eq:good-event-surplus} it immediately follows that 
\begin{eq}\label{eq:modified-J1-1}
\sup_{l\leq Tb_n} |\Lambda_n(l) - l | = \oP(b_n). 
\end{eq}
Moreover, the occurrence of each surplus edge causes $|S_n'(l) - S_n(\Lambda_n(l))|$ to increase by at most 2, so that  
\begin{eq}\label{eq:modified-J1-2}
\sup_{l\leq Tb_n} |S_n'(l) - S_n(\Lambda_n(l))| = \oP(a_n). 
\end{eq}
Then, \eqref{eq:Sn-Snprime-close} follows by combining \eqref{eq:modified-J1-1} and \eqref{eq:modified-J1-2}. 
We now proceed to complete the proof of Lemma~\ref{lem:surp:poisson-conv}.
Let us start by setting up some notation for the rest of the proof. Fix $T>0$, $k\geq 0$ and let $\mathrm{Surp}_T = \{l_1,\dots ,l_k\}$ be the surplus generation times, where $1\leq l_1 <l_2<\dots<l_k \leq \floor{Tb_n}+k$. Let $(z_l)_{l\leq \floor{Tb_n}+k}$ be a sequence of integers such that $z_{l_i} = -2$ and  $z_l\geq -1$ for $l\notin \{l_1,\dots,l_k\}$. 
Thus $(z_l)_{l\leq \floor{Tb_n}+k }$ represents the increments of a sample path of $S_n$ which has explored $k$ surplus edges, and $\mathrm{Surp}_T $ is the set of times when surplus edges are found. 
Next, $(z_l')_{l\leq \floor{Tb_n}}$ denote the sequence obtained from $(z_l)_{l\leq \floor{Tb_n}+k}$ by deleting the $-2$'s. 
Thus, $(z_l')_{l\leq \floor{Tb_n}}$ corresponds to the increments of a sample path of $S_n'$. 
Recall that $\zeta_l = S_n(l) - S_n(l-1)$. 
Thus, 
\begin{eq}\label{sample-path-prob-conditional}
&\PR\bigg(N_n^{\lambda}(\floor{Tb_n}+k) = k \ \Big\vert\   (S_n'(l))_{l\leq \floor{Tb_n}}= \Big(\sum_{j\leq l}z_j'\Big)_{l\leq \floor{Tb_n}}, N_n^{\lambda}(\floor{Tb_n}+k)  \leq \omega_n\bigg) \\
& = \sum_{1\leq l_1<\dots<l_k\leq Tb_n+k}\PR\bigg(\text{surplus occurs only at times }l_1, \dots,l_k\ \bigg\vert\  \substack{\big(S_n'(l)\big)_{l\leq Tb_n} = \big(\sum_{j\leq l}z_j'\big)_{l\leq \floor{Tb_n}},\\ N_n^\lambda(\floor{T b_n}+k)\leq \omega_n}\bigg)\\
& = \sum_{1\leq l_1<\dots<l_k\leq Tb_n+k} \frac{\PR(\zeta_l = z_{l}, \ \text{ for all } 1\leq l\leq \floor{Tb_n}+k)}{\PR(\big(S_n'(l)\big)_{l\leq Tb_n} = \big(\sum_{j\leq l}z_j'\big)_{l\leq \floor{Tb_n}}, N_n^\lambda(\floor{T b_n}+k)\leq \omega_n)}
\end{eq}
Define $m_1= \#\{i\in [n]:d_i = z_1+2\}$, and for $l\notin\mathrm{Surp}_T$, we denote 
$m_l = \#\{i\in [n]:d_i = z_l+2\} - \#\{j< l:z_j = z_l\}.$
Thus, $m_l$ gives the number of degree $z_l+2$ vertices in the system at time $l$, which are potential candidates to cause a jump of $z_l$ at time $l$.
Next, let $a_l$ denote the number of active half-edges at time $l$ when the exploration process takes the path $(z_l)_{l\leq \floor{Tb_n}+k}$, and $a_l' = S_n'(l)-\min_{j<l} S_n'(j)$. 
Now, 
\begin{eq}\label{exact-prob-sample-path}
&\PR(\zeta_l = z_l, \ \forall l\leq \floor{Tb_n}+k) = \frac{\prod_{l\notin \mathrm{Surp}_T} m_l \times \prod_{j=1}^k(a_{l_j-1}-1)}{(\ell_n - 1)(\ell_n-3)\cdots (\ell_n - 2\floor{Tb_n}-2k+1)} \\
& = \frac{\prod_{l\notin \mathrm{Surp}_T} m_l \times \prod_{j=1}^k(a_{l_j-1}-1)}{(\ell_n - 1)\cdots (\ell_n - 2\floor{Tb_n}+1)} \times  (1+o(1))\prod_{j=1}^k\frac{a_{l_j-1}'}{\ell_n},
\end{eq}where the $o(1)$ term above is uniform in  $k\leq \omega_n = \log n$. 
Thus, 
\begin{eq}\label{sample-path-prob-conditional-2}
\eqref{sample-path-prob-conditional}& =(1+o(1))\frac{\sum_{1\leq l_1<\dots<l_k\leq \floor{Tb_n}+k}\prod_{j=1}^k\frac{a_{l_j-1}'}{\ell_n^k} }{\sum_{r=0}^{\omega_n}\sum_{1\leq l_1<\dots<l_r\leq  \floor{Tb_n}+r}\prod_{j=1}^r\frac{a_{l_j-1}'}{\ell_n^r}} =: (1+o(1)) \frac{\beta_{n,k}}{\sum_{r=0}^\infty \beta_{n,r}},
\end{eq}where $\beta_{n,r} = 0$ for $r>\omega_n$. 
Now, Theorem~\ref{thm::convegence::exploration_process} together with \eqref{eq:Sn-Snprime-close} also implies that 
$\bar{\mathbf{S}}_n' \xrightarrow{\sss d} \mathbf{S}^\lambda_\infty$ with respect to the Skorohod $J_1$-topology. 
Since the reflection of a process is continuous in Skorohod $J_1$-topology (see \cite[Lemma 13.5.1]{W02}) it also follows that $\mathrm{refl}(\bar{\mathbf{S}}_n')  \xrightarrow{\sss d} \mathrm{refl} (\mathbf{S}^\lambda_\infty )$. 
Thus,
\begin{eq}\label{eq:beta-S-joint-convergence}
\Big((\beta_{n,r})_{r\geq 0}, (\bar{S}_n'(u))_{u\leq T} \Big) \dto \bigg(\Big(\frac{1}{r!}\Big(\frac{1}{\mu} \int_0^T \refl{\bar{S}_\infty^\lambda(u)}\dif u\Big)^r\Big)_{r\geq 0}, (\bar{S}_\infty^\lambda(u))_{u\leq T} \bigg),
\end{eq}
where the convergence of $(\beta_{n,r})_{r\geq 0}$ holds with respect to the product topology on $\R^\infty$.
Next, let us ensure that $\sum_{r=0}^\infty \beta_{n,r}$ in \eqref{sample-path-prob-conditional} converges to the desired quantity. 
To this end, consider a probability space where the convergence of \eqref{eq:beta-S-joint-convergence} holds almost surely. 
On this space, $\sup_{l\leq Tb_n+k} \refl{S_n'(l)}\leq 2 (\sup_{l\leq Tb_n+k} S_n'(l) +\omega_n)=: X_n(T)$, and thus
\begin{eq}
\beta_{n,r} \leq \frac{(Tb_n+\omega_n)^r}{r!} \frac{X_n(T)^r}{\ell_n^r}. 
\end{eq}
Since $a_n^{-1}\sup_{l\leq Tb_n+k} S_n'(l)$ converges, an application of Dominated Convergence Theorem yields that 
\begin{eq}\label{eq:normalization-surplus-limit}
\sum_{r\geq 0} \beta_{n,r} \asto \sum_{r\geq 0}\frac{1}{r!}\Big(\frac{1}{\mu} \int_0^T \refl{\bar{S}_\infty^\lambda(u)}\dif u\Big)^r = \exp \bigg(\frac{1}{\mu} \int_0^T \refl{\bar{S}_\infty^\lambda(u)}\dif u\bigg).
\end{eq}
Next, for bounded continuous functions $\phi_1: \mathbb{D}([0,T], \R) \to \R$ and $\phi_2: \N \to \R$, 
\begin{eq}
&\E \big[\phi_1\big(\big(\bar{S}_n'(u)\big)_{u\leq T}\big)\phi_2(\bar{N}_n^{\lambda}(T)) \big] \\
&=  o(1) +\E \big[\phi_1\big(\big(\bar{S}_n'(u)\big)_{u\leq T}\big)\phi_2(\bar{N}_n^{\lambda}(T)) \ind{N_n^{\lambda}(\floor{Tb_n} + k)  \leq \omega_n}\big]\\
& = o(1) + \E \bigg[\phi_1\big(\big(\bar{S}_n'(u)\big)_{u\leq T}\big)\ind{N_n^{\lambda}(\floor{Tb_n} + k)   \leq \omega_n}\times (1+o(1)) \frac{\sum_{k\geq 0}\phi_2(k) \beta_{n,k}}{\sum_{r\geq 0} \beta_{n,r}}\bigg] \\ 
& = o(1) + \E \bigg[\phi_1\big(\big(\bar{S}_n'(u)\big)_{u\leq T}\big)\times \frac{\sum_{k\geq 0}\phi_2(k) \beta_{n,k}}{\sum_{r\geq 0} \beta_{n,r}}\bigg] \to \E \big[\phi_1\big(\big(\bar{S}_\infty^\lambda(u)\big)_{u\leq T}\big)\phi_2(N(T)) \big],
\end{eq}where $N(T)$, conditionally on $(\bar{S}_\infty^\lambda(u))_{u\leq T}$, is distributed as Poisson$(\frac{1}{\mu}\int_0^T\refl{\bar{S}_\infty^\lambda(u)} \dif u)$.
We have used \eqref{1d-tight-surplus} in the third step, and the final step follows by combining \eqref{eq:beta-S-joint-convergence} and \eqref{eq:normalization-surplus-limit}.
Hence, we have shown that, for any $T>0$, 
\begin{eq}\label{total-limit-law}
\Big(\big(\bar{S}_n'(u)\big)_{u\leq T},\bar{N}_n^{\lambda}(T)\Big)\dto \Big(\big(\bar{S}_\infty^\lambda(u)\big)_{u\leq T},N(T)\Big).
\end{eq}
Next, let $U_1^n<U_2^n<...$ denote the location of surplus edges in the process $S_n$. 
Then, using \eqref{exact-prob-sample-path} yields
\begin{eq}\label{location-law}
&\PR\Big(U_j^n = l_j, \text{ for all }j\in [k]\Big\vert \big(\bar{S}_n'(u)\big)_{u\leq T},\bar{N}_n^{\lambda}(T) = k\Big)\\
& = (1+o(1)) \frac{\frac{1}{\ell_n^k}\prod_{j=1}^k (A_{l_j}-1)}{\sum_{1\leq l_1'<\dots < l_k'\leq \floor{Tb_n}+k}\frac{1}{\ell_n^k}\prod_{j=1}^k (A_{l_j'}-1)}.
\end{eq}
From this, it can be seen that the law of $b_n^{-1} (U_j^n)_{j\in [k]}$, conditionally on $(\bar{S}_n'(u))_{u\leq T}$, and $\bar{N}_n^{\lambda}(T) = k$, converges to the order-statistics of $k$ i.i.d.~random variables with density $\frac{\ind{u\in [0,T]} \refl{\bar{S}_\infty^\lambda (u)}}{\int_0^T\refl{\bar{S}_\infty^\lambda (u)}\dif u}$. 
Now, combining \eqref{total-limit-law}, \eqref{location-law} together with \eqref{eq:Sn-Snprime-close} completes the proof of Lemma~\ref{lem:surp:poisson-conv}.
\end{proof}

\begin{remark}\normalfont In \cite{DHLS15} and an earlier version of the paper, we concluded the proof of Lemma~\ref{lem:surp:poisson-conv} from the convergence of compensators only. 
We thank Lorenzo Federico and Tim Hulshof for pointing out a gap in this argument.
Indeed, for Erd\H{o}s-R\'enyi random graphs \cite{A97} or rank-one inhomogeneous random graphs \cite{BHL10,BHL12}, showing the convergence of compensators suffices using \cite[Theorem 1]{Bro81} since the surplus edges can be added independently after we have observed the whole exploration process. However, this is not true for the configuration model because the surplus edges occur precisely at places with jumps $-2$. 
For this reason, we have modified the proof and the identical argument also completes the proof in \cite[Lemma 5.16]{DHLS15}. 
\end{remark}

\begin{theorem} \label{thm::component-sizes-finite-dim}
 For any $m\geq 1$, as $n\to \infty$,
 \begin{equation}
  b_n^{-1}\big( |\mathscr{C}_{\sss (1)}|, |\mathscr{C}_{\sss (2)}|, \dots, |\mathscr{C}_{\sss(m)}|\big) \dto (\gamma_1(\lambda), \gamma_2(\lambda), \dots, \gamma_m(\lambda))
 \end{equation} with respect to the product topology, where $\gamma_i(\lambda)$ is the $i^{th}$ largest excursion of $\bar{\mathbf{S}}_{\infty}$ defined in \eqref{defn::limiting::process}.
\end{theorem}
\begin{proof}
Fix any $m\geq 1$. Let $\mathscr{C}_{\sss (i)}^{\sss T}$ be the $i^{th}$ largest component explored by Algorithm~\ref{algo-expl} up to time $Tb_n$. Denote by $\mathscr{D}_{\sss (i)}^{\sss \mathrm{ord},T}$ the $i^{th}$ largest value of $(\sum_{k\in \mathscr{C}_{\sss(i)}^{\sss T}}d_k)_{i\geq 1}$. Let $g:\R^m\mapsto \R$ be a bounded continuous function. By Lemma~\ref{lem:levy-as-good}, the sample paths of $\bar{\mathbf{S}}_{\infty}$ are almost surely good. Thus, using Theorem~\ref{thm::convegence::exploration_process}, Lemma~\ref{lem::good:function:continuity} gives
\begin{equation}
 \lim_{n\to\infty} \expt{g\Big((2b_n)^{-1}\big( \mathscr{D}_{\sss(1)}^{\sss\mathrm{ord},T}, \mathscr{D}_{\sss(2)}^{\sss\mathrm{ord},T}, \dots, \mathscr{D}_{\sss(m)}^{\sss\mathrm{ord},T}\big)\Big)} = \expt{g\big(\gamma_1^{\sss T}(\lambda), \gamma_2^{\sss T}(\lambda), \dots, \gamma_m^{\sss T}(\lambda)\big)},
\end{equation}where $\gamma_i^{\sss T}(\lambda)$ is the $i^{th}$ largest excursion of $\bar{\mathbf{S}}_{\infty}$ restricted to the time interval $[0,T]$. Now the support of the joint distribution of $(\gamma_i^{\sss T}(\lambda))_{i\geq 1}$ is concentrated on $\{(x_1,x_2,\dots): x_1>x_2>\dots\}$. Thus, using Lemma~\ref{lem:surp:poisson-conv}, it follows that 
\begin{equation}\label{finite-dim-D}
\lim_{n\to\infty} \expt{g\Big(b_n^{-1}\big( |\mathscr{C}_{\sss (1)}^{\sss T}|, |\mathscr{C}_{\sss (2)}^{\sss T}|, \dots, |\mathscr{C}_{\sss (m)}^{\sss T}|\big)\Big)} = \expt{g\big(\gamma_1^{\sss T}(\lambda), \gamma_2^{\sss T}(\lambda), \dots, \gamma_m^{\sss T}(\lambda)\big)}.
\end{equation}Since $\mathbf{S}_{\infty}^{\lambda}$ satisfies Definition~\ref{defn:good_function-infty}~(b), (c), it follows that
\begin{equation}\label{no-inf-exc}
 \lim_{T\to\infty}\expt{g\big(\gamma_1^{\sss T}(\lambda), \gamma_2^{\sss T}(\lambda), \dots, \gamma_m^{\sss T}(\lambda)\big)}=\expt{g\big(\gamma_1(\lambda), \gamma_2(\lambda), \dots, \gamma_m(\lambda)\big)}
\end{equation}
 Finally, using Lemma~\ref{lem:no-large-comp-later}, the proof of Theorem~\ref{thm::component-sizes-finite-dim} is completed by \eqref{finite-dim-D} and \eqref{no-inf-exc}.

\end{proof}

\begin{proof}[Proof of Theorem~\ref{thm::conv:component:size}]
The proof of Theorem~\ref{thm::conv:component:size}  follows directly from Theorem~\ref{thm::component-sizes-finite-dim} and Proposition~\ref{prop-l2-tightness}.
\end{proof}

\section{Convergence in the $\mathbb{U}^0_{\shortarrow}$ topology}\label{sec:surplus}
 The goal of this section is to prove the joint convergence of the component sizes and the surplus edges as described in Theorem~\ref{thm:spls}. We start with a preparatory lemma:
\begin{lemma}\label{lem:spls:prod} The convergence in \eqref{thm:eqn:spls} holds with respect to the $\ell^2_{\shortarrow}\times \mathbbm{N}^{\infty}$ topology.
\end{lemma}
\begin{proof}
 Note that Lemma~\ref{lem:no-large-comp-later} already states that we do not see large components being explored after the time  $Tb_n$ for large $T>0$. Thus the proof is a consequence of Lemmas~\ref{lem::good:function:continuity},~\ref{lem:surp:poisson-conv}, Remark~\ref{rem:excursion-area-cont} and Theorem~\ref{thm::conv:component:size}.
\end{proof}
Recall the definition of the metric $\mathrm{d}_{\sss \mathbb{U}}$ from \eqref{defn_U_metric}. Using Lemma~\ref{lem:spls:prod}, it now remains to obtain a uniform summability condition on the tail of the sum of products of the scaled component sizes and the surplus edges: 
 \begin{proposition}\label{prop-surp-u-0} For any $\varepsilon >0$,
 \begin{equation}
 \lim_{\delta\to 0}\limsup_{n\to\infty} \PR\bigg(\sum_{i: |\mathscr{C}_{(i)}|\leq \delta b_n }|\mathscr{C}_{\sss(i)}|\times \surp{\mathscr{C}_{\sss(i)}}> \varepsilon b_n\bigg)=0.
 \end{equation}
 \end{proposition} 
\begin{proof}[Proof of Theorem~\ref{thm:spls}]
First $(X_{ni},Y_{ni})_{i\geq 1} \dto  (X_{i},Y_{i})_{i\geq 1} $ in $\mathbb{U}^0_{\shortarrow}$-topology if the following three conditions hold: 
\begin{enumerate}[(i)]
    \item $(X_{ni},Y_{ni})_{i =1}^k \dto  (X_{i},Y_{i})_{i=1}^k$ for all $k\geq 1$.
    \item $\lim_{\delta\to 0}\limsup_{n\to\infty} \PR\big(\sum_{i: X_{ni}\leq \delta  }X_{ni}^2> \varepsilon\big)=0,$ for any $\varepsilon>0$.
    \item $\lim_{\delta\to 0}\limsup_{n\to\infty} \PR\big(\sum_{i: X_{ni}\leq \delta  }X_{ni}Y_{ni}> \varepsilon \big)=0,$ for any $\varepsilon>0$.
\end{enumerate}
To see this, note that $\|(X_{ni},Y_{ni})_{i\geq 1}\|_{\mathbb{U}^0} = \|(X_{ni})_{i\geq 1}\|_{2} + \sum_i X_{ni}Y_{ni}$, where $\|\cdot \|_{\sss \mathbb{U}^0}$ denotes the norm induced by metric in \eqref{defn_U_metric}.
Conditions (i) and (ii) together implies that $(\|(X_{ni})_{i\geq 1}\|_{2})_{n\geq 1}$ is tight, and (i) and (iii) together implies that $(\sum_i X_{ni}Y_{ni})_{n\geq 1}$ is tight.
Thus under (i)--(iii), $((X_{ni},Y_{ni})_{i\geq 1})_{n\geq 1}$ is tight in $\mathbb{U}^0_{\shortarrow}$-topology, and the finite dimensional convergence in (i) yields that $(X_{ni},Y_{ni})_{i\geq 1} \dto  (X_{i},Y_{i})_{i\geq 1} $ in $\mathbb{U}^0_{\shortarrow}$-topology.

To complete the proof of Theorem~\ref{thm:spls}, note that Lemma~\ref{lem:spls:prod} yields conditions (i) and (ii), and Proposition~\ref{prop-surp-u-0} yields condition (iii). Thus the proof follows. 
\end{proof}
 The remainder of this section is devoted to the proof of Proposition~\ref{prop-surp-u-0}.
 The following estimate will be the crucial ingredient to complete the proof of Proposition~\ref{prop-surp-u-0}. The proof of Lemma~\ref{lem:sp-cv-n} is postponed to Appendix~\ref{sec_appendix} since this uses similar ideas as used in~\cite{DHLS15}.
\begin{lemma} \label{lem:sp-cv-n}
Assume that $\limsup_{n\to\infty} c_n(\nu_n-1)<0$. Let $V_n$ denote a vertex chosen uniformly at random, independently of the graph $\mathrm{CM}_n(\boldsymbol{d})$ and let $\mathscr{C}(V_n)$ denote the component containing $V_n$.  Let $\delta_k=\delta k^{-0.12}$. Then, for $\delta > 0$ sufficiently small,
\begin{equation}
 \prob{\surp{\mathscr{C}(V_n)}\geq K,|\mathscr{C}(V_n)|\in (\delta_K b_n,2\delta_Kb_n)}\leq \frac{C\sqrt{\delta}}{a_nK^{1.1}}
\end{equation}
 where $C$ is a fixed constant independent of $n,\delta, K$. 
 \end{lemma}

\begin{proof}[Proof of Proposition~\ref{prop-surp-u-0} using Lemma~\ref{lem:sp-cv-n}]First consider the case $\lambda<0$. Fix any $\varepsilon, \delta >0$. Note that
\begin{equation}\label{eq:u-0-calc}
 \begin{split}
  &\PR\bigg( \sum_{|\mathscr{C}_{\sss (i)}|\leq \delta b_n} |\mathscr{C}_{\sss (i)}|\times \surp{\mathscr{C}_{\sss (i)}}> \varepsilon b_n \bigg)\leq \frac{1}{\varepsilon b_n}\E \bigg[\sum_{i=1}^{\infty}|\mathscr{C}_{\sss (i)}|\times  \surp{\mathscr{C}_{\sss (i)}} \1_{\{ |\mathscr{C}_{\sss (i)}|\leq \delta b_n\}} \bigg]
  \\&= \frac{a_n}{\varepsilon}\expt{\mathrm{SP}(\mathscr{C}(V_n))\1_{\{ |\mathscr{C}(V_n)|\leq \delta b_n\}}}\\
  &= \frac{a_n}{\varepsilon}\sum_{k=1}^{\infty}\sum_{i\geq \log_2(1/(k^{0.12}\delta))}\PR\big(\mathrm{SP}(\mathscr{C}(V_n))\geq k, |\mathscr{C}(V_n)|\in (2^{-(i+1)}k^{-0.12}b_n, 2^{-i}k^{-0.12}b_n] \big)\\
  &\leq \frac{C}{\varepsilon} \sum_{k=1}^{\infty}\frac{1}{k^{1.1}}\sum_{i\geq \log_2(1/(k^{0.12}\delta))} 2^{-i/2} \leq \frac{C}{\varepsilon}\sum_{k=1}^{\infty}\frac{\sqrt{\delta}}{k^{1.04}}  =O(\sqrt{\delta}).
 \end{split}
\end{equation} 
The third step in \eqref{eq:u-0-calc} follows using 
\begin{eq}
&\expt{\mathrm{SP}(\mathscr{C}(V_n))\1_{\{ |\mathscr{C}(V_n)|\leq \delta b_n\}}} = \sum_{k= 1}^\infty \sum_{j\geq k} \PR(\mathrm{SP}(\mathscr{C}(V_n)) = j,  |\mathscr{C}(V_n)|\leq \delta b_n) \\
& = \sum_{k= 1}^\infty \sum_{j\geq k} \sum_{i\geq \log_2(1/(k^{0.12}\delta))}\PR\big(\mathrm{SP}(\mathscr{C}(V_n))= j, |\mathscr{C}(V_n)|\in (2^{-(i+1)}k^{-0.12}b_n, 2^{-i}k^{-0.12}b_n]\big),
\end{eq}
and the second-last step in \eqref{eq:u-0-calc} follows from Lemma~\ref{lem:sp-cv-n}.
 The proof of Proposition~\ref{prop-surp-u-0} now follows for $\lambda <0$. 
\par Next, consider the case $\lambda > 0$. Fix a large integer $R\geq 1$ such that $\lambda - \sum_{i=1}^R\theta_i^2<0$. This can be done because $\boldsymbol{\theta}\notin \ell^{2}_{\shortarrow}$. Using \eqref{eq:CgeqT1}, for any $\delta>0 $, it is possible to choose $T>0$ such that for all sufficiently large $n$,
\begin{equation}\label{eq:early-expl}
 \prob{\text{all the vertices }1,\dots,R \text{ are explored within time }Tb_n }> 1-\delta.
\end{equation} Let $T_e$ denote the first time after $Tb_n$ when we finish exploring a component. By Theorem~\ref{thm::convegence::exploration_process}, $(b_n^{-1}T_e)_{n\geq 1}$ is a tight sequence. Let $\mathcal{G}^*_T$ denote the graph obtained by removing the components explored up to time $T_e$. Then, $\mathcal{G}^*_T$ is again a configuration model conditioned on its degrees. Let $\nu_n^*$ denote the value of the criticality parameter for $\mathcal{G}^*$. Note that 
\begin{equation}
 \sum_{i\notin \mathscr{V}_{\sss T_e}}d_i\geq \ell_n-2Tb_n \implies  \sum_{i\notin \mathscr{V}_{\sss T_e}}d_i = \ell_n+\oP(n),
\end{equation}and thus conditionally on $\mathscr{F}_{\sss T_e}$ and the fact that $(1,\dots, R)$ are explored within time $Tb_n$,
\begin{equation}\label{ub-nu*}
 \begin{split}
 \nu_n^*\leq \frac{\sum_{i\in [n]}d_i(d_i-1)-\sum_{i=1}^Rd_i(d_i-1)}{\sum_{i\notin \mathscr{V}_{\sss T_e}}d_i}-1=1+c_n^{-1}\big(\lambda-\sum_{i=1}^R\theta_i^2\big) +o(c_n^{-1}).
 \end{split}
\end{equation}
Therefore, combining \eqref{eq:early-expl}, \eqref{ub-nu*}, we can use Lemma~\ref{lem:sp-cv-n}  on $\mathcal{G}^*_T$ since $c_n(\nu_n^*-1)<0$. Thus, if $\mathscr{C}_{\sss(i)}^*$ denotes the $i^{th}$ largest component of $\mathcal{G}_T^*$, then
\begin{equation}\label{surplus-positivelamba1} \lim_{T\to\infty}\lim_{\delta\to 0}\limsup_{n\to\infty}\PR\bigg(\sum_{i: |\mathscr{C}_{(i)}^*|\leq \delta b_n }|\mathscr{C}_{\sss(i)}^*|\times \surp{\mathscr{C}_{\sss(i)}^*}> \varepsilon b_n\bigg)=0.
\end{equation} To conclude the proof for the whole graph $\mathrm{CM}_n(\boldsymbol{d})$ (with $\lambda >0$), let $$\mathcal{K}_n^T:=\{i:|\mathscr{C}_{\sss(i)}|\leq \delta b_n, |\mathscr{C}_{\sss(i)}| \text{ is explored before the time }T_e  \}.$$ Note that
\begin{equation}
 \begin{split}
  \sum_{i \in \mathcal{K}_n^T}|\mathscr{C}_{\sss (i)}|\times \mathrm{SP}(\mathscr{C}_{\sss (i)})&\leq \Big( \sum_{i\in \mathcal{K}_n^T}|\mathscr{C}_{\sss (i)}|^2\Big)^{1/2}\times \Big(\sum_{i\in \mathcal{K}_n^T}\mathrm{SP}(\mathscr{C}_{\sss (i)})^2 \Big)^{1/2}\\
  &\leq  \bigg( \sum_{i:|\mathscr{C}_{\sss (i)}|\leq \delta b_n}|\mathscr{C}_{\sss (i)}|^2\bigg)^{1/2}\times \mathrm{SP}(T_e),
 \end{split}
\end{equation}where $\mathrm{SP}(t)$ is the number of surplus edges explored up to time $tb_n$ and we have used the fact that $\sum_{i\in\mathcal{K}_n^T}\mathrm{SP}(\mathscr{C}_{\sss (i)})^2\leq (\sum_{i\in\mathcal{K}_n^T}\mathrm{SP}(\mathscr{C}_{\sss (i)}))^2\leq \mathrm{SP}(T_e)^2$. From Lemma~\ref{lem:surp:poisson-conv} and the $\ell^{2}_\shortarrow$ tightness in Theorem~\ref{thm::conv:component:size}, we can conclude that for any $T>0$,
\begin{equation}\label{surplus-positivelamba2}
 \lim_{\delta\to 0}\limsup_{n\to\infty}\PR\bigg(\sum_{i \in \mathcal{K}_n^T}|\mathscr{C}_{\sss (i)}|\times \mathrm{SP}(\mathscr{C}_{\sss (i)})>\varepsilon b_n\bigg)=0.
\end{equation}
  The proof is now complete for the case $\lambda > 0$ by combining \eqref{surplus-positivelamba1} and \eqref{surplus-positivelamba2}.
\end{proof}
\section{Proof for simple graphs} \label{sec:simple-graphs}
In this section, we give a proof of Theorem~\ref{thm:simple-graph}. Let $\PR_s(\cdot)$ (respectively $\E_s[\cdot]$) denote the probability measure (respectively the expectation) conditionally on the graph $\mathrm{CM}_n(\boldsymbol{d})$ being simple. For any process $\mathbf{X}$ on $\mathbb{D}([0,\infty),\R)$, we define $\mathbf{X}^T:=(X(t))_{t\leq T}$. Thus the truncated process $\mathbf{X}^T$ is $\mathbb{D}([0,T],\R)$-valued. Now, by \cite[Theorem 1.1]{J09c}, $\liminf_{n\to\infty}\PR(\mathrm{CM}_n(\boldsymbol{d})\text{ is simple})>0$. This fact ensures that, under the conditional measure $\PR_s$,  $(b_n^{-1}|\mathscr{C}_{\sss (i)}|)_{i\geq 1}$ is tight with respect to the $\ell^2_{\shortarrow}$-topology.
 Therefore, to conclude Theorem~\ref{thm:simple-graph}, it suffices to show that the exploration process $\bar{\mathbf{S}}_n$, defined in \eqref{eqn::scaled_process},  has the  same limit (in distribution) under $\PR_s$ as obtained in Theorem~\ref{thm::convegence::exploration_process} so that the finite-dimensional limit of $(b_n^{-1}|\mathscr{C}_{\sss (i)}|)_{i\geq 1}$ remains unchanged under $\PR_s$. Thus, it is enough to show that for any bounded continuous function $f:\mathbb{D}([0,T],\R)\mapsto\R$,
\begin{equation}\label{diff:expt:bounded:cont}
 \big| \E[f(\bar{\mathbf{S}}_n^T)]-\E_s[f(\bar{\mathbf{S}}_n^T)]\big|\to 0.
\end{equation}
Let $\ell_n':=\ell_n-2Tb_n-d_{1}+1$. We first estimate the number of multiple edges or self-loops discovered in the graph up to time $Tb_n$. 
Let $v_l$ denote the \emph{exploring} vertex in the breadth-first exploration given by Algorithm~\ref{algo-expl}, and let $d_{v_l}$ denote the degree of $v_l$.
Note that, uniformly over $l\leq Tb_n$, any half-edge of $v_l$ creates a self-loop with probability at most $d_{v_l}/\ell_n'$.
Thus, the expected number of self-loops attached to $v_l$, conditionally on $\mathscr{F}_{l-1}$, is at most $d_{v_l}^2/\ell_n'$.
Moreover, the expected number of multiple edge attached to $v_l$, conditionally on $\mathscr{F}_{l-1}$, is at most
\begin{eq}
\frac{d_{v_l}(d_{v_l}-1) \sum_{i\in[n]} d_i(d_{i}-1)}{\ell_n'(\ell_n'-1)} = (1+o(1))\frac{d_{v_l}^2}{\ell_n'},
\end{eq}where we have used that $\nu_n = 1+o(1)$.
Therefore,
\begin{equation}
 \expt{\#\{\text{self-loops or multiple edges discovered while }v_l \text{ is exploring}\}|\mathscr{F}_{l-1}}\leq \frac{3d_{v_l}^2}{\ell_n'}.
\end{equation}Thus, for any $T>0$,
\begin{equation}
\begin{split}
 &\expt{\#\{\text{self-loops or multiple edges discovered up to time }Tb_n\}}\\
 &\leq \frac{3}{\ell_n'}\E\bigg[\sum_{i\in [n]}d_i^2\mathcal{I}^n_i(Tb_n)\bigg]=\frac{3}{\ell_n'}\E\bigg[\sum_{i=1}^K d_i^2\mathcal{I}^n_i(Tb_n)\bigg]+\frac{3}{\ell_n'}\E\bigg[\sum_{i=K+1}^nd_i^2\mathcal{I}^n_i(Tb_n)\bigg],
 \end{split}
\end{equation}where $\mathcal{I}^n_i(l)=\ind{i\in \mathscr{V}_l}$. Now, using Assumption~\ref{assumption1}~\eqref{assumption1-1}, for every fixed $K\geq 1$,
\begin{equation}
 \frac{3}{\ell_n'}\E\bigg[\sum_{i=1}^K d_i^2\mathcal{I}^n_i(Tb_n)\bigg]\leq \frac{3}{\ell_n'} \sum_{i=1}^K d_i^2 \to 0,
\end{equation} since $2\alpha-1<0$. Moreover, recall from  \eqref{eq:prob-ind} that $\prob{\mathcal{I}^n_i(Tb_n)=1}\leq Tb_nd_i/\ell_n'$. Therefore, for some constant $C>0$,
\begin{equation}
 \begin{split} 
  \frac{3}{\ell_n'}\E\bigg[\sum_{i=K+1}^nd_i^2\mathcal{I}_i^n(Tb_n)\bigg]\leq \frac{3Tb_n}{\ell_n'^2}\sum_{i=K+1}^nd_i^3\leq C \bigg( a_n^{-3}\sum_{i=K+1}^nd_i^3 \bigg), 
 \end{split}
\end{equation}which, by Assumption~\ref{assumption1}~\eqref{assumption1-2}, tends to zero if we first take $\limsup_{n\to\infty}$ and then take $\lim_{K\to\infty}$. Consequently, for any fixed $T>0$, as $n\to\infty$,
\begin{equation}
 \prob{\text{at least one self-loop or multiple edge is discovered before time }Tb_n}\to 0.
\end{equation}
Now,
 \begin{equation}\label{simple-after-Tbn-enough}
 \begin{split}
  &\expt{f(\bar{\mathbf{S}}_n^T) \ind{\CM \text{ is simple}}}\\
  & =\expt{f(\bar{\mathbf{S}}_n^T) \ind{\text{no self-loops or multiple  edges found after time }Tb_n}}+o(1)\\
  &= \expt{f(\bar{\mathbf{S}}_n^T) \prob{\text{no self-loops or multiple  edges found after time }T b_n\vert \mathscr{F}_{\sss Tb_n}}} +o(1),
  \end{split}
 \end{equation}
 Define $T_e=\inf\{l\geq Tb_n: \text{a component is finished exploring at time } l\}$. Using the fact that $(b_n^{-1}T_e)_{n\geq 1}$ is a tight sequence, the limit of the expected number of self-loops or multiple edges discovered between time $Tb_n$ and $T_e$ is again zero. As in the proof of Proposition~\ref{prop-surp-u-0}, consider the graph $\mathcal{G}^*$, obtained by removing the components obtained up to time $T_e$. Thus, $\mathcal{G}^*$ is a configuration model, conditioned on its degree sequence. Let $\nu^*_n$ be the criticality parameter of $\mathcal{G}^*$. Then, we claim that $\nu^*_n\pto 1$. To see this note that $\sum_{i\notin \mathscr{V}_{\sss T_e}}d_i=\ell_n+\oP(n)$. 
Further, note that by Assumption~\ref{assumption1}~\eqref{assumption1-2} \eqref{eqn:increment:indicator}, for any $t>0$,
\begin{equation}\label{finite-expt-di2}
\begin{split}
\limsup_{n\to\infty}  \E\bigg[a_n^{-2}\sum_{i\in [n]}d_i^2\mathcal{I}_i^n(tb_n)\bigg] \leq\limsup_{n\to\infty} a_n^{-2}tb_n\frac{\sum_{i\in [n]}d_i^3}{\ell_n-2tb_n}<\infty,
\end{split}
\end{equation} 
which implies that  
 $\sum_{i\notin \mathscr{V}_{\sss T_e}}d_i^2=\sum_{i\in [n]}d_i^2+\oP(n)$ and thus the claim is proved. Since the degree distribution has finite second moment, using  \cite[Theorem 7.11]{RGCN1} we get \begin{equation}\label{prob-of-simple}
\prob{\mathcal{G}^* \text{ is simple}\Big\vert\mathscr{F}_{\sss T_e}}\pto \e^{-3/4}.
\end{equation}Now using  \eqref{simple-after-Tbn-enough}, \eqref{prob-of-simple} and the dominated convergence theorem, we conclude that
\begin{equation}
 \expt{f(\bar{\mathbf{S}}_n^T) \ind{\CM \text{ is simple}}}=  \expt{f(\bar{\mathbf{S}}_n^T)}\prob{\CM \text{ is simple}}+o(1).
\end{equation}Therefore, \eqref{diff:expt:bounded:cont} follows and the proof of Theorem~\ref{thm:simple-graph} is complete.
\qed
\section{Scaling limits for component functionals}\label{sec:comp-functional}
Suppose that vertex $i$ has an associated weight $w_i$. The total weight of the component $\mathscr{C}_{\sss (i)}$ is denoted by $\mathscr{W}_i = \sum_{k\in \mathscr{C}_{\sss (i)}} w_k$. The goal of this section is to derive the scaling limits for $(\mathscr{W}_i )_{i\geq 1}$ when the weight sequence satisfies some regularity conditions:
\begin{assumption}\label{assumption-weight} \normalfont The weight sequences $\bld{w} = (w_i)_{i\in [n]}$ satisfies
\begin{enumerate}[(i)]
\item \label{assumption-weight-1} $\sum_{i\in [n]}w_i = O(n)$, and  $\lim_{n\to\infty}\frac{1}{\ell_n}\sum_{i\in [n]} d_i w_i = \mu_{w}$.
\item \label{assumption-weight-2} $\max\{\sum_{i\in [n]}d_iw_i^2,\sum_{i\in [n]}d_i^2w_i\} = O(a_n^3)$.
\end{enumerate}
\end{assumption}
\begin{theorem}\label{thm:comp-functionals}Consider $\CM$ satisfying {\rm Assumption~\ref{assumption1}} and a weight sequence $\bld{w}$ satisfying {\rm Assumption~\ref{assumption-weight}}. Denote $\mathbf{Z}^w_n= \ord( b_n^{-1}\mathscr{W}_i,\mathrm{SP}(\mathscr{C}_{\sss (i)}))_{i\geq 1}$ and $\mathbf{Z}^w:=\ord(\mu_{w}\gamma_i(\lambda), N(\gamma_i))_{i\geq 1}$, where $\gamma_i(\lambda)$, and  $N(\gamma_i)$ are defined in {\rm Theorem~\ref{thm:spls}}. As $n\to\infty$,
\begin{equation}\label{eq:weight-conv-amc}
 \mathbf{Z}_n^w \dto \mathbf{Z}^w,
\end{equation}with respect to the $\mathbb{U}^0_{\shortarrow}$ topology.
\end{theorem}
The proof Theorem~\ref{thm:comp-functionals} can be decomposed into two main steps: the first one is to obtain the finite-dimensional limits of $\mathbf{Z}_n^w$ and secondly to  prove the $\mathbb{U}^0_{\shortarrow}$ convergence. The finite-dimensional limit is a consequence of the fact that the total weight of the clusters is approximately equal to the cluster sizes. The argument for the tightness with respect to the $\mathbb{U}^0_{\shortarrow}$ topology is similar to Propositions~\ref{prop-l2-tightness}~and~\ref{prop-surp-u-0} and therefore we only provide a sketch with pointers to all the necessary ingredients. Recall that $\mathcal{I}_i^n(l) = \ind{i\in \mathscr{V}_l}$, where $\mathscr{V}_l$ is the set of discovered vertices by Algorithm~\ref{algo-expl} upto time $l$.
\begin{lemma}\label{lem:weight-approx-size}Under {\rm Assumptions~\ref{assumption1},~\ref{assumption-weight}}, for any $T>0$, 
\begin{equation}\label{weight-expl-prop}
 \sup_{u\leq T}\bigg| \sum_{i\in [n]} w_i\mathcal{I}_i^n(ub_n)-\frac{\sum_{i\in [n]}d_iw_i}{\ell_n}ub_n\bigg|=\OP(a_n).
\end{equation} Consequently, for each fixed $i\geq 1$,
\begin{equation}\label{weight-approx-size}
   \mathscr{W}_i = \mu_{w} \big| \mathscr{C}_{\sss (i)} \big| +\oP(b_n).   
  \end{equation}
\end{lemma}
\begin{proof}Fix any $T>0$. Let $\ell_n(T)=\ell_n-2Tb_n+1$. Define 
\begin{equation}
W_n(l)=\sum_{i\in [n]}w_i\mathcal{I}_i^n(l) - \frac{\sum_{i\in [n]}d_iw_i}{\ell_n(T)} l.
\end{equation}
 The goal is to use the supermartingale inequality \eqref{eqn:supmg:ineq} in the same spirit as in the proof of \eqref{tail::martingale}. 
 Firstly, observe from \eqref{eq:prob-ind} that
\begin{equation}
 \begin{split}
  \E[W_n(l+1)-W_n(l) \mid \mathscr{F}_l]&=\E\bigg[\sum_{i\in [n]}w_i \left(\mathcal{I}^n_i(l+1)-\mathcal{I}_i^n(l)\right)\Big| \mathscr{F}_l\bigg]-\frac{\sum_{i\in [n]}d_iw_i}{\ell_n(T)}\\
  &= \sum_{i\in [n]} w_i\E\big[\mathcal{I}^n_i(l+1)\big| \mathscr{F}_l\big]\ind{\mathcal{I}_i^n(l)=0} -\frac{\sum_{i\in [n]}d_iw_i}{\ell_n(T)}\leq 0,
  \end{split}
 \end{equation}
uniformly over $l\leq Tb_n$ and therefore, $(W_n(l))_{l=1}^{Tb_n}$ is a supermartingale. Using \eqref{prob-ind-lb}, we compute
\begin{equation}\label{expt-W-n}
\begin{split}
 \big|\E[W_n(l)]\big|&=\sum_{i\in [n]}w_i\left(\frac{d_i}{\ell_n}-\prob{\mathcal{I}_i^n(l)=1}\right)\\
 &\leq \sum_{i\in [n]} w_i\bigg(1-\bigg(1-\frac{d_i}{\ell_n}\bigg)^l-\frac{d_i}{\ell_n}l\bigg) + l\sum_{i\in [n]}w_i \bigg(\frac{d_i}{\ell_n(T)}-\frac{d_i}{\ell_n}\bigg)\\
 &\leq 2(2Tb_n)^2 \frac{\sum_{i\in [n]}d_i^2w_i}{\ell_n(T)^2}=O(b_n^2a_n^3/n^2) = O(a_n),
\end{split}
\end{equation}
uniformly over $l\leq Tb_n$. Also, using \eqref{neg:correlation}, \eqref{var-ind-ub}, and Assumption~\ref{assumption-weight}~\eqref{assumption-weight-2},
\begin{equation}\label{var-W-n}
\mathrm{Var}(W_n(l))\leq \sum_{i\in [n]} w_i^2 \mathrm{var}(\mathcal{I}_i^n(l))\leq Tb_n\frac{\sum_{i\in [n]}d_iw_i^2}{\ell_n(T)} =O(a_n^2),
\end{equation}uniformly over $l\leq Tb_n$. Using \eqref{eqn:supmg:ineq}, \eqref{expt-W-n} and \eqref{var-W-n}, we conclude the proof of \eqref{weight-expl-prop}.
The proof of \eqref{weight-approx-size} follows using Lemma~\ref{lem:no-large-comp-later} and simply observing that $a_n=o(b_n)$. 
\end{proof}
\begin{proof}[Proof of Theorem~\ref{thm:comp-functionals}]
Lemma~\ref{lem:weight-approx-size} ensures the finite-dimensional convergence in \eqref{eq:weight-conv-amc}. Thus, the proof is complete if we can show that, for any $\varepsilon > 0$ 
\begin{subequations}
\begin{equation}\label{eq:suff-U0-conv-1}
 \lim_{K\to \infty}\limsup_{n\to\infty}\PR\bigg(\sum_{i>K}\mathscr{W}_i^2>\varepsilon b_n^2\bigg)=0,
\end{equation}and
\begin{equation}\label{eq:suff-U0-conv-2}
\lim_{\delta \to 0}\limsup_{n\to\infty}\PR\bigg( \sum_{|\mathscr{C}_{\sss (i)}|\leq \delta b_n} \mathscr{W}_i\times \surp{\mathscr{C}_{\sss (i)}}> \varepsilon b_n\bigg) = 0.
\end{equation}
\end{subequations}
The arguments for proving \eqref{eq:suff-U0-conv-1} and \eqref{eq:suff-U0-conv-2} are similar to Propositions~\ref{prop-l2-tightness}~and~\ref{prop-surp-u-0} and thus we only sketch a brief outline. 
Denote $\ell_n^w = \sum_{i\in [n]}w_i$. 
The main ingredient to the proof of Proposition~\ref{prop-l2-tightness} is Lemma~\ref{lem::tail_sum_squares}, and the proof of  Lemma~\ref{lem::tail_sum_squares} uses the fact that the expected sum of squares of the cluster sizes can be written in terms of susceptibility functions in \eqref{comp-vs-rnd-choice} and then we made use of the estimate for the susceptibility function in \eqref{bound::expt-cluster-size}. Let $V_n'$ denote a vertex chosen according to the distribution $(w_i/\ell_n^w)_{i\in [n]}$, independently of the graph. Notice that
\begin{equation}\label{W-vs-WVn}
 \E\bigg[\sum_{i\geq 1 }\mathscr{W}_i^2\bigg]=   \ell_n^w\E\big[ \mathscr{W}(V_n')\big]. 
\end{equation}
Now, \cite[Lemma 5.2]{J09b} can be extended using an identical argument  to compute the weight-based susceptibility function in the right hand side of \eqref{W-vs-WVn}. See Lemma~\ref{lem:gen-path-count} given in Appendix~\ref{sec:appendix-gen-path-counting}. The proof of \eqref{eq:suff-U0-conv-2} can also be completed using an identical argument as in the proof of Proposition~\ref{prop-surp-u-0} by observing that 
\begin{equation}
 \PR\bigg( \sum_{|\mathscr{C}_{\sss (i)}|\leq \delta b_n} \mathscr{W}_i\times\surp{\mathscr{C}_{\sss (i)}}> \varepsilon b_n \bigg)\leq \frac{\ell_n^w}{\varepsilon b_n}\expt{\mathrm{SP}(\mathscr{C}(V_n'))\1_{\{ |\mathscr{C}(V_n')|\leq \delta b_n\}}}.
\end{equation} 
Moreover, an analog of Lemma~\ref{lem:sp-cv-n} also holds for $V_n'$ (see Appendix~\ref{sec_appendix}), and the proof of \eqref{eq:suff-U0-conv-2} can now be completed in an identical manner as in the proof of Proposition~\ref{prop-surp-u-0}.
\end{proof}
While studying percolation in the next section, we will need an estimate for the proportion of degree-one vertices in the large components. In fact, an application of Theorem~\ref{thm:comp-functionals}, yields the following result about the degree composition of the largest clusters:
\begin{corollary}\label{cor:degree-k} Consider $\mathrm{CM}_n(\boldsymbol{d})$ satisfying {\rm Assumption~\ref{assumption1}}. Let $v_k(G)$ denote the number of vertices of degree $k$ in the graph $G$.  Then, for any fixed $i\geq 1$,
\begin{equation} \label{eqn_vertices_of_degree_k}
   v_k \big( \mathscr{C}_{\sss(i)} \big) = \frac{kr_k}{\mu} \big| \mathscr{C}_{\sss (i)} \big| +\oP(b_n),   
  \end{equation}
  where $r_k=\mathbbm{P}(D=k)$. 
  Denote $\mathbf{Z}^k_n= \ord( b_n^{-1} v_k ( \mathscr{C}_{\sss(i)} ),\mathrm{SP}(\mathscr{C}_{\sss (i)}))_{i\geq 1}$, $\mathbf{Z}^k:=\ord(\frac{kr_k}{\mu}\gamma_i(\lambda), N(\gamma_i))_{i\geq 1}$, where $\gamma_i(\lambda)$, and  $N(\gamma_i)$ are defined in {\rm Theorem~\ref{thm:spls}}. As $n\to\infty$,
\begin{equation}\label{eq:conv-amc-k}
 \mathbf{Z}_n^k \dto \mathbf{Z}^k,
\end{equation}with respect to the $\mathbb{U}^0_{\shortarrow}$ topology. 
\end{corollary}
\begin{proof}
 The proof follows directly from Theorem~\ref{thm:comp-functionals} by putting $w_i = \ind{d_i=k}$. The fact that this weight sequence satisfies Assumption~\ref{assumption-weight} with $\mu_w = kr_k/\mu$ is a consequence of Assumption~\ref{assumption1}.
\end{proof}

\section{Percolation}\label{sec:perc} 
In this section, we study critical percolation on the configuration model for fixed $\lambda\in\R$ and complete the proof of Theorem~\ref{thm:percolation}. 
As discussed earlier, $\mathrm{CM}_n(\boldsymbol{d},p)$ is obtained by first constructing $\mathrm{CM}_n(\boldsymbol{d})$ and then deleting each edge with probability $1-p$, independently of each other, and of the graph $\CM$. 
An interesting property of the configuration model is that $\mathrm{CM}_n(\boldsymbol{d},p)$ is also  distributed as a configuration model conditional on the degrees \cite{F07}. The rough idea here is to show that the degree distribution of $\mathrm{CM}_n(\bld{d},p_n(\lambda))$ satisfies Assumption~\ref{assumption1}, where $p_n(\lambda)$ is given by Assumption~\ref{assumption2}. This allows us to invoke Theorem~\ref{thm:spls} and complete the proof of Theorem~\ref{thm:percolation}. 
Recall from Assumption~\ref{assumption2} that $\nu=\lim_{n\to\infty}\nu_n>1$, and $p_n=p_n(\lambda)=\nu_n^{-1}(1+\lambda c_n^{-1})$. We start by describing an algorithm due to \citet{J09} that is easier to work with:
\begin{algo}[Construction of $\mathrm{CM}_n(\boldsymbol{d},p_n)$] \label{algo:perc}
\normalfont Initially, vertex $i$ has $d_i$ half-edges incident to it. For each half-edge $e$, let $v_e$ be the vertex to which $e$ is incident.
\begin{itemize}
 \item[(S1)]  With probability $1-\sqrt{p_n}$, one detaches $e$ from $v_e$ and associates $e$ to a new vertex $v'$ of degree-one. Color the new vertex $red$. This is done independently for every existing half-edge and we call this whole process $explosion$. Let $n_+$ be the number of red vertices created by explosion and $\tilde{n}=n+n_+$.  Denote the degree sequence obtained from the above procedure by $\Mtilde{\boldsymbol{d}} = ( \tilde{d}_i )_{i \in [\tilde{n}]}$, i.e., $\tilde{d}_i \sim \text{Bin} (d_i, \sqrt{p_n})$ for $i \in [n]$ and $\tilde{d}_i=1$ for $i \in [\tilde{n}] \setminus [n]$; 
 \item[(S2)] Construct $\mathrm{CM}_{\tilde{n}}(\Mtilde{\boldsymbol{d}})$ independently of (S1);
 \item[(S3)] Delete all the red vertices and the edges attached to them.
 \end{itemize}
 \end{algo}
It was also shown in \cite{J09} that the obtained multigraph has the same distribution as $\mathrm{CM}_{n}(\boldsymbol{d},p)$ if we replace (S3) by
 \begin{itemize}
 \item[(S3$'$)] instead of deleting red vertices, choose $n_+$ degree-one vertices uniformly at random without replacement, independently of (S1) and (S2), and delete them. 
 \end{itemize}

\begin{remark}\normalfont Notice that Algorithm~\ref{algo:perc}~(S1) induces a probability measure $\mathbbm{P}_p^n$ on $\N^\infty$. Denote their product measure by $\mathbbm{P}_p$. In words, for different $n$, (S1) is carried out independently.  All the almost sure statements about the degrees in this section will be with  respect to the probability measure $\mathbbm{P}_p$. 
\end{remark} 
 Let us first show that $\Mtilde{\boldsymbol{d}}$ also satisfies Assumption~\ref{assumption1}~\eqref{assumption1-2}. Note that the total number of half-edges remains unchanged during the explosion in Algorithm~\ref{algo:perc}~(S1) and therefore  $\sum_{i\in [\tilde{n}]}\tilde{d}_i=\sum_{i\in [n]}d_i$ and by Assumption~\ref{assumption2}~\eqref{assumption2-1},
\begin{equation}\label{eq:perc-mu}
 \frac{1}{n}\sum_{i\in [\tilde{n}]}\tilde{d}_i \to \mu \quad \PR_p-\text{ a.s.}
\end{equation}This verifies the first moment condition in Assumption~\ref{assumption1}~\eqref{assumption1-2} for the percolated degree sequence $\PR_p$ a.s.
 Let $I_{ij}$:= the indicator of the $j^{th}$ half-edge corresponding to vertex $i$ being kept after the explosion. Then $I_{ij} \sim \text{Ber} (\sqrt{p_n})$ independently for $i \in [n]$, $j \in [d_i]$. Let 
 \begin{equation} 
 \mathbf{I}:= (I_{ij})_{j \in [d_i], i \in [n]}  \quad\text{and} \quad  f_1(\mathbf{I}):=\sum_{i\in [n]} \tilde{d_i}(\tilde{d}_i-1).
 \end{equation}Note that $f_1(\mathbf{I})=\sum_{i\in [\tilde{n}]}\tilde{d}_i(\tilde{d}_i-1)$ since the degree-one vertices do not contribute to the sum. One can check that by changing the status of one half-edge corresponding to vertex $k$ we can change $f_1$ by at most $2(d_{k}+1)$. Therefore an application  of \cite[Corollary 2.27]{JLR00} yields 
 \begin{equation}\label{eqn::prob:ineq:third}
 \mathbbm{P}_p \Big( \Big|\sum_{i\in [n]} \tilde{d_i}(\tilde{d}_i-1)- p_n \sum_{i\in [n]} d_i(d_i-1) \Big| >t \Big) \leq 2 \exp \bigg( -\frac{t^2}{2\sum_{i \in [n]} d_i (d_{i}+1)^2}\bigg).
 \end{equation}
Now by Assumption~\ref{assumption2}~\eqref{assumption2-1}, $\sum_{i\in [n]}d_i^3=O(a_n^3)$. If we set $t=n^{1-\varepsilon}c_n^{-1}$, then $t^2/(\sum_{i\in [n]}d_i^3)$ is of the order $n^{\alpha-2\varepsilon}/L(n)$. Thus, choosing $\varepsilon<\alpha/2$, using \eqref{eqn::prob:ineq:third} and the Borel-Cantelli lemma we conclude that 
 \begin{equation}\label{perc:2nd-moment}\sum_{i\in [n]} \tilde{d_i}(\tilde{d}_i-1)= p_n \sum_{i\in [n]} d_i(d_i-1) +o(nc_n^{-1}) \quad \mathbbm{P}_p\text{-a.s.} 
 \end{equation}Thus, using Assumption~\ref{assumption2}, the second moment condition in Assumption~\ref{assumption1}~\eqref{assumption1-2} is verified for the percolated degree sequence $\PR_p$ a.s. Let $\tilde{d}_{\sss (i)}$ denote the $i^{th}$ largest value of $(\tilde{d}_i)_{i\in [\tilde{n}]}$. The third-moment condition in Assumption~\ref{assumption1}~\eqref{assumption1-2} is obtained by noting that $\tilde{d}_i\leq d_i$ for all $i\in [n]$ and  
 \begin{equation}\label{eq:third-moment-perc}\begin{split}
  &\lim_{K\to\infty}\limsup_{n\to\infty}a_n^{-3}\sum_{i=K+1}^{\tilde{n}}\tilde{d}_{\sss (i)}^{3}\leq \lim_{K\to\infty}\limsup_{n\to\infty}a_n^{-3}\sum_{i=K+1}^{\tilde{n}}\tilde{d}_{i}^{3}\\ &\leq \lim_{K\to\infty}\limsup_{n\to\infty}a_n^{-3}\Big(\sum_{i=K+1}^{n}\tilde{d}_{i}^{3}+ n_+\Big)\leq  \lim_{K\to\infty}\limsup_{n\to\infty}a_n^{-3}\Big(\sum_{i=K+1}^{n}d_{i}^{3}+ n_+\Big) \to 0 \quad \PR_p\ \text{a.s.,}
\end{split}
\end{equation}where we have used Assumption~\ref{assumption2}~\eqref{assumption2-1} and the fact that $a_n^{-3}n_+\to 0$, $\PR_p$ a.s., which follows by observing that $n_+\sim\mathrm{Bin}(\ell_n,1-\sqrt{p_n})$, and $a_n \gg n^{1/3}$ for $\tau\in (3,4)$. To see that $\Mtilde{\boldsymbol{d}}$ satisfies Assumption~\ref{assumption1}~\eqref{assumption1-3} note that by \eqref{perc:2nd-moment},
 \begin{equation}
  \frac{\sum_{i\in [\tilde{n}]}\tilde{d}_i(\tilde{d}_i-1)}{\sum_{i\in [\tilde{n}]}\tilde{d}_i}=p_n\frac{\sum_{i\in [n] }d_i(d_i-1)}{\sum_{i\in [n] }d_i}+o(c_n^{-1})=1+\lambda c_n^{-1}+o(c_n^{-1}) \quad \PR_p \text{ a.s.,}
 \end{equation}where the last step follows from Assumption~\ref{assumption2}~\eqref{assumption2-2}. Assumption~\ref{assumption1}~\eqref{assumption1-4} is trivially satisfied by $\Mtilde{\boldsymbol{d}}$. 
Finally, in order to verify Assumption~\ref{assumption1}~\eqref{assumption1-1}, it suffices to show that 
 \begin{equation}\label{eq:assump-1-perc}
  \frac{\tilde{d}_{\sss(i)}}{a_n}\to \theta_i\sqrt{p}, \quad \PR_p \text{ a.s.,}
 \end{equation}where $p=1/\nu$.
 Recall that $\Mtilde{d}_i\sim \mathrm{Bin}(d_i,\sqrt{p_n})$. A standard concentration inequality for the binomial distribution \cite[(2.9)]{JLR00} yields that, for any $0<\varepsilon\leq 3/2$, 
\begin{equation} 
 \PR(|\tilde{d}_i-d_i\sqrt{p_n}|>\varepsilon d_i\sqrt{p_n})\leq 2\mathrm{exp}(-\varepsilon^2d_i\sqrt{p_n}/3),
 \end{equation}
  and using the Borel-Cantelli lemma it follows that $\mathbbm{P}_p$ almost surely, $\Mtilde{d}_i=d_i\sqrt{p_n}(1+o(1))$ for all fixed $i$. 
Moreover, an application of \eqref{eq:third-moment-perc} yields that 
 \begin{equation}
  \lim_{K\to\infty}\limsup_{n\to\infty}a_n^{-3}\max_{i>K}\tilde{d}_i^3  =0.
 \end{equation}
  Now, since $\boldsymbol{\theta}$ is an ordered vector, the proof of \eqref{eq:assump-1-perc} follows.
 
To summarize, the above discussion in \eqref{eq:perc-mu},~\eqref{perc:2nd-moment},~\eqref{eq:third-moment-perc},~and~\eqref{eq:assump-1-perc} yields that the degree sequence $\Mtilde{d}$ satisfies all the conditions in Assumption~\ref{assumption1}. Therefore, Theorem~\ref{thm:spls} can be applied to $\mathrm{CM}_{\tilde{n}}(\Mtilde{\boldsymbol{d}})$.
Denote by $\tilde{\mathscr{C}}_{(i)}$ the $i^{th}$ largest component of $\mathrm{CM}_{\tilde{n}}(\Mtilde{\boldsymbol{d}})$. 
Let $\tilde{\mathbf{Z}}_n= \ord( b_n^{-1}|\tilde{\mathscr{C}}_{(i)}|,\mathrm{SP}(\tilde{\mathscr{C}}_{(i)})_{i\geq 1}$ and $\tilde{\mathbf{Z}}:=\ord(\tilde{\gamma}_i(\lambda), N(\tilde{\gamma}_i))_{i\geq 1}$, where $\gamma_i(\lambda)$, and  $N(\gamma_i)$ are defined in {\rm Theorem~\ref{thm:percolation}}. Now,  Theorem~\ref{thm:spls} implies
\begin{equation}
\tilde{\mathbf{Z}}_n \dto \tilde{\mathbf{Z}},
\end{equation}with respect to the $\mathbb{U}^0_{\shortarrow}$ topology.

Since the percolated degree sequence satisfies Assumption~\ref{assumption1} $\mathbbm{P}_p$-a.s., \eqref{eqn_vertices_of_degree_k} holds for $\tilde{\mathscr{C}}_{\sss (i)}$ also. Let $v^d_1(\tilde{\mathscr{C}}_{\sss (i)})$ be the number of degree-one vertices of $\tilde{\mathscr{C}}_{\sss (i)}$ which are deleted while creating the graph $\mathrm{CM}_{n}(\boldsymbol{d},p_{n})$ from  $\mathrm{CM}_{\tilde{n}}(\Mtilde{\boldsymbol{d}})$. 
Recall that $\tilde{n}_1$ is the number of degree-one vertices left after Algorithm~\ref{algo:perc}~(S2).
Since the vertices are to be chosen uniformly from all degree-one vertices as described in (S3$'$), 
\begin{equation} \label{degree_one_vertices}
\begin{split}
 v^d_1(\tilde{\mathscr{C}}_{\sss (i)}) &= \frac{n_+}{\tilde{n}_1}v_1(\tilde{\mathscr{C}}_{\sss (i)})+ \oP(b_n)= \frac{n_+}{\tilde{n}_1} \frac{\tilde{n}_1}{\ell_n} \big| \tilde{\mathscr{C}}_{\sss (i)}\big| + \oP(b_n) = \frac{n_+}{\ell_n}\big| \tilde{\mathscr{C}}_{(i)} \big|+ \oP(b_n)\\
 & = \frac{\mu\big(1-\sqrt{p}_n\big)+o(1)}{\mu+o(1)}\big| \tilde{\mathscr{C}}_{\sss (i)} \big|+ \oP(b_n) = \big(1-\sqrt{p}_n\big) \big| \tilde{\mathscr{C}}_{\sss (i)}\big| + \oP(b_n),
 \end{split}
\end{equation}where the penultimate equality follows by observing that $n_+\sim\mathrm{Bin}(\ell_n,1-\sqrt{p_n})$. Now, notice that by removing degree-one vertices, the components are not broken up, so the  vector of component sizes for percolation can be obtained by just  subtracting the number of red vertices from the component sizes of $\mathrm{CM}_{\tilde{n}}(\Mtilde{\boldsymbol{d}})$. Moreover, the removal of degree-one vertices does not effect the surplus edge counts.
Therefore, the proof of Theorem~\ref{thm:percolation} is complete by using Corollary~\ref{cor:degree-k}. 

\section{Convergence to AMC}
\label{sec:conv-amc}
Let us give an overview of the organization of this section:
In Section~\ref{sec:dyn-cons-coup}, we discuss an alternative dynamic  construction that approximates the  percolated graph process, coupled in a natural way.
This construction enables us to compare the coupled percolated graphs with a dynamic construction.
Then, we describe a modified system that evolves as an exact augmented multiplicative coalescent and the rest of the section is devoted to comparing the exact augmented multiplicative coalescent and the corresponding quantities for the graphs generated by the dynamic construction.
For all the results in this section, we assume that the assumptions of Theorem~\ref{thm:mul:conv} holds.

\subsection{The dynamic construction and the coupling} \label{sec:dyn-cons-coup}
Let us consider graphs generated dynamically as follows: 
\begin{algo}\label{algo:dyn-cons-2} \normalfont Let $s_1(t)$ be the total number of unpaired or \emph{open} half-edges at time $t$, and  $\Xi_n$ be an inhomogeneous Poisson process with rate $s_1(t)$ at time $t$. 
\begin{itemize}
\item[{\rm (S0)}] Initially, $s_1(0)=\ell_n$, and $\mathcal{G}_n(0)$ is the empty graph on vertex set $[n]$. 
\item[{\rm(S1)}] At each event time of $\Xi_n$, choose two open half-edges uniformly at random and pair them. The graph $\mathcal{G}_n(t)$ is obtained by adding this edge to $\mathcal{G}_n(t-)$. Decrease $s_1(t)$ by two. Continue until $s_1(t)$ becomes zero.
\end{itemize} 
\end{algo} 
Notice that $\mathcal{G}_n(\infty)$ is distributed as $\CM$ since an open  half-edge is paired with another uniformly chosen open half-edge. 
The next proposition ensures that the graph process generated by Algorithm~\ref{algo:dyn-cons-2} \emph{sandwiches} the graph process $(\mathrm{CM}_n(\bld{d},p_n(\lambda)))_{\lambda\in\R}$. This result was proved in \cite[Proposition~8.4]{DHLS15}. The proof is identical under Assumption~\ref{assumption2} and therefore is omitted here.
Define
\begin{equation}\label{defn:t-n-lambda}
t_n(\lambda)=\frac{1}{2}\log\bigg(\frac{\nu_n}{\nu_n-1}\bigg)+\frac{1}{2(\nu_n-1)}\frac{\lambda}{c_n}.
\end{equation}
\begin{proposition}\label{prop:coupling-whp} Fix $-\infty<\lambda_\star<\lambda^\star<\infty$. There exists a coupling such that with high probability
\begin{subequations}
\begin{equation}
 \mathcal{G}_n(t_n(\lambda)-\varepsilon_n)\subset \mathrm{CM}_n(\bld{d},p_n(\lambda)) \subset\mathcal{G}_n(t_n(\lambda)+\varepsilon_n),\quad \forall \lambda \in [\lambda_\star,\lambda^\star],
\end{equation} and 
\begin{equation}\label{eq:prop-coup-2}
 \mathrm{CM}_n(\bld{d},p_n(\lambda)-\varepsilon_n)\subset \mathcal{G}_n(t_n(\lambda))\subset \mathrm{CM}_n(\bld{d},p_n(\lambda)+\varepsilon_n),\quad \forall \lambda \in [\lambda_\star,\lambda^\star],
\end{equation}
\end{subequations}where $\varepsilon_{n}=cn^{-\gamma_0}$, for some $\eta<\gamma_0<1/2$ and the constant $c$ does not depend on $\lambda$.
\end{proposition}
We next prove an important consequence of Proposition~\ref{prop:coupling-whp} which says that  the rescaled component sizes and surplus edges for $\mathrm{CM}_n(\bld{d},p_n(\lambda))$ and $\mathcal{G}_n(t_n(\lambda))$ are close in $\mathbb{U}^0_{\shortarrow}$. 
This allows us to focus our attention to  $\mathcal{G}_n(t_n(\lambda))$ in the subsequent analysis. 
From here onward, we augment $\lambda$ to a predefined notation to emphasize the dependence on~$\lambda$. Let $\mathscr{C}_{\sss (i)} (\lambda)$ and $\mathscr{C}_{\sss (i)} ^p(\lambda)$ denote respectively the $i$-th largest component of $\mathrm{CM}_n(\bld{d},p_n(\lambda))$ and $\mathcal{G}_n(t_n(\lambda))$. 
Also, let $\mathbf{Z}_n(\lambda)$ (respectively $\mathbf{Z}_n^p(\lambda)$) denote the vector $(b_n^{-1}|\mathscr{C}_{\sss (i)} (\lambda)|,\mathrm{SP}(\mathscr{C}_{\sss (i)} (\lambda)))_{i\geq 1}$ (respectively $(b_n^{-1}|\mathscr{C}_{\sss (i)}^p (\lambda)|,\mathrm{SP}(\mathscr{C}_{\sss (i)}^p (\lambda)))_{i\geq 1}$ ), ordered as an element of $\mathbb{U}^0_{\shortarrow}$. Recall the definition of the metric $\mathrm{d}_{\mathbb{U}}$ from \eqref{defn_U_metric}.
\begin{proposition}\label{prop:perc-dyn-close}
Under the coupling in {\rm Proposition~\ref{prop:coupling-whp}}, as $n\to\infty$, $\mathrm{d}_{\mathbb{U}} (\mathbf{Z}_n(\lambda), \mathbf{Z}_n^p(\lambda)) \pto 0$ for any $\lambda\in\R$.
\end{proposition}
We first prove the following elementary fact that will be crucial in our analysis:
\begin{fact}\label{fact:exc}
$\PR(\forall i\geq 1: \gamma_{i+1} < \gamma_{i}) = 1,$ where $\gamma_i$ denotes the length of the $i$-th largest excursion in \eqref{defn::limiting::process}.
\end{fact}
\begin{claimproof}
It is enough to show that, with probability one, no two excursions of a process in \eqref{defn::limiting::process} have same length. 
For any rational $q$, let $e(q)$ be the excursion containing $q$. 
Thus it is enough to show that for two rationals $q_1,q_2$,
\begin{eq}
\PR(e(q_1) \neq e(q_2), \text{ but }|e(q_1)| = |e(q_2)|) =0.
\end{eq}
Without loss of generality, suppose $e(q_1) $ appears earlier than $e(q_2)$. 
Let $V_{q_2} = \{i: \mathcal{I}_i (q_2) = 1\}$. 
Thus conditional on the sigma-field $\sigma (\bar{S}_\infty^\lambda(s):s\leq q_2)$, the process $(\iS(q_2+t) - \iS(q_2))_{t\geq 0}$ is distributed as $\hat{\mathbf{S}}_{\infty}^\lambda$ given by  
\begin{eq}\label{eq:conditional-process}
\hat{S}_{\infty}^\lambda(t) =  \sum_{i\notin V_{q_2}} \theta_i\left(\mathcal{I}_i(t)- (\theta_i/\mu)t\right)+\lambda t.
\end{eq}
So the process in \eqref{eq:conditional-process} again has the form \eqref{defn::limiting::process}. 
Now, for any $x>0$, the probability that $|e(q_2)| = x$, conditionally on $\sigma (\bar{S}_\infty^\lambda(s):s\leq q_2)$ and $e(q_1) = x$, is zero using Lemma~\ref{lem:no-atom} from Appendix~\ref{sec:appendix-density-existence}. This concludes the proof of Fact~\ref{fact:exc}
\end{claimproof}
\begin{proof}[Proof of Proposition~\ref{prop:perc-dyn-close}]
 Let $\mathscr{C}_{\sss (i)} ^{p,+}(\lambda)$ (respectively  $\mathscr{C}_{\sss (i)} ^{p,-}(\lambda)$) denote the $i$-th largest component of $\mathrm{CM}_n(\bld{d},p_n(\lambda)+\varepsilon_n)$ (respectively $\mathrm{CM}_n(\bld{d},p_n(\lambda)-\varepsilon_n)$) where we have chosen $\varepsilon_n$ according to Proposition~\ref{prop:coupling-whp}.
 Let $\mathbf{Z}_n^{p,+}(\lambda)$ (respectively $\mathbf{Z}_n^{p,-}(\lambda)$) denote the vector $(b_n^{-1}|\mathscr{C}_{\sss (i)}^{p.+} (\lambda)|,\mathrm{SP}(\mathscr{C}_{\sss (i)}^{p,+} (\lambda)))_{i\geq 1}$ (respectively $(b_n^{-1}|\mathscr{C}_{\sss (i)}^{p.-} (\lambda)|,\mathrm{SP}(\mathscr{C}_{\sss (i)}^{p,-} (\lambda)))_{i\geq 1}$), ordered as an element of $\mathbb{U}^0_{\shortarrow}$. 
 Using Theorem~\ref{thm:percolation}, both $\mathbf{Z}_n^{p,+}(\lambda)$ and $\mathbf{Z}_n^{p,-}(\lambda)$ have the same scaling limit. 

Next, for any $K\geq 1$, 
 \begin{eq}\label{eq:sandwich-perc-12}
 \lim_{n\to\infty} \PR(\mathscr{C}_{\sss (i)} ^{p,-}(\lambda) \subset \mathscr{C}_{\sss (i)} ^{p,+}(\lambda), \ \forall i\leq K) = 1.
 \end{eq}
If $\mathscr{C}_{\sss (1)} ^{p,-}(\lambda) $ is not contained in $ \mathscr{C}_{\sss (1)} ^{p,+}(\lambda)$, then $|\mathscr{C}_{\sss (1)} ^{p,-}(\lambda)|\leq |\mathscr{C}_{\sss (j)} ^{p,+}(\lambda)|$ for some $j\geq 2$, which implies that $|\mathscr{C}_{\sss (1)} ^{p,-}(\lambda)|\leq |\mathscr{C}_{\sss (2)} ^{p,+}(\lambda)|$. 
Suppose that there is a subsequence $(n_{0k})_{k\geq 1}$ along which 
\begin{eq}\label{liminf-prob-pos}
\lim_{n_{0k}\to\infty}   \PR(|\mathscr{C}_{\sss (1)}^{p,-}(\lambda)| \leq |\mathscr{C}_{\sss (2)}^{p,+}(\lambda)|) >0.
\end{eq}
If \eqref{liminf-prob-pos} yields a contradiction, then \eqref{eq:sandwich-perc-12} is proved for $K=1$. 
To this end, note that $(b_{n}^{-1}(|\mathscr{C}_{\sss (i)}^{p,-}(\lambda)|, |\mathscr{C}_{\sss (i)}^{p,+}(\lambda)|)_{i\geq 1})_{n\geq 1}$ is tight in $(\ell^2_{\shortarrow})^2$.
Thus taking a convergent subsequence along $(n_k)_{k\geq 1} \subset (n_{0k})_{k\geq 1}$, 
it follows that 
\begin{eq}
b_{n}^{-1}(|\mathscr{C}_{\sss (i)}^{p,-}(\lambda)|, |\mathscr{C}_{\sss (i)}^{p,+}(\lambda)|)_{i\geq 1} \dto (\gamma_i,\bar{\gamma}_i)_{i\geq 1} \quad \text{ in }(\ell^2_{\shortarrow})^2,
\end{eq}
 where $(\gamma_i)_{i\geq 1} \stackrel{\sss d}{=}(\bar{\gamma}_i)_{i\geq 1}$.
Thus, along the subsequence $(n_k)_{k\geq 1}$, 
\begin{eq} \label{eq:prob-second-larger-largest-2}
\lim_{n_k\to\infty} \PR(|\mathscr{C}_{\sss (1)}^{p,-}(\lambda)| \leq |\mathscr{C}_{\sss (2)}^{p,+}(\lambda)|) = \PR(\gamma_1\leq \bar{\gamma}_2).
\end{eq}
\begin{fact}\label{fact:same-dist-coupling-2}
For all $i\geq 1$, $\gamma_i = \bar{\gamma}_i$ almost surely.
\end{fact}
\begin{claimproof}
Under the coupling in Proposition~\ref{prop:coupling-whp}, $\sum_{j\leq i}|\mathscr{C}_{\sss (j)}^{p,-}(\lambda)|\leq \sum_{j\leq i}|\mathscr{C}^{p,+}_{\sss (j)}(\lambda)|$ and therefore $\PR(\sum_{j\leq i}\gamma_j\leq \sum_{j\leq i} \bar{\gamma}_j) =1$, for each fixed $i\geq 1$. 
In particular, $\gamma_1 \leq \bar{\gamma}_1$. But, since $\gamma_1,\bar{\gamma}_1$ have the same distribution, it must be the case that $\gamma_1 = \bar{\gamma}_1$ almost surely. 
Inductively, we can prove $\gamma_i = \bar{\gamma}_i$ almost surely.
\end{claimproof}
\noindent Thus, using Fact~\ref{fact:same-dist-coupling-2}, \eqref{eq:prob-second-larger-largest-2} reduces to
\begin{eq} \label{eq:prob-second-larger-largest-3}
\lim_{n_k\to\infty} \PR(|\mathscr{C}_{\sss (1)}^{p,-}(\lambda)| \leq |\mathscr{C}^{p,+}_{\sss (2)}(\lambda)|) = \PR(\gamma_1\leq \gamma_2) = \PR(\gamma_1= \gamma_2) = 0,
\end{eq}
where the last equality follows from Fact~\ref{fact:exc}. Now, \eqref{eq:prob-second-larger-largest-3} contradicts \eqref{liminf-prob-pos}, and thus \eqref{eq:sandwich-perc-12} follows for $K=1$.
For $K\geq 2$, we can use a similar argument to show that, with high probability, $\cup_{i\leq K}\mathscr{C}_{\sss (i)}^{p,-}(\lambda)\subset \cup_{i\leq K}\mathscr{C}_{\sss (i)}^{p,+}(\lambda)$. 
Now, if both $\mathscr{C}_{\sss (1)}^{p,-}(\lambda)$ and  $\mathscr{C}_{\sss (2)}^{p,-}(\lambda)$ are  contained in $\mathscr{C}^{p,+}_{\sss (1)}(\lambda)$, then $|\mathscr{C}^{p,+}_{\sss (1)}(\lambda)| \geq |\mathscr{C}_{\sss (1)}^{p,-}(\lambda)|+|\mathscr{C}_{\sss (2)}^{p,-}(\lambda)|$, which occurs with probability tending to zero.
This follows using Fact~\ref{fact:same-dist-coupling-2} and $\PR(\gamma_1\geq \gamma_1+\gamma_2) = 0$.
Thus, $\mathscr{C}_{\sss (2)}^{p,-}(\lambda) \subset \mathscr{C}^{p,+}_{\sss (2)}(\lambda)$ with high probability and we can use similar arguments to conclude this for $i\leq K$. Thus, the proof of \eqref{eq:sandwich-perc-12} follows for general $K\geq 1$. 

Next, we show that, for any $K\geq 1$, 
 \begin{eq}\label{eq:sandwich-perc-123}
 \lim_{n\to\infty} \PR\big(\mathscr{C}_{\sss (i)} ^{p,-}(\lambda) \subset \mathscr{C}_{\sss (i)} (\lambda) \subset \mathscr{C}_{\sss (i)} ^{p,+}(\lambda), \ \forall i\leq K\big) = 1.
 \end{eq}
If $\mathscr{C}_{\sss (1)} (\lambda) $ is not contained in $ \mathscr{C}_{\sss (1)} ^{p,+}(\lambda)$, then $|\mathscr{C}_{\sss (1)}(\lambda)|\leq |\mathscr{C}_{\sss (2)} ^{p,+}(\lambda)|$. 
However, since $|\mathscr{C}_{\sss (1)}^{p,-}(\lambda)|\leq |\mathscr{C}_{\sss (1)}(\lambda)|$, it follows that $|\mathscr{C}_{\sss (1)}^{p,-}(\lambda)|\leq |\mathscr{C}_{\sss (2)} ^{p,+}(\lambda)|$.
Now, one can repeat identical argument as in   \eqref{eq:sandwich-perc-12} to prove that $\mathscr{C}_{\sss (i)} (\lambda) \subset \mathscr{C}_{\sss (i)} ^{p,+}(\lambda)$ for all $i\leq K$ with high probability.
Moreover, since $\mathrm{CM}_n(\bld{d},p_n(\lambda)-\varepsilon_n)\subset \mathcal{G}_n(t_n(\lambda))$ and $\mathscr{C}_{\sss (i)} ^{p,-}(\lambda) \subset \mathscr{C}_{\sss (i)} ^{p,+}(\lambda),$ for all $i\leq K$ with high probability, it must also be the case that $\mathscr{C}_{\sss (i)} ^{p,-}(\lambda) \subset \mathscr{C}_{\sss (i)} (\lambda)\subset \mathscr{C}_{\sss (i)} ^{p,+}(\lambda),$ for all $i\leq K$ with high probability. Thus we conclude \eqref{eq:sandwich-perc-123}.

Similarly, one can also show that 
\begin{eq}\label{eq:sandwich-perc-124}
 \lim_{n\to\infty} \PR\big(\mathscr{C}_{\sss (i)} ^{p,-}(\lambda) \subset \mathscr{C}_{\sss (i)}^p (\lambda) \subset \mathscr{C}_{\sss (i)} ^{p,+}(\lambda), \ \forall i\leq K\big) = 1.
 \end{eq}
Finally, since $\mathbf{Z}_n^{p,-}(\lambda)$ and $\mathbf{Z}_n^{p,+}(\lambda)$ have the same distributional limit, it follows using \eqref{eq:sandwich-perc-12} that for all $i\leq K$
\begin{eq}
|\mathscr{C}_{\sss (i)} ^{p,+}(\lambda)| - |\mathscr{C}_{\sss (i)} ^{p,-}(\lambda)| = \oP(b_n) \quad\text{and}\quad  \mathrm{SP}(\mathscr{C}_{\sss (i)} ^{p,+}(\lambda))-\mathrm{SP}(\mathscr{C}_{\sss (i)} ^{p,-}(\lambda)) \to 0.
\end{eq}
Thus, \eqref{eq:sandwich-perc-123} and \eqref{eq:sandwich-perc-124} yields
 \begin{eq}\label{eq:bounds-comp-surp}
\big| |\mathscr{C}_{\sss (i)} ^{p}(\lambda)| - |\mathscr{C}_{\sss (i)} (\lambda)| \big|= \oP(b_n) \quad\text{and}\quad  \big|\mathrm{SP}(\mathscr{C}_{\sss (i)} ^p(\lambda))-\mathrm{SP}(\mathscr{C}_{\sss (i)} (\lambda)) \big|\to 0.
\end{eq} 
Moreover, $(\mathbf{Z}_n^{p}(\lambda))_{n\geq 1}$ is tight in $\mathbb{U}^0_{\shortarrow}$, and
since both $(\mathbf{Z}_n^{p,-}(\lambda))_{n\geq 1}$ and $(\mathbf{Z}_n^{p,+}(\lambda))_{n\geq 1}$ are tight in $\mathbb{U}^0_{\shortarrow}$, it also follows that $(\mathbf{Z}_n(\lambda))_{n\geq 1}$ is tight in $\mathbb{U}^0_{\shortarrow}$.
Let
$\pi_k,T_k:\mathbb{U}^0_{\shortarrow}\mapsto\mathbb{U}^0_{\shortarrow}$ be the functions such that for $\mathbf{z}=((x_i,y_i))_{i\geq 1}$, $\pi_k(\mathbf{z})$ consists of only $(x_i,y_i)$ for $i\leq k$ and zeroes in other coordinates, and $T_k(\mathbf{z})$ consists only of $(x_i,y_i)$ for $i>k$. Thus,
\begin{equation}\label{split-up-fixed-lambda-1}
\begin{split}
 \mathrm{d}_{\mathbb{U}}\left(\mathbf{Z}_n^p(\lambda),\mathbf{Z}_n(\lambda)\right)&\leq\mathrm{d}_{\mathbb{U}}\left(\pi_K(\mathbf{Z}_{n}^p(\lambda)),\pi_K(\mathbf{Z}_n(\lambda))\right)+\|T_K(\mathbf{Z}_n^p(\lambda))\|_{\sss\mathbb{U}}+\|T_K(\mathbf{Z}_n(\lambda))\|_{\sss\mathbb{U}}.
 \end{split}
\end{equation}
 Now, for each fixed $K\geq 1$ the first term in the right hand side of \eqref{split-up-fixed-lambda-1} converges in probability to zero by \eqref{eq:bounds-comp-surp}. Also, using the tightness of both $(\bZ_n(\lambda))_{n\geq 1}$ and $(\bZ_n^p(\lambda))_{n\geq 1}$ with respect to the $\mathbb{U}^0_{\shortarrow}$ topology, it follows that for any $\varepsilon>0$,
\begin{equation}
\lim_{K\to\infty}\lim_{n\to\infty}\prob{\|T_K(\mathbf{Z}_n(\lambda))\|_{\sss \mathbb{U}}>\varepsilon}=\lim_{K\to\infty}\lim_{n\to\infty}\prob{\|T_K(\mathbf{Z}_n^p(\lambda))\|_{\sss \mathbb{U}}>\varepsilon}=0
\end{equation}
Thus, the proof of Proposition~\ref{prop:perc-dyn-close} now follows.

\end{proof}

Let us now describe the evolution of the components under the dynamic construction Algorithm~\ref{algo:dyn-cons-2}.
We write $\mathscr{C}_{\sss (i)}(\lambda)$  for the $i^{th}$ largest component of $\mathcal{G}_n(t_n(\lambda))$ and define 
\begin{equation} \label{defn:open-half-edge}
\mathcal{O}_i(\lambda)=\# \text{ open half-edges in }\mathscr{C}_{\sss (i)}(\lambda).
\end{equation}
Think of $\mathcal{O}_i(\lambda)$ as the \emph{mass} of the component $\mathscr{C}_{\sss (i)}(\lambda)$. 
Let $\mathbf{Z}_n^o(\lambda)$ denote the vector of the number of open half-edges (re-scaled by $b_n$) and surplus edges of $\mathcal{G}_n(t_n(\lambda))$, ordered as an element of $\mathbb{U}^0_{\shortarrow}$. 
For a process $\mathbf{X}$, we will write $\mathbf{X}[\lambda_\star,\lambda^\star]$ to denote the restricted process $(X(\lambda))_{\lambda\in[\lambda_\star,\lambda^\star]}$. 
 Let $\ell_n^o(\lambda) = \sum_{i\geq 1}\mathcal{O}_i(\lambda)$. 
Note that 
\begin{equation}\label{eq:asympt-ell-n-o}
\ell_n^o(\lambda) =  \frac{n\mu(\nu-1)}{\nu}(1+\oP(1)).
\end{equation} 
Indeed, \eqref{eq:asympt-ell-n-o} is a consequence of \cite[Lemma 8.2]{BBSX14} since the proof only uses the facts that $|\ell_n/n-\mu|=o(n^{-\gamma})$ for all $\gamma<1/2$, and $\sum_{i\in [n]}d_i(d_i-1)/\ell_n\to\nu$.  
Now, observe that during the evolution of the graph process generated  by Algorithm~\ref{algo:dyn-cons-2} in the time interval $[t_n(\lambda),t_n(\lambda+\dif \lambda)]$, the $i^{th}$ and $j^{th}$ ($i> j$) largest components merge at rate 
 \begin{equation}\label{rate:function}
2\mathcal{O}_{i}(\lambda) \mathcal{O}_{j}(\lambda)\times\frac{1}{\ell_n^o(\lambda)-1}\times \frac{1}{2(\nu_n-1)c_n}\approx \frac{\nu}{\mu(\nu-1)^2} \big(b_n^{-1}\mathcal{O}_{i}(\lambda)\big)\big(b_n^{-1}\mathcal{O}_{j}(\lambda)\big),
\end{equation}and create a component with open half-edges $\mathcal{O}_{i}(\lambda)+\mathcal{O}_{j}(\lambda)-2$ and surplus edges $\mathrm{SP}(\mathscr{C}_{\sss(i)}(\lambda))+\mathrm{SP}(\mathscr{C}_{\sss(j)}(\lambda))$. 
Also, a surplus edge is created in $\mathscr{C}_{\sss(i)}(\lambda)$ at rate
\begin{equation}\label{rate:function-spls}
\mathcal{O}_i(\lambda)(\mathcal{O}_i(\lambda)-1)\times\frac{1}{\ell_n^o(\lambda)-1}\times \frac{1}{2(\nu_n-1)c_n}\approx \frac{\nu}{2\mu(\nu-1)^2} \big(b_n^{-1}\mathcal{O}_{i}(\lambda)\big)^2,
\end{equation}and $\mathscr{C}_{\sss(i)}(\lambda)$ becomes a component with surplus edges  $\mathrm{SP}(\mathscr{C}_{\sss(i)}(\lambda))+1$  and open half-edges $\mathcal{O}_{i}(\lambda)-2$. Thus $\mathbf{Z}_n^o[\lambda_\star,\lambda^\star]$ does \emph{not} evolve as an AMC process but it is close. 
The fact that two half-edges are killed after pairing, makes the masses (the number of open half-edges) of the components  and the system to deplete. If there were no such depletion of mass, then the vector of open half-edges, along with the surplus edges, would in fact  merge as an augmented multiplicative coalescent. 
Let us define the modified process \cite[Algorithm~7]{DHLS15} that in fact evolves as augmented multiplicative coalescent:

\begin{algo}\label{algo:modify-dyn-cons} \normalfont Initialize $\bar{\mathcal{G}}_n(t_n(\lambda_\star)) = \mathcal{G}_n(t_n(\lambda_\star))$.  Let $\mathscr{O}$ denote the set of open half-edges in the graph $\mathcal{G}_n(t_n(\lambda_\star))$, $\bar{s}_1 = |\mathscr{O}|$ and $\bar{\Xi}_n$ denote a Poisson process with rate $\bar{s}_1$. At each event time of the Poisson process $\bar{\Xi}_n$, select two half-edges from $\mathscr{O}$ and create an edge between the corresponding vertices. However, the selected half-edges are kept alive, so that they can be selected again. Denote the resulting graph by $\bar{\mathcal{G}}_n(t_n(\lambda))$.
\end{algo} 

\begin{remark}\label{rem:modify-AMC}\normalfont The only difference between Algorithms~\ref{algo:dyn-cons-2}~and~\ref{algo:modify-dyn-cons}  is that the \emph{paired} half-edges are not discarded and thus more edges are created by Algorithm~\ref{algo:modify-dyn-cons}. Thus, there is a natural coupling between the graphs generated by Algorithms~\ref{algo:dyn-cons-2}~and~\ref{algo:modify-dyn-cons} such that $\mathcal{G}_n(t_n(\lambda))\subset \bar{\mathcal{G}}_n(t_n(\lambda))$ for all $\lambda\in [\lambda_\star,\lambda^\star]$, with probability one. In the subsequent part of this section, we will always work under this coupling. The extra edges that are created by Algorithm~\ref{algo:modify-dyn-cons} will be called \emph{bad} edges.
\end{remark}
 In the subsequent part of this paper, we will augment a predefined notation with a bar to denote the corresponding quantity for $\bar{\mathcal{G}}_n(t_n(\lambda))$. 
 Denote $\beta_n = (\bar{s}_1(\nu_n-1)c_n)^{1/2}$ and let  $\bar{\mathbf{Z}}_n^{o,{\sss \mathrm{scl}}}(\lambda)$ denote the vector $\ord(\beta_n^{-1}\bar{\mathcal{O}}_i(\lambda),\mathrm{SP}(\bar{\mathscr{C}}_{\sss (i)}(\lambda)))_{i\geq 1}$. 
 Using an argument identical to \eqref{rate:function}~and~\eqref{rate:function-spls}, it follows that $\bar{\mathbf{Z}}_n^{o,{\sss \mathrm{scl}}}[\lambda_\star,\lambda^\star]$ evolves as a standard augmented multiplicative coalescent.
 Note that there exists a constant $c>0$ such that $\beta_n = cb_n(1+\oP(1))$, and therefore the scaling limit of any finite-dimensional distributions of $\bar{\mathbf{Z}}_n^{o}[\lambda_\star,\lambda^\star]$ can be obtained from $\bar{\mathbf{Z}}_n^{o,{\sss \mathrm{scl}}}[\lambda_\star,\lambda^\star]$.

\subsubsection{Multiplicative coalescent with mass and weight}
To deduce the scaling limits involving the components sizes, let us consider a dynamic process that additionally tracks the evolution of some weights of components.
Initially, the system consists of particles (possibly infinitely many) where particle $i$ has mass $x_i$, weight $z_i$.
Think of $x_i$'s as the number of  open half-edges, and $z_i$'s as the component sizes.
Let $(X_i(t),Z_i(t))_{i\geq 1}$ 
denote masses, and weights  at time $t$. 
We always order $(Z_{i} (t))_{i\geq 1}$ in a decreasing manner.
The dynamics of the system is described as follows: At time $t$,
\begin{itemize}
\item[$\rhd$]  particles $i$ and $j$ coalesce at rate $X_i(t)X_j(t)$ and create a particle with mass $X_i(t)+X_j(t)$, weight $Z_i(t)+Z_j(t)$ and attribute $Y_i(t)+Y_j(t)$.
\end{itemize}
For $\bld{x},\bld{z}\in\ell^2_{\shortarrow}$, we 
denote by  $\mathrm{MC}_2(\bld{x},\bld{z},t)$  
the vector $(Z_i(t))_{i\geq 1}$.
We will need the following theorem:
\begin{theorem}\label{thm:AMC-2D}
Suppose that $(\bld{x}_n,\bld{z}_n) \to (\bld{x},\bld{x})$ in $\ell^2\times \ell^2_{\shortarrow}$. Then, for any $t\geq 0$,
 \begin{equation}\label{mul-coal-2d}
  \mathrm{MC}_2(\bld{x}_n,\bld{z}_n,t) \dto \mathrm{MC}_2(\bld{x},\bld{x},t).
 \end{equation}
\end{theorem}

\begin{proof}
For $\bld{x}_n = (x_i^n)_{i\geq 1}$ and $\bld{z}_n = (z_i^n)_{i\geq 1}$,  let $\bld{w}_n^+ = \mathrm{ord}(x_i^n\vee z_i^n)$, $\bld{w}_n^-=\mathrm{ord}(x_i^n\wedge z_i^n)$, where $\mathrm{ord}$ denotes the decreasing ordering of the elements. 
Notice that $\bld{w}_n^+ \to \bld{x}$, and $\bld{w}_n^- \to \bld{x}$ in $\ell^2_{\shortarrow}$.
Using the Feller property of the multiplicative coalescent \cite[Proposition 5]{A97}, it follows that
\begin{equation}\label{limit-ub-lb}
 \mathrm{MC}_2(\bld{w}_n^+,\bld{w}_n^+,t)\dto  \mathrm{MC}_2(\bld{x},\bld{x},t), \quad \text{and} \quad \mathrm{MC}_2(\bld{w}_n^-,\bld{w}_n^-,t)\dto  \mathrm{MC}_2(\bld{x},\bld{x},t),
\end{equation}with respect to the $\ell^2_{\shortarrow}$ topology. 
Suppose that $\mathrm{MC}_2(\bld{w}_n^+,\bld{w}_n^+,t)$ and $\mathrm{MC}_2(\bld{w}_n^-,\bld{w}_n^-,t)$ are coupled through the subgraph coupling (see \cite[Page 838]{A97}). 
Under the subgraph coupling, \eqref{limit-ub-lb} yields
\begin{equation}
\|\mathrm{MC}_2(\bld{w}_n^+,\bld{w}_n^+,t)\|_{\sss 2}^2-\|\mathrm{MC}_2(\bld{w}_n^-,\bld{w}_n^-,t)\|_{\sss 2}^2 \pto 0.
\end{equation}
Moreover,
\begin{equation}
 \|\mathrm{MC}_2(\bld{w}_n^-,\bld{w}_n^-,t)\|_{\sss 2}^2 \leq \|\mathrm{MC}_2(\bld{x}_n,\bld{z}_n,t)\|_{\sss 2}^2 \leq \|\mathrm{MC}_2(\bld{w}_n^+,\bld{w}_n^+,t)\|_{\sss 2}^2.
\end{equation}
Hence, using \cite[Lemma 17]{A97}, under the subgraph coupling,
\begin{equation}
 \|\mathrm{MC}_2(\bld{w}_n^+,\bld{w}_n^+,t) - \mathrm{MC}_2(\bld{x}_n,\bld{z}_n,t)\|_{\sss 2}^2\leq \|\mathrm{MC}_2(\bld{w}_n^+,\bld{w}_n^+,t)\|_{\sss 2}^2 - \|\mathrm{MC}_2(\bld{x}_n,\bld{z}_n,t)\|_{\sss 2}^2 \pto 0,
\end{equation}
and thus the proof of \eqref{mul-coal-2d} follows.
\end{proof}
\subsection{Asymptotics for the open half-edges}
The following lemma shows that the number of open half-edges in $\mathcal{G}_n(t_n(\lambda))$ is \emph{approximately} proportional to the component sizes. 
This will enable us to apply Theorem~\ref{thm:AMC-2D} for deducing the scaling limits of the required quantities for the graph $\bar{\mathcal{G}}_n(t_n(\lambda))$. Let $\mathcal{O}_{\sss (i)}(\lambda)$ to denote the $i$-th largest number in the sequence  $(\mathcal{O}_{j}(\lambda))_{j\geq 1}$. 
Also recall that $\mathbf{Z}_n^o(\lambda)$ denotes the vector $(\mathcal{O}_i(\lambda), \mathrm{SP} (\mathscr{C}_{\sss (i)}(\lambda)))_{i\geq 1}$, ordered as an element of $\mathbb{U}^0_{\shortarrow}$.
\begin{lemma}\label{thm:open-comp}
 There exists a constant $\kappa > 0$ such that, for any $i\geq 1$, 
 \begin{equation}\label{open-he}
  \mathcal{O}_i(\lambda)= \kappa |\mathscr{C}_{\sss (i)}(\lambda)|+o_{\sss \PR}(b_n).
 \end{equation}
Also, for any fixed $K\geq 1$,
\begin{eq}\label{open-he-compare}
\lim_{n\to\infty}\PR(\mathcal{O}_i(\lambda) = \mathcal{O}_{\sss (i)}(\lambda), \  \forall i\leq K).
\end{eq}
 Further, $(\mathbf{Z}_n^o(\lambda))_{n\geq 1}$ is tight in  $\mathbb{U}^0_{\shortarrow}$. 
\end{lemma}
\begin{proof} Let $\bld{d}^\lambda = (d_k^\lambda)_{k\in [n]}$ be the degree sequence of $\cG_n(t_n(\lambda))$.
Now, conditionally on $\bld{d}^\lambda$,  $\cG_n(t_n(\lambda))$ is distributed as a configuration model.
We apply Theorem~\ref{thm:comp-functionals} with $w_i$ being the number of open half-edges incident to vertex $i$. 
To this end, it is enough to show that Assumption~\ref{assumption-weight} holds with high probability.
Since $w_i \leq d_i$ and $d_i^\lambda \leq d_i$,  Assumption~\ref{assumption-weight}~(ii) is obvious. 
Thus it is enough to show that $\sum_{i\in [n]} d_i^\lambda w_i / \sum_{i\in [n]} d_i^\lambda$ converges in probability to some constant. 
This can be verified using the differential equation method. See Appendix~\ref{sec:susceptibility} for details.
Thus, \eqref{open-he} follows from \eqref{weight-approx-size}. 
The $\mathbb{U}^0_{\shortarrow}$-tightness also follows from Theorem~\ref{thm:comp-functionals}.

Next, we prove \eqref{open-he-compare}. 
Using \eqref{eq:suff-U0-conv-1}, for any $\varepsilon,\eta>0$, there exists $M=M(\varepsilon,\eta)\geq 1$ such that
\begin{eq}
\PR \bigg(\sum_{i>M} \mathcal{O}_i(\lambda)^2 > \varepsilon b_n^2\bigg) < \eta.
\end{eq}
Now, if $\mathcal{O}_1(\lambda) \neq \mathcal{O}_{\sss (1)} (\lambda)$, then there exists a $j\geq 2$ such that  $\mathcal{O}_1(\lambda) \leq \mathcal{O}_{j} (\lambda)$.
Thus, 
\begin{eq}
\PR(\mathcal{O}_1(\lambda) \neq \mathcal{O}_{\sss (1)} (\lambda)) \leq \PR(\exists 2\leq j\leq M: \mathcal{O}_1(\lambda) \leq \mathcal{O}_{j} (\lambda)) + \PR(\exists j> M: \mathcal{O}_1(\lambda) \leq \mathcal{O}_{j} (\lambda)).
\end{eq}
Using \eqref{open-he} and Theorem~\ref{thm:percolation} together with Proposition~\ref{prop:perc-dyn-close}, it follows that $(b_n^{-1} \mathcal{O}_i(\lambda))_{i \leq M}$ converge in distribution to $(\gamma_i)_{i\in [M]}$, where $\gamma_i$ denotes the excursion length of a process of the form \eqref{defn::limiting::process}.
An application of Fact~\ref{fact:exc} yields
\begin{eq}\label{O-1-1}
\lim_{n\to\infty} \PR(\exists 2\leq j\leq M: \mathcal{O}_1(\lambda) \leq \mathcal{O}_{j} (\lambda)) \leq \PR(\gamma_1 \leq \gamma_2) = \PR(\gamma_1 = \gamma_2) = 0.
\end{eq}
Moreover, 
\begin{eq}\label{O-1-2}
&\PR(\exists j> M: \mathcal{O}_1(\lambda) \leq \mathcal{O}_{j} (\lambda))\leq \PR \bigg( b_n^{-2}\sum_{i>M} \mathcal{O}_i(\lambda)^2 > b_n^{-2}\mathcal{O}_1(\lambda)^2\bigg) \\ 
& \leq \PR(b_n^{-1}\mathcal{O}_1(\lambda) \leq \varepsilon) + \PR \bigg(\sum_{i>M} \mathcal{O}_i(\lambda)^2 > \varepsilon b_n^2\bigg).
\end{eq}
The limsup as $n\to\infty$ of the first term is at most  $\PR(\gamma_1\leq\varepsilon) $. Since the distribution of $\gamma_1$ does not have an atom at zero, $\PR(\gamma_1\leq\varepsilon) $ can be taken to be at most $\eta$ by choosing $\varepsilon>0$ sufficiently small. 
Combining \eqref{O-1-1} and \eqref{O-1-2} yields that $\PR(\mathcal{O}_1(\lambda) \neq \mathcal{O}_{\sss (1)} (\lambda)) \leq 2\eta$ for all sufficiently large $n$.
Similarly, 
\begin{eq}
\PR(\mathcal{O}_2(\lambda) \neq \mathcal{O}_{\sss (2)} (\lambda) \text{ and } \mathcal{O}_1(\lambda) = \mathcal{O}_{\sss (1)} (\lambda)) \leq \PR(\exists j\geq 3: \mathcal{O}_2(\lambda) \leq \mathcal{O}_{j} (\lambda) ) + \PR(\mathcal{O}_2(\lambda) = \mathcal{O}_1(\lambda)).
\end{eq}
Both of the terms can be bounded using similar arguments as above. The proof for bounding $\PR(\mathcal{O}_i(\lambda) \neq \mathcal{O}_{\sss (i)} (\lambda))$ for $i\geq 2$ is similar.
\end{proof}

For an element $\mathbf{z} = (x_i,y_i)_{i\geq 1} \in \mathbb{U}^0_{\shortarrow}$ and a constant $c>0$, denote $c \mathbf{z} = (c x_i,y_i)_{i\geq 1}$.  Thus, Lemma~\ref{thm:open-comp} states that, for each fixed $\lambda$, $\mathbf{Z}^o_n(\lambda)$ is close to $\kappa \bZ_n(\lambda)$. The following result states that formally:
\begin{corollary}\label{cor-fixed-lambda}  For each fixed $\lambda$, as $n\to\infty$, $\mathrm{d}_{\sss\mathbb{U}}(\mathbf{Z}^o_n(\lambda),\kappa \bZ_n(\lambda))\pto 0$.
\end{corollary}
\begin{proof}
Using Lemma~\ref{thm:open-comp}, the proof follows using identical arguments as shown after \eqref{eq:bounds-comp-surp} in the proof of Proposition~\ref{prop:perc-dyn-close}. 
\end{proof}

\subsection{Comparison between the dynamic construction and the modified process}
The goal of this section is to prove that for large components, the modified construction does not change the open half-edge, component sizes and surplus edges considerably.
To this end, recall Algorithm~\ref{algo:modify-dyn-cons} and the associated notation. 
Thus $\bar{\mathscr{C}}_{\sss (i)} (\lambda)$ 
is the $i$-th largest component of $\bar{\mathcal{G}}_n(t_n(\lambda))$, and 
$\bar{\mathcal{O}}_i(\lambda)$ is the number of open half-edges in $\bar{\mathscr{C}}_{\sss (i)} (\lambda)$.
Note that we stop discarding the open half-edges in Algorithm~\ref{algo:modify-dyn-cons}, so that
$\bar{\mathcal{O}}_i(\lambda)$ counts the number of open half-edges associated to vertices at time $t_n(\lambda_\star)$.
Also, we write $\bar{\mathcal{O}}_{\sss (i)}(\lambda)$ and $\mathcal{O}_{\sss (i)}(\lambda)$ to denote the $i$-th largest number in the sequence  $(\bar{\mathcal{O}}_{j}(\lambda))_{j\geq 1}$ and $(\mathcal{O}_{j}(\lambda))_{j\geq 1}$ respectively.
\begin{proposition} \label{prop:comparison-comp-surplus-modified}
For each fixed $i\geq 1$, as $n\to\infty$,
\begin{enumerate}[(a)]
    \item $\bar{\mathcal{O}}_{\sss (i)} (\lambda)- \mathcal{O}_{i}(\lambda)= \oP(b_n)$, 
    \item $\mathrm{SP}(\bar{\mathscr{C}}_{\sss (i)} (\lambda))-\mathrm{SP}(\mathscr{C}_{\sss (i)} (\lambda)) \pto 0.$
\end{enumerate}
\end{proposition}

\subsubsection*{Comparison of component sizes.}

We start by showing that the component sizes and open half-edges in $\mathcal{G}_n(t_n(\lambda))$ and $\bar{\mathcal{G}}_n(t_n(\lambda))$ have identical distributional limit.

\begin{lemma}\label{lem:same-limit-Ois}For any $\lambda>\lambda_\star$, the sequences  $(b_n^{-1}\bar{\mathcal{O}}_{\sss (i)}(\lambda))_{i\geq 1}$ and $(b_n^{-1}|\bar{\mathscr{C}}_{\sss (i)}(\lambda)|)_{i\geq 1}$ converges in distribution with respect to the $\ell^2_{\shortarrow}$-topology, and the  distributional limits are identical to those of $(b_n^{-1}\mathcal{O}_{i}(\lambda))_{i\geq 1}$ and $(b_n^{-1}|\mathscr{C}_{\sss (i)}(\lambda)|)_{i\geq 1}$ respectively.
\end{lemma}
\begin{proof}
Let $\bld{\xi}(\bld{\theta},\lambda)$ denote the ordered vector of excursion lengths of the process \eqref{defn::limiting::process} by putting $\mu = 1$. 
Recall the discussion after Algorithm~\ref{algo:modify-dyn-cons} that, if $\beta_n = (\bar{s}_1(\nu_n-1)c_n)^{1/2} = cb_n(1+
\oP(1))$, then $((\beta_n^{-1} \bar{\mathcal{O}}_{\sss (i)}(\lambda))_{i\geq 1})_{\lambda>\lambda_\star}$ evolves exactly as a standard multiplicative coalescent.

Now, using Lemma~\ref{thm:open-comp}, together with Proposition~\ref{prop:coupling-whp} and the scaling limit from Theorem~\ref{thm:percolation}, we know that
\begin{eq}\label{eq:lim-comp-rescaled}
(\beta_n^{-1}|\mathcal{O}_{i}(\lambda)|)_{i\geq 1} \dto \bld{\xi}(c_1\bld{\theta},c_2\lambda),
\end{eq}with respect to the $\ell^2$-topology, for some constants $c_1,c_2>0$. 
Here we have used the fact from Lemma~\ref{thm:open-comp} that $\mathcal{O}_i(\lambda) = \mathcal{O}_{\sss (i)}(\lambda)$ with high probability for each fixed $i\geq 1$.
Also, to adjust a multiplicative $\sqrt{\nu}$ factor in the scaling limit in Theorem~\ref{thm:percolation}, we are using the fact that, for $\eta_1,\eta_2>0$, $\bld{\theta}\in \ell^3_{\shortarrow}\setminus \ell^2_{\shortarrow}$ and $\lambda\in\R$,
$ \bld{\xi}(\eta_1\bld{\theta},\eta_2\lambda)\stackrel{\sss d}{=} \frac{1}{\eta_1}\bld{\xi}\big(\bld{\theta},\frac{\eta_2}{\eta_1^2}\lambda\big).$
Thus, in particular, $(\beta_n^{-1}|\mathcal{O}_{i}(\lambda_\star)|)_{i\geq 1} \xrightarrow{\sss d} \bld{\xi}(c_1\bld{\theta},c_2\lambda_\star)$.
Using \cite[Theorem~2]{AL98}, there exists a version of the standard multiplicative coalescent $(\mathrm{MC}_1(\lambda))_{-\infty<\lambda<\infty}$ process such that $\mathrm{MC}_1(\lambda)$ has the same distribution as $\bld{\xi}(c_1\bld{\theta},c_2\lambda)$. 
Using the Feller property  of the multiplicative coalescent \cite[Proposition 5]{A97}, it follows that 
\begin{eq}\label{limit-o-bar-ordered}
(\beta_n^{-1}|\bar{\mathcal{O}}_{\sss (i)}(\lambda)|)_{i\geq 1} \dto \bld{\xi}(c_1\bld{\theta},c_2\lambda),
\end{eq}with respect to the $\ell^2_{\shortarrow}$-topology. 
For the component sizes, one can use \eqref{mul-coal-2d} and \eqref{open-he} to conclude the proof. 
\end{proof}
Next we show that with high probability the largest components of $\mathcal{G}_n(t_n(\lambda))$ contains those of $\bar{\mathcal{G}}_n(t_n(\lambda))$. 
\begin{lemma}\label{lem:sandwich-components-whp}
 For any $K\geq 1$, 
 \begin{eq} \label{eq:sandwich}
 \lim_{n\to\infty}\PR\big(\mathscr{C}_{\sss (i)}(\lambda) \subset \bar{\mathscr{C}}_{\sss (i)}(\lambda), \ \forall i\leq K\big) = 1.
 \end{eq}
\end{lemma}
\begin{proof}
Under the coupling in described in Remark~\ref{rem:modify-AMC}, and using Lemma~\ref{lem:same-limit-Ois}, the proof is identical to the proof of \eqref{eq:sandwich-perc-12}. We skip redoing the proof again.
\end{proof} 

\begin{proof}[Proof of Proposition~\ref{prop:comparison-comp-surplus-modified}~(a)]
Let $\mathcal{A}_K$ denote the event in \eqref{eq:sandwich}. On $\mathcal{A}_K$, $\mathcal{O}_1(\lambda) \leq \bar{\mathcal{O}}_1(\lambda) \leq \bar{\mathcal{O}}_{\sss (1)}(\lambda)$. 
Using the fact from Lemma~\ref{lem:same-limit-Ois} that $b_n^{-1}\bar{\mathcal{O}}_{\sss (1)}(\lambda)$ and $b_n^{-1}\mathcal{O}_1(\lambda)$ have the same scaling limit, it follows that  $\bar{\mathcal{O}}_{\sss (1)}(\lambda) - \mathcal{O}_1(\lambda) = \oP(b_n)$. 
Thus, the proof follows for $i=1$ by applying Lemma~\ref{lem:sandwich-components-whp}.
For $i\geq 2$, note that on $\mathcal{A}_K$, $\sum_{j\leq i}\mathcal{O}_j(\lambda) \leq \sum_{j\leq i}\bar{\mathcal{O}}_j(\lambda) \leq \sum_{j\leq i}\bar{\mathcal{O}}_{\sss (j)}(\lambda)$.
Using this relation the proof can be completed inductively as before by applying Lemmas~\ref{lem:same-limit-Ois}, and \ref{lem:sandwich-components-whp}.
\end{proof}

\subsubsection*{Comparison of surplus edges}
Next, we compare the surplus edges. 
Using Lemma~\ref{lem:sandwich-components-whp}, it follows that  $\mathrm{SP}(\mathscr{C}_{\sss (i)} (\lambda)) \leq \mathrm{SP}(\bar{\mathscr{C}}_{\sss (i)} (\lambda))$ with high probability for each fixed $i\geq 1$.
Let $B_{\sss \mathrm{SP}}(\lambda)$ denote the total number of
 surplus edges created during the formation of some bad edges. Then, with high probability, 
 \begin{eq} \label{eq:split-up-surplus}
 \mathrm{SP}(\bar{\mathscr{C}}_{\sss (i)} (\lambda)) \leq \mathrm{SP}(\mathscr{C}_{\sss (i)} (\lambda)) + B_{\sss \mathrm{SP}}(\lambda) + \sum_{j\neq i: \mathscr{C}_{\sss (j)}(\lambda) \subset \bar{\mathscr{C}}_{\sss (i)}(\lambda)} \mathrm{SP}(\mathscr{C}_{\sss (j)} (\lambda)).
 \end{eq}
 Thus, it is enough to show that the last two terms converge in probability to zero.
We first bound $B_{\sss \mathrm{SP}}(\lambda)$:
\begin{lemma} \label{lem:bad-estimation}
For any $\lambda\geq \lambda_\star$, $B_{\sss \mathrm{SP}}(\lambda)\pto 0$.
\end{lemma}
\begin{proof}
Before going into the proof, recall Algorithms~\ref{algo:dyn-cons-2}~and~\ref{algo:modify-dyn-cons}, and all the definitions therein. 
We write $C>0$ for a generic constant whose value can be different in different lines. 
Firstly, let $\mathcal{B}_i(\lambda)$ denote the number of open half-edges in $\bar{\mathscr{C}}_{\sss (i)}(\lambda)$ that caused the creation of at least one edge in  $[\lambda_\star,\lambda)$.
We refer to these as bad half-edges since pairing one of them again would give rise to a bad edge. 
Let $\bld{\mathcal{B}}(\lambda)$ be the corresponding vector obtained by ordering them as an element of $\ell^2_{\shortarrow}$. 
We will show that 
\begin{eq}\label{eq:bad-he-ell-2-zero}
\sup_{\lambda'\in [\lambda_\star,\lambda]} b_n^{-1}\|\bld{\mathcal{B}}(\lambda') \|_{2}=b_n^{-1}\|\bld{\mathcal{B}}(\lambda) \|_{2} \pto 0.
\end{eq}
For any fixed $K\geq 1$, define the event $\mathcal{A}_K$ that $\mathscr{C}_{(\sss i)}(\lambda)\subset \bar{\mathscr{C}}_{(\sss i)}(\lambda)$ for all $i\leq K$. 
Then $\PR(\mathcal{A}_K) \to 1$, by Lemma~\ref{lem:sandwich-components-whp}.
On the event $\mathcal{A}_K$, one can bound $\mathcal{B}_i(\lambda) \leq \bar{\mathcal{O}}_i(\lambda) - \mathcal{O}_i(\lambda)$ for all $i\leq K$ and $\mathcal{B}_i(\lambda) \leq \bar{\mathcal{O}}_i(\lambda)$ for all $i>K$. 
Further, on $\mathcal{A}_K$, $\sum_{j\leq i} \mathcal{O}_j(\lambda) \leq \sum_{j\leq i} \bar{\mathcal{O}}_j(\lambda)\leq \sum_{j\leq i} \bar{\mathcal{O}}_{\sss (j)}(\lambda)$ for all $i\leq K$.
Now, by Lemma~\ref{lem:same-limit-Ois}, $(b_n^{-1}\bar{\mathcal{O}}_{\sss (i)}(\lambda))_{i\geq 1}$ and $(b_n^{-1}\mathcal{O}_i(\lambda))_{i\geq 1}$ have the same distributional limit.
Therefore, $b_n^{-1}\mathcal{O}_j(\lambda)$ and $b_{n}^{-1}\bar{\mathcal{O}}_j(\lambda)$ have the same distributional limit for all $j\leq K$.
Moreover, since $(b_n^{-1}\bar{\mathcal{O}}_{\sss (i)}(\lambda))_{i\geq 1}$ converges in $\ell^2_{\shortarrow}$, we can choose $K$ sufficiently large so that 
$\sum_{i>K} \bar{\mathcal{O}}_i(\lambda)^2 <\varepsilon b_n^2$ with probability at most $\varepsilon$, for any $\varepsilon>0$.
Thus, \eqref{eq:bad-he-ell-2-zero} follows.

 For any semi-martingale $(Y_t)_{t\geq 0}$, we write $\mathrm{D}(Y_t)$ to denote the compensator.
Let us now turn to the asymptotics of $B_{\sss \mathrm{SP}}(\lambda)$. 
Note that $B_{\sss \mathrm{SP}}(\lambda)$ has jump size at most 1. Moreover, a surplus is bad if at least one of the end-points is bad. 
Therefore, 
\begin{eq}
\mathrm{D}(B_{\sss \mathrm{SP}}(\lambda)) \leq C\int_{\lambda_\star}^\lambda \frac{\sum_{i\geq 1} \bar{\mathcal{O}}_i (\lambda')\mathcal{B}_i(\lambda ' )}{\bar{s}_1c_n}\dif \lambda'.
\end{eq}
Now, the Cauchy-Schwarz inequality together with Theorem~\ref{thm:AMC-2D} and \eqref{eq:bad-he-ell-2-zero} shows that $\mathrm{D}(B_{\sss \mathrm{SP}}(\lambda))=\oP(1)$.  
Now, we use Lenglart inequality \cite{Leng77} which says that, for any non-negative process $(X_t)_{t\geq 0}$ with $X_0 = 0$, and $\varepsilon,\delta>0$,
\begin{eq}\label{eq:lenglart}
\PR\Big(\sup_{s\in [0,t]} X_s>\varepsilon\Big)\leq \frac{\E[\min\{\mathrm{D}(X_t),\delta\}]}{\varepsilon} + \PR(\mathrm{D}(X_t)>\varepsilon).
\end{eq}
Now, using \eqref{eq:lenglart}, together with  $\mathrm{D}(B_{\sss \mathrm{SP}}(\lambda))=\oP(1)$, implies the required statement for bad surplus edges. 
\end{proof}

Next, we bound the final term in \eqref{eq:split-up-surplus}. 
Suppose that, at time $\lambda_\star$, we have colored the components $(\mathscr{C}_{\sss (i)}(\lambda_\star))_{i\in [M]}$  blue, say, and then let Algorithms~\ref{algo:dyn-cons-2}~and~\ref{algo:modify-dyn-cons} evolve. 
Additionally, we color all the components blue that get connected to one of the blue components during the evolution. 
Let $\mathcal{C}_M(\lambda)$, $\bar{\mathcal{C}}_M(\lambda)$ denote the union of all such blue components in $\mathcal{G}_n(t_n(\lambda))$ and $\bar{\mathcal{G}}_n(t_n(\lambda))$. 
We track the surplus edges in  $\bar{\mathcal{C}}_M(\lambda)$ that were created when a bad edge caused some component with a surplus edge to merge with  a component in $\bar{\mathcal{C}}_M(\lambda)$. 
Let $F_M(\lambda)$ denote the number of times when such surplus edges were created upto time $\lambda$ (note that $F_M(\lambda)$ does not count the total number of these unwanted surplus edges).
We show that $F_M(\lambda)$ is zero with high probability. 
\begin{lemma}\label{lem:no-bad-large}
For any $\lambda\geq \lambda_\star$ and $M\geq 1$, $F_M(\lambda)\pto 0$.
\end{lemma}
\begin{proof}
The argument is similar to Lemma~\ref{lem:bad-estimation}.
Let $\mathcal{I}_M:= \{i:\bar{\mathscr{C}}_{\sss (i)}(\lambda)\subset \bar{\mathcal{C}}_M(\lambda)\}$. 
Then, the rate at which $F_M(\lambda)$ increases by $1$ is given by 
\begin{eq}
\mathrm{D}(F_M(\lambda)) &\leq C
\int_{\lambda_\star}^\lambda \bigg(\frac{ \sum_{i\in \mathcal{I}_M}\mathcal{B}_i(\lambda ' ) \sum_{j\notin \mathcal{I}_M: \mathrm{SP} (\bar{\mathscr{C}}_{\sss (j)}(\lambda '))\geq 1} \bar{\mathcal{O}}_j (\lambda')}{\bar{s}_1c_n}  \\ & \hspace{1cm}+  \frac{ \sum_{i\in \mathcal{I}_M}\bar{\mathcal{O}}_i (\lambda') \sum_{j\notin \mathcal{I}_M: \mathrm{SP} (\bar{\mathscr{C}}_{\sss (j)}(\lambda '))\geq 1} \mathcal{B}_j(\lambda ' )}{\bar{s}_1c_n}  \bigg) \dif  \lambda'.
\end{eq}
Let us write the two terms above by $(I)$ and $(II)$ respectively. 
Then, using the fact that $|\mathcal{I}_M| \leq M$,
\begin{eq}
(I) \leq \frac{Cb_n^2}{\bar{s}_1c_n}\int_{\lambda_\star}^\lambda b_n^{-2} \sum_{i\leq M} \mathcal{B}_i(\lambda') \times \sum_{j\geq 1}\bar{\mathcal{O}}_j (\lambda') \mathrm{SP} (\bar{\mathscr{C}}_{\sss (j)}(\lambda')) \dif \lambda' \to 0,
\end{eq}where the last step follows using \eqref{eq:bad-he-ell-2-zero}, $\bar{s}_1c_n = \OP(b_n^2)$, and the $\mathbb{U}^0_{\shortarrow}$ tightness of $(\bar{\mathbf{Z}}_n^o(\lambda'))_{n\geq 1}$. 
The last statement is a consequence of the near Feller property of the augmented multiplicative coalescent \cite[Theorem 3.1]{BBW12}.
Similarly, 
\begin{eq}
(II) \leq \frac{Cb_n^2}{\bar{s}_1c_n}\int_{\lambda_\star}^\lambda b_n^{-2} \sum_{i\leq M} \bar{\mathcal{O}}_i (\lambda') \times \sum_{j\geq 1}\mathcal{B}_i(\lambda')\mathrm{SP} (\bar{\mathscr{C}}_{\sss (j)}(\lambda')) \dif \lambda' \to 0,
\end{eq}
where we have used used that $b_n^{-1}\sum_{j\geq 1}\mathcal{B}_i(\lambda')\mathrm{SP} (\bar{\mathscr{C}}_{\sss (j)}(\lambda')) \pto 0$, which can be proved using identical arguments as \eqref{eq:bad-he-ell-2-zero} and the $\mathbb{U}^0_{\shortarrow}$ tightness of $(\bar{\mathbf{Z}}_n^o(\lambda'))_{n\geq 1}$. 
Now we conclude the proof using Lenglart's inequality using similarly as Lemma~\ref{lem:bad-estimation}.
\end{proof}

The following is the last ingredient that will be needed in the proof:
\begin{lemma} \label{lem:large-comp-contain} 
Fix any $\lambda >\lambda_\star$. For any $\varepsilon>0$, and $K\geq 1$, there exists $M = M(\varepsilon,K,\lambda)$ such that
\begin{equation}
\limsup_{n\to\infty}\prob{\bar{\mathscr{C}}_{\sss (1)}(\lambda), \dots,\bar{\mathscr{C}}_{\sss (K)}(\lambda) \text{ are not contained in }\bar{\mathcal{C}}_M(\lambda)}\leq \varepsilon.
\end{equation}
\end{lemma}
\begin{proof}
 Let $\mathcal{I}_M:= \{i:\bar{\mathscr{C}}_{\sss (i)}(\lambda)\subset \bar{\mathcal{C}}_M(\lambda)\}$. 
If $\bar{\mathscr{C}}_{\sss (i_0)}(\lambda)$ is not contained in $ \bar{\mathcal{C}}_M(\lambda)$ for some $i_0\leq K$, then that would imply that  $\sum_{i\notin\mathcal{I}_M}|\bar{\mathscr{C}}_{\sss (i)}(\lambda)|^2 \geq |\bar{\mathscr{C}}_{\sss (i_0)}(\lambda)|^2$ is not negligible. 
 Thus, it is enough to show that, for any $\varepsilon>0$, there exists $M$ such that 
 \begin{equation}\label{trunc-large-1}
  \limsup_{n\to\infty}\PR\bigg(\sum_{i\notin\mathcal{I}_M}|\bar{\mathscr{C}}_{\sss (i)}(\lambda)|^2>\varepsilon b_n^2 \bigg)\leq \varepsilon.
 \end{equation} 
For any $M\geq 1$, consider the merging dynamics of Algorithm~\ref{algo:modify-dyn-cons}, where at time $\lambda_\star$, all the components $(\bar{\mathscr{C}}_{\sss (i)}(\lambda_\star))_{i\in [M]}$ are removed. We refer to the above evolution as the $M$-truncated system.
We augment a previously defined notation with a superscript $>M$ to denote the corresponding quantity for the $M$-truncated system.
We assume that the $M$-truncated system and the modified system are coupled in a natural way such that at each event time of the modified truncated system, an edge is created in the $M$-truncated system if both the half-edges are selected from the outside of $\cup_{i=1}^M\bar{\mathscr{C}}_{\sss (i)}(\lambda_\star)$. 
Under this coupling,
\begin{equation}\label{trunc-large-2}
\sum_{i\notin\mathcal{I}_M}|\bar{\mathscr{C}}_{\sss (i)}(\lambda)|^2 \leq \sum_{i\geq 1} |\bar{\mathscr{C}}_{\sss (i)}^{\sss >M}(\lambda)|^2.
\end{equation}
Now, $(\bar{\mathbf{Z}}_n(\lambda_\star))_{n\geq 1}$ is tight in $\mathbb{U}^0_{\shortarrow}$.
Thus, the $\ell^2_{\shortarrow}$-norm of  $(b_n^{-1}|\bar{\mathscr{C}}_{\sss (i)}(\lambda_\star)|)_{i>M}$ can be made aritrarily small. 
Therefore, using the Feller property of the multiplicative coalescent process, the proof now follows.
\end{proof}

\begin{proof}[Proof of Proposition~\ref{prop:comparison-comp-surplus-modified} (b)]
Recall the expression \eqref{eq:split-up-surplus}. The second term goes to zero in probability using Lemma~\ref{lem:bad-estimation} and the third term goes to zero in probability using Lemmas~\ref{lem:no-bad-large}~and~\ref{lem:large-comp-contain}. Thus the proof follows. 
\end{proof}

\subsection{Proof of Theorem~\ref{thm:mul:conv}}
We now have all the ingredients to complete the proof of Theorem~\ref{thm:mul:conv}. 
Take $\lambda_\star = \lambda_1$. 

Recall that, for an element $\mathbf{z} = (x_i,y_i)_{i\geq 1} \in \mathbb{U}^0_{\shortarrow}$ and a constant $c>0$, we denote $c \mathbf{z} = (c x_i,y_i)_{i\geq 1}$. Recall the definition of $\mathbf{Z}_n^o(\lambda)$. Denote $\beta_n = (\bar{s}_1(\nu_n-1)c_n)^{1/2}$ and let  $\mathbf{Z}_n^{o,{\sss \mathrm{scl}}}(\lambda) = \beta_n^{-1}b_n\mathbf{Z}_n^o(\lambda)$. 
We define $\bar{\mathbf{Z}}_n^o(\lambda)$ and $\bar{\mathbf{Z}}_n^{o,{\sss \mathrm{scl}}}(\lambda)$ analogously for the modified graph $\bar{\mathcal{G}}_n(t_n(\lambda))$.
Thus, for any $\lambda\in \R$
\begin{eq}\label{one-dimensional-scaling-limit}
\bar{\mathbf{Z}}_n^{o,{\sss \mathrm{scl}}}(\lambda) \dto c\kappa\mathbf{Z}(\lambda),
\end{eq}where $ b_n \beta_n^{-1} \pto c$ (see Lemma~\ref{lem:total-open-he-heavy}), $\kappa$ is as in Corollary~\ref{cor-fixed-lambda}, and $\mathbf{Z}(\lambda)$ is given by Theorem~\ref{thm:percolation}. \eqref{one-dimensional-scaling-limit} follows by applying Theorem~\ref{thm:percolation}, Corollary~\ref{cor-fixed-lambda}, Proposition~\ref{prop:comparison-comp-surplus-modified}, together with Proposition~\ref{prop:coupling-whp} and \eqref{open-he-compare}. 

Next, we had argued after Remark~\ref{rem:modify-AMC} that $(\bar{\mathbf{Z}}_n^{o,{\sss \mathrm{scl}}}(\lambda))_{\lambda \in [\lambda_\star,\lambda^\star]}$ evolves as a standard augmented multiplicative coalescent.
Let $(\cT_{\lambda})_{\lambda\in \R}$ denote the semigroup associated to the standard augmented multiplicative coalescent process. 
Using the nearly Feller property for the augmented multiplicative coalescent \cite[Theorem~3.1]{BBW12}, it follows from \eqref{one-dimensional-scaling-limit} that 
\begin{equation}\label{eq:conv-modi-open}
(\bar{\mathbf{Z}}_n^{o,{\sss \mathrm{scl}}}(\lambda_1),\bar{\mathbf{Z}}_n^{o,{\sss \mathrm{scl}}}(\lambda_2))\dto (c\kappa\mathbf{Z}(\lambda_1),\cT_{\lambda_2-\lambda_1}(c\kappa\mathbf{Z}(\lambda_1))).
\end{equation}
Define $\cT'_{\lambda} = \cT_{\lambda/(c\kappa)^2}$.
Applying Proposition~\ref{prop:comparison-comp-surplus-modified}, it follows that 
\begin{equation}\label{eq:conv-modi-open-2}
(\mathbf{Z}_n^o(\lambda_1),\mathbf{Z}_n^o(\lambda_2))\dto (\kappa\mathbf{Z}(\lambda_1), \kappa\cT'_{\lambda_2-\lambda_1}(\mathbf{Z}(\lambda_1))).
\end{equation}
Now, Corollary~\ref{cor-fixed-lambda} yields 
that 
\begin{equation}\label{eq:conv-modi-open-3}
(\mathbf{Z}_n(\lambda_1),\mathbf{Z}_n(\lambda_2))\dto (\mathbf{Z}(\lambda_1), \cT'_{\lambda_2-\lambda_1}(\mathbf{Z}(\lambda_1))).
\end{equation}
Now, repeating the above argument inductively, there exists a version of the augmented multiplicative coalescent $\mathbf{AMC} = (\mathrm{AMC}(\lambda))_{\lambda\in\R}$, with the semigroup given by $(\cT'_\lambda)_{\lambda\in\R}$ such that for any $k \geq 1$
\begin{equation}
 (\mathbf{Z}_n(\lambda_1),\dots,\mathbf{Z}_n(\lambda_k))\dto (\mathrm{AMC}(\lambda_1),\dots,\mathrm{AMC}(\lambda_k)).
\end{equation}
Moreover, the scaling limit in Theorem~\ref{thm:spls}  shows that $\cT'_{\lambda-\lambda_1}(\mathbf{Z}(\lambda_1))$ and $\mathbf{Z}(\lambda)$ have identical distribution.
This ensures that there exists a version of the augmented multiplicative coalescent whose distribution at each fixed time $\lambda$ is identical to $\mathbf{Z}(\lambda)$.
Finally, the proof of Theorem~\ref{thm:spls} is completed by using Proposition~\ref{prop:perc-dyn-close}.  \qed

\appendix

\section{Path counting}\label{sec:appendix-gen-path-counting} In this section,  we derive a generalization of \cite[Lemma 5.1]{J09b} by extending the argument therein.  
Let $V_n'$ denote  the vertex chosen according to the distribution $G_n$  on $[n]$, independently of the graph. 
We will later take $G_n$ to be the uniform distribution on $[n]$, and the size-biased distribution with the sizes being proportional to the degrees.
Also, let $D_n'$ denote the degree of $V_n'$, $D_n$ denote the degree of a uniformly chosen vertex (independently of the graph) and  $\mathscr{C}(v)$ denote the connected component containing $v$.
\begin{lemma}\label{lem:gen-path-count}
 Let $\bld{w} = (w_i)_{i\in [n]}$ be a weight sequence and consider $\mathrm{CM}_n(\bld{d})$ such that $\nu_n<1$. Then,
 \begin{equation}
  \E\bigg[\sum_{i\in \mathscr{C}(V_n')}w_i\bigg]\leq \E\big[w_{V_n'}\big]+\frac{ \expt{D_n'}\E\big[ D_nw_{V_n}\big]}{\expt{D_n}(1-\nu_n)}.
 \end{equation}
\end{lemma}
\begin{proof} Consider all possible paths of length $l$ starting from $V_n'$ and the $w$-value at the end of those paths. 
If we sum over all such paths together with a sum over all possible $l$, then we obtain an upper bound on $\sum_{i\in \mathscr{C}(V_n')}w_i$. Write $\E_v[\cdot]$ for the expectation conditional on $V_n'=v$.
 Thus,
 \begin{equation}
  \begin{split}
   \E_v\bigg[\sum_{i\in \mathscr{C}(V_n')}w_i\bigg]\leq w_v+d_v \sum_{l\geq 1}\sum_{\substack{x_1,\cdots,x_l\\x_i\neq x_j, \forall i\neq j}} \frac{\prod_{i=1}^{l-1}d_{x_i}(d_{x_i}-1)d_{x_l}w_{x_l}}{(\ell_n-1)\dots(\ell_n-2l+1)}.
  \end{split}
 \end{equation} 
 Now, using the exactly same arguments as \cite[Lemma 5.1]{J09b}, it follows that
 \begin{equation}
  \E\bigg[\sum_{i\in \mathscr{C}(V_n')}w_i\bigg]\leq \E\big[w_{V_n'}\big] + \frac{ \expt{D_n'}\expt{D_nw_{V_n}}}{\expt{D_n}}\sum_{l\geq 1}\nu_n^{l-1},
 \end{equation}and this completes the proof.
\end{proof}
\section{Proof of Lemma~\ref{lem:sp-cv-n}}
\label{sec_appendix}
\begin{proof}[\nopunct]
The proof is an adaptation of the proof of \cite[Lemma 20]{DHLS15}. 
Let $V_n'$ denote the vertex chosen according to the distribution $G_n$  on $[n]$, independently of the graph and let $D_n'$ denote the degree of $V_n'$. 
Suppose that $\limsup_{n\to\infty}\E[D_n']<\infty$.
We use a generic constant $C$ to denote a positive constant independent of $n,\delta,K$. 
Consider the graph exploration described in Algorithm~\ref{algo-expl}, but now we start by choosing vertex $V_n'$ at Stage 0 and declaring all its half-edges active. 
The exploration process is still given by \eqref{defn:exploration:process} with $S_n(0)=D_n'$.
Note that $\mathscr{C}(V_n')$ is explored when $\mathbf{S}_n$ hits zero. For $H>0$,  let \begin{equation} \label{defn:gamma}
\gamma := \inf \{ l\geq 1: S_n(l)\geq H \text{ or }  S_n(l)= 0 \}\wedge (2\delta_K b_n).
\end{equation} Note that
\begin{equation}\label{exploration:super_martingale}
\begin{split}
 \expt{S_n(l+1)-S_n(l)\vert \left( \mathcal{I}_i^n(l)\right)_{i=1}^n}&= \sum_{i\in [n]}d_i\prob{i\notin \mathscr{V}_l, i\in \mathscr{V}_{l+1}\vert \left( \mathcal{I}_i^n(l)\right)_{i=1}^n} -2\\
 &= \frac{ \sum_{i\notin \mathscr{V}_l}d_i^2}{\ell_n-2l-1}-2\leq \frac{ \sum_{i\in [n]}d_i^2}{\ell_n-2l-1}-2\\:
 & =\lambda c_n^{-1}+o(c_n^{-1})+\frac{2l+1}{\ell_n-2l-1}\times \frac{\sum_{i\in [n]}d_i^2}{\ell_n}   \leq 0
\end{split}
\end{equation} uniformly over $l\leq 2\delta_K b_n$ for all small $\delta >0$ and large $n$, where the last step follows from the fact that $\lambda<0$. Therefore, $\{S_n(l)\}_{l= 1}^{2\delta_Kb_n}$ is a super-martingale. The optional stopping theorem now implies
  \begin{equation}
   \mathbbm{E}\left[D_n'\right] \geq \mathbbm{E}\left[S_n(\gamma)\right] \geq H \mathbbm{P}\left( S_n(\gamma) \geq H \right).
  \end{equation} Thus,
  \begin{equation} \label{eqn::bound_geq_H_at_stopping_time}
    \mathbbm{P}\left( S_n(\gamma) \geq H \right) \leq \frac{\expt{D_n'}}{H}.
  \end{equation}
Put $H=a_nK^{1.1}/\sqrt{\delta}$. To simplify the writing, we  write $S_n[0,t]\in A$ to denote that $S_n(l)\in A,$ for all $ l\in [0,t]$.  Notice that
 \begin{equation}\label{surp:sup:less}\begin{split}
  &\prob{\surp{\mathscr{C}(V_n')}\geq K,|\mathscr{C}(V_n')|\in (\delta_K b_n,2\delta_Kb_n)}\\
  &\leq \prob{S_n(\gamma)\geq H}+\prob{\surp{\mathscr{C}(V_n')}\geq K, S_n[0,2\delta_K b_n]< H, S_n[0,\delta_K b_n]>0}.
  \end{split}
 \end{equation}Now,
 \begin{equation}
  \begin{split}
   &\prob{\surp{\mathscr{C}(V_n')}\geq K, S_n[0,2\delta_K b_n]< H, S_n[0,\delta_K b_n]>0}\\
  &\leq \sum_{1\leq l_1<\dots< l_K\leq 2\delta_K b_n} \prob{\text{surpluses occur at times } l_1,\dots,l_K,  S_n[0,2\delta_K b_n]< H, S_n[0,\delta_K b_n]>0}\\
  &=\sum_{1\leq l_1<\dots<l_K\leq 2\delta_K b_n}\expt{\ind{0<S_n[0,l_K-1]<H, \mathrm{SP}(l_K-1)=K-1}Y},
  \end{split}
 \end{equation}
 where
 \begin{equation}
 \begin{split}
  Y&=\prob{K^{th}\text{ surplus occurs at time }l_K,  S_n[l_K,2\delta_Kb_n]< H, S_n[l_K,\gamma]>0\mid \mathscr{F}_{l_K-1} }\\
  &\leq \frac{CK^{1.1}a_n}{\ell_n\sqrt{\delta}}\leq \frac{CK^{1.1}}{b_n\sqrt{\delta}}.
 \end{split}
 \end{equation}Therefore, using induction, \eqref{surp:sup:less} yields
 \begin{equation}\label{exploration:bounded:surplus}
 \begin{split}
  &\prob{\surp{\mathscr{C}(V_n')}\geq K, S_n[0,2\delta_K b_n]< H, S_n[0,\delta_K b_n]>0}\\
  &\leq C\bigg( \frac{K^{1.1}}{\sqrt{\delta}b_n}\bigg)^K\frac{(2\delta b_n)^{K-1}}{K^{0.12(K-1)}(K-1)!}\sum_{l_1=1}^{2\delta_K b_n}\prob{|\mathscr{C}(V_n')|\geq l_1}\leq C \frac{\delta^{K/2}}{K^{1.1}b_n}  \expt{|\mathscr{C}(V_n')|},
  \end{split}
 \end{equation}where we have used the fact that $\#\{1\leq l_2,\dots,l_K\leq 2\delta b_n\}=(2\delta b_n)^{K-1}/(K-1)!$ and Stirling's approximation for $(K-1)!$, as well as $2\e\sqrt{\delta}<1$ in the last step. Since $\lambda <0$, we can use Lemma~\ref{lem:gen-path-count} to conclude that for all sufficiently large $n$
 \begin{equation} \label{expectation:random:vert:comp}
  \expt{|\mathscr{C}(V_n')|}\leq Cc_n,
 \end{equation} for some constant $C>0$ and we get the desired bound for \eqref{surp:sup:less}.
  The proof of Lemma~\ref{lem:sp-cv-n} is now complete.
\end{proof}

\section{Non-atomic distribution for the excursion process}\label{sec:appendix-density-existence}
Below we prove the scaling limit of the exploration process \eqref{defn:exploration:process} has a density at each fixed time. This was proved in \cite[Lemma 3.5]{BHL12} for $\theta_i = i^{-\alpha}$. However, the argument does not generalize directly for general $\boldsymbol{\theta}\in \ell^3_{\shortarrow}\setminus \ell^2_{\shortarrow}$.
Below we give a proof for the general case $\boldsymbol{\theta}$:
\begin{lemma}\label{lem:no-atom}
Let $S(t) =  \sum_{i=1}^{\infty} \theta_i(\mathcal{I}_i(t)- \theta_it)$, where $\xi_i\sim \mathrm{Exp}(\theta_i)$ independently, and $\boldsymbol{\theta}=(\theta_1,\theta_2,\dots)\in \ell^3_{\shortarrow}\setminus \ell^2_{\shortarrow}$. 
Then the distribution of $S_t$ has no atoms for all $t>0$.
\end{lemma}

\begin{proof}
Let $\phi_t (v) = \E[\e^{\im v S(t)}]$ for $v\in\R$. 
Then, 
\begin{eq}
\phi_t(v) = \prod_{j=1}^{\infty} \e^{- \im v \theta_j^2t} (1+ (\e^{\im v \theta_j} - 1)(1-\e^{- \theta_jt})).
\end{eq}
Therefore, 
\begin{eq}
|\phi_t (v)|^2 &\leq \prod_{j=1}^\infty \big|1+ (\cos (v\theta_j ) - 1) (1-\e^{-t\theta_j })+ \im (1-\e^{-t\theta_j})\sin(v\theta_j)\big|^2\\
&= \prod_{j=1}^\infty \Big(\big(\cos(v\theta_j)+ \e^{-t\theta_j}(1-\cos(v\theta_j))\big)^2+ (1-\e^{-t\theta_j})^2 \sin^2(v\theta_j)\Big) \\
& = \prod_{j=1}^\infty \Big(\e^{-2t\theta_j} (1-2\cos(v\theta_j))+ 2 \e^{-t\theta_j} \cos(v\theta_j) + (1-\e^{-t\theta_j})^2 \Big) \\ 
& = \prod_{j=1}^\infty \Big(1- 2 \e^{-t\theta_j}(1- \e^{-t\theta_j}) (1-\cos(v\theta_j))\Big)\\
&\leq \e^{-\sum_{j=1}^\infty 2 \e^{-t\theta_j}(1- \e^{-t\theta_j}) (1-\cos(v\theta_j))},
\end{eq}where in the last step we have used the fact that $1-x\leq \e^{-x}$ for all $x>0$.
Let $j_0(v,t)\geq 1$ be such that $ \max \{|v|\theta_j,t\theta_j \} \leq 1$ for all $j\geq j_0 (v,t)$.
Now, for $j\geq j_0(v,t)$, we have that $\e^{-t\theta_j} \geq \e^{-1}$, $(1-\e^{-t\theta_j}) \geq t\theta_j/2$ and $1-\cos(v\theta_j) \geq \frac{2}{\pi} v^2\theta_j^2$.
Thus, 
\begin{eq}\label{eq:ub-char-funct}
|\phi_t (v)|\leq \e^{\frac{2t}{\e\pi} v^2\sum_{j\geq j_0(v,t)} \theta_j^3}.
\end{eq}
Let $g(|v|):= \sum_{j\geq j_v} \theta_j^3$. Since $t$ is fixed, we ignored $t$ in the notation here. We need the following facts about $g$:
\begin{fact} \label{fact:int-g-infty}
The function $g:[0,\infty)\mapsto (0.\infty)$ is non-increasing and left continuous. Further, $\int_0^\infty g(v)\dif v = \infty$.
\end{fact}
\begin{claimproof}
The non-increasing and left continuity follow easily from the definition. 
Let $C_0 = \sum_{j=1}^\infty \theta_j^3$ and $f(x) = C_0^{-1/3} \theta_j$ for $j-1\leq x<j$ and $f(x) =0$ for $x<0$. Let $X$ be a random variable with density function $f^3$. Then 
\begin{eq}
\int_0^\infty \PR\bigg(\frac{1}{f(X) C_0^{1/3}}\geq v\bigg) \dif v=\E \bigg[\frac{1}{f(X) C_0^{1/3}}\bigg] = C_0^{-1/3}\int_0^\infty f^2(x) \dif x = C_0^{-1} \sum_{j=1}^\infty \theta_j^2 = \infty.
\end{eq}
Moreover, 
\begin{eq}
 \PR\bigg(\frac{1}{f(X) C_0^{1/3}}\geq v\bigg)  = \sum_{j=1}^\infty C_0^{-1} \theta_j^3 \mathbbm{1}_{\{\frac{1}{\theta_j}\geq v\}} = C_0^{-1} g(v),
\end{eq}for $v\geq t$. Thus, $\int_0^\infty g(v)\dif v = \infty$.
\end{claimproof}
\begin{fact}\label{fact:second-fact}
For every $M>0$, there exists $T_n = T_n(M) \nearrow \infty$ such that 
\begin{eq}
\lim_{n\to\infty}\frac{1}{T_n} \Lambda \big(\{ 0\leq u\leq T_n:u^2g(u) \leq M\} \big) = 0,
\end{eq}where $\Lambda$ denotes the Lebesgue measure.
\end{fact}
\begin{claimproof}
Assume the contrary, i.e., there exists $\varepsilon>0$ and  $T_0>0$ such that $\forall T\geq T_0$,
$\frac{1}{T}  \Lambda \big(\{ 0\leq u\leq T:u^2g(u) \leq M\} \big) \geq \varepsilon.$ Thus, for any $T\geq T_0$,
\begin{eq}
u_T:= \sup\{ 0\leq u\leq T:u^2g(u) \leq M\} \geq \varepsilon T.
\end{eq}
Since $g$ is left continuous, $u_T^2g(u_T) \leq M$ and thus 
\begin{eq}
T^2 g(T) \leq T^2 g(u_T) \leq \frac{T^2 M}{u_T^2} \leq \frac{M}{\varepsilon^2},
\end{eq}where in the first step, we have used that $g$ is non-decreasing from Fact~\ref{fact:int-g-infty}. Thus $g(T) \leq M/(\varepsilon T)^2$ for all $T\geq T_0$ and therefore  
$\int_0^\infty g(v)\dif v <\infty$, which contradicts Fact~\ref{fact:int-g-infty}.
\end{claimproof}
Continuing to the proof of Lemma~\ref{lem:no-atom}, note that for any $a\in\R$ and $M>0$, and $T_n$ chosen according to Fact~\ref{fact:second-fact}, 
\begin{eq}
\PR(S(t) = a) &= \lim_{T\to\infty} \frac{1}{2T} \int_{-T}^T \e^{-\im v a} \phi_t(v) \dif v = \lim_{n\to\infty} \frac{1}{2T_n} \int_{-T_n}^{T_n} \e^{-\im v a} \phi_t(v) \dif v\\
& \leq \limsup_{n\to\infty}\frac{1}{2T_n} \int_{-T_n}^{T_n} \e^{- \frac{2t}{\e\pi}v^2 g(|v|)} \dif v=\limsup_{n\to\infty}\frac{1}{T_n} \int_{0}^{T_n} \e^{- \frac{2t}{\e\pi}v^2 g(v)} \dif v \\
& \leq \limsup_{n\to\infty} \frac{1}{T_n} \Lambda \big(\{ 0\leq u\leq T_n:u^2g(u) \leq M\} \big) + \e^{-\frac{2t}{\e\pi}M} = \e^{-\frac{2t}{\e\pi}M}
\end{eq}where in the third step we have used \eqref{eq:ub-char-funct}, and the last-but-one step follows from Fact~\ref{fact:second-fact}. Since $M>0$ is arbitrary, $\PR(S(t) = a) =0$ and the proof follows.
\end{proof}

\section{Evolution of moments in the dynamic construction} \label{sec:susceptibility}
Let $w_i(t)$ denote the number of unpaired/open half-edges incident to  vertex~$i$ at time $t$ in Algorithm~\ref{algo:dyn-cons-2}. 
Let $s_1(t)$ denote the total number of unpaired half-edges at time $t$. Denote also $s_2(t)=\sum_{i\in [n]} w_i(t)^2$, $s_{d,w}(t)=\sum_{i\in [n]}d_iw_i(t)$. 
Further, we write $\mu_n=\ell_n/n$.
\begin{lemma}\label{lem:total-open-he-heavy}  Under {\rm Assumption~\ref{assumption1}}, the quantities 
 $\sup_{t\leq T}|\frac{1}{n}s_1(t)-\mu_n\e^{-2t}|$, $\sup_{t\leq T}|\frac{1}{n} s_2(t) - \mu_n\e^{-4t}(\nu_n+\e^{2t})|,$ $\sup_{t\leq T}|\frac{1}{n}s_{d,w}(t) - \mu_n(1+\nu_n)\e^{-2t}|$  are all  $\OP(n^{-1/2})$, for any $T>0$.
\end{lemma}
\begin{proof}
 The proof uses the differential equation method \cite{wormald1995differential}. 
 Notice that, after each ring of an exponential clock in Algorithm~\ref{algo:dyn-cons-2}, $s_1(t)$ decreases by two. Let $Y$ denote a unit rate Poisson process. Using the random time change representation \cite{EK86},
 \begin{equation}\label{eq:diff-eqn}
  s_1(t) = \ell_n - 2 Y\bigg(\int_{0}^t s_1(u)\mathrm{d} u\bigg) = \ell_n +M_n(t)-2\int_{0}^t s_1(u)\mathrm{d} u,
 \end{equation}where $\bld{M}_n$ is a martingale. 
 Now, the quadratic variation of $\bld{M}_n$ satisfies $ \langle M_n \rangle (t)\leq 4t\ell_n = O(n),$ which implies that $\sup_{t\leq T}|M_n(t)|= \OP(\sqrt{n}).$ 
 Moreover, notice that the function $f(t)=\mu_n\e^{-2t}$ satisfies $f(t)=\mu_n-2\int_0^tf(u)\mathrm{d}u$. 
  Therefore,
 \begin{equation}
 \begin{split}
  \sup_{t\leq T}\bigg|\frac{1}{n}s_1(t)- \mu_n\e^{-2t}\bigg| &\leq \sup_{t\leq T}\frac{|M_n(t)|}{n}+2\int_0^T\sup_{t\leq u}\bigg|\frac{1}{n}s_1(t)- \mu_n\e^{-2t}\bigg| \mathrm{d}u.
 \end{split}
 \end{equation} Using Gr\H{o}nwall's inequality \cite[Proposition 1.4]{M86}, it follows that
 \begin{equation}\label{eq:diff-eqn-gronwall}
   \sup_{t\leq T}\bigg|\frac{1}{n}s_1(t)- \mu_n\e^{-2t}\bigg|\leq \e^{2T} \sup_{t\leq T}\frac{|M_n(t)|}{n} =\OP(n^{-1/2}),
 \end{equation}as required. 
For $s_2(t)$, note that if half-edges corresponding to vertices $i$ and $j$ are paired, $s_2$ changes by $-2w_i-2w_j+2$ and if two half-edges corresponding to $i$ are paired, $s_2$ changes by $-4w_i+4$. Thus, 
\begin{equation}
\begin{split}
\sum_{i\in [n]}w_i(t)^2&=\sum_{i\in [n]}d_i^2+M_n'(t)+\int_0^t \sum_{i\neq j}\frac{w_i(u)w_j(u)(-2w_i(u)-2w_j(u)+2)}{s_1(u)-1} \dif u\\
 &\hspace{4cm}+\int_0^t\sum_{i\in [n]}\frac{w_i(u)(w_i(u)-1)(-4w_i(u)+4)}{s_1(u)-1}\dif u\\
 & = n\mu_n(1+\nu_n)+M_n'(t)+\int_0^t(-4s_2(u)+2s_1(u))\mathrm{d}u+O(1),
 \end{split}
\end{equation}where $\bld{M}_n'$ is a martingale with quadratic variation given by $\langle M_n'\rangle (t) =O(n)$. 
Again, an estimate equivalent to \eqref{eq:diff-eqn-gronwall} follows using Gr\H{o}nwall's inequality.
Notice also that when a clock corresponding to vertex $i$ rings and it is paired to vertex $j$, then $s_{d,w}$ decreases by $d_i+d_j$. 
Thus,
 \begin{equation}
  \begin{split}
   s_{d,w}(t)&=\sum_{i\in [n]}d_i^2+M_n''(t)-\int_0^t \sum_{i\neq j}\frac{w_i(u)w_j(u)(d_i+d_j)}{s_1(u)-1} \mathrm{d}u-\int_0^t \sum_{i\in [n]}\frac{w_i(u)(w_i(u)-1)2d_i}{s_1(u)-1}\mathrm{d}u\\
   & = n\mu_n(1+\nu_n)+M_n''(t)- 2 \int_0^ts_{d,w}(u)\mathrm{d}u,
  \end{split}
 \end{equation}where $\bld{M}_n''$ is a martingale with quadratic variation given by $\langle M_n''\rangle (t) \leq 2t \sum_{i\in [n]}d_i^2=O(n)$. 
 We can now apply Gr\H{o}nwall's inequality as before. 
 The proof of Lemma~\ref{lem:total-open-he-heavy} is now complete. 
\end{proof}

\section*{Acknowledgement} 
We sincerely thank Shankar Bhamidi for helpful discussions.
This research have been supported by the Netherlands Organisation for Scientific
Research (NWO) through Gravitation Networks grant 024.002.003. In addition, RvdH has been supported by VICI grant 639.033.806, JvL has been supported by the European Research Council (ERC), and SS has been supported by EPSRC grant EP/J019496/1, a CRM-ISM fellowship, and a UGC CAS-II grant, Grant number F.510/25/CAS-II/2018(SAP-I). 
We thank Lorenzo Federico and Tim Hulshof for pointing out an error in the proof of Lemma~\ref{lem:surp:poisson-conv} in the previous arXiv version of this paper. 
\bibliographystyle{apa}
\bibliography{project2}

\end{document}